\documentclass{article}

\usepackage[letterpaper,margin=1in]{geometry}
\usepackage{times}
\usepackage{natbib}
\setcitestyle{authoryear,round,citesep={;},aysep={,},yysep={;}}
\usepackage{xcolor}

\usepackage{amsmath,amsfonts,bm}

\def\eqref#1{equation~\ref{#1}}

\def\1{\bm{1}}

\def\rb{{\textnormal{b}}}

\def\rvw{{\mathbf{w}}}

\def\vzero{{\bm{0}}}

\def\vtheta{{\bm{\theta}}}

\def\vb{{\bm{b}}}

\def\ve{{\bm{e}}}

\def\vg{{\bm{g}}}

\def\vr{{\bm{r}}}

\def\vu{{\bm{u}}}
\def\vv{{\bm{v}}}
\def\vw{{\bm{w}}}
\def\vx{{\bm{x}}}

\def\vz{{\bm{z}}}

\def\evw{{w}}
\def\evx{{x}}

\def\mA{{\bm{A}}}

\def\mC{{\bm{C}}}
\def\mD{{\bm{D}}}

\def\mH{{\bm{H}}}
\def\mI{{\bm{I}}}
\def\mJ{{\bm{J}}}

\def\mM{{\bm{M}}}

\def\mP{{\bm{P}}}

\def\mS{{\bm{S}}}

\def\mV{{\bm{V}}}

\DeclareMathAlphabet{\mathsfit}{\encodingdefault}{\sfdefault}{m}{sl}
\SetMathAlphabet{\mathsfit}{bold}{\encodingdefault}{\sfdefault}{bx}{n}

\def\sA{{\mathbb{A}}}

\def\sI{{\mathbb{I}}}

\def\sK{{\mathbb{K}}}

\def\sP{{\mathbb{P}}}

\def\sS{{\mathbb{S}}}

\newcommand{\R}{\mathbb{R}}

\newcommand{\normltwo}{L^2}

\usepackage[utf8]{inputenc}
\usepackage{amsmath, amssymb, amsfonts}
\usepackage{amsthm}
\usepackage{graphicx}
\usepackage{float}
\usepackage{subfigure}
\usepackage{algorithm}
\usepackage{algorithmic}
\usepackage{booktabs}
\usepackage{tabularx}
\usepackage{capt-of}
\usepackage{captcont}
\usepackage{hyperref}
\usepackage{url}

\newtheorem{theorem}{\bf Theorem}
\newtheorem{remark}{Remark}
\newtheorem{assumption}{Assumption}
\newtheorem{lemma}{Lemma}

\title{Physics-Informed Stochastic Configuration Machine: A Backpropagation-Free Neural Network with Fast Training involving Nonlinear Differential Equations}
\author{
	Yuehao~Song\textsuperscript{1}, Zhong~Chen\textsuperscript{1},
	Lihui~Cen\textsuperscript{1,*}, Liang~Wu\textsuperscript{2},
	and~Kai~Zhang\textsuperscript{3}\\[0.25em]
	{\normalfont\small
		\begin{tabular}{@{}ll@{}}
			\textsuperscript{1}School of Automation, Central South University
			& \quad\textsuperscript{2}Johns Hopkins University\\
			\multicolumn{2}{@{}l}{\textsuperscript{3}State Key Laboratory of Simulation and Regulation of Water Cycle in River Basin,}\\
			\multicolumn{2}{@{}l}{\phantom{\textsuperscript{3}}China Institute of Water Resources and Hydropower Research}
		\end{tabular}
	}
} 
\date{}

\begin{document}
	\maketitle
	\begingroup
		\renewcommand{\thefootnote}{\fnsymbol{footnote}}
		\footnotetext[1]{Corresponding author: Lihui Cen (e-mail: \mbox{lhcen@csu.edu.cn}).}
	\endgroup
	
	\begin{abstract}
		While Physics-Informed Neural Networks (PINNs) have emerged as a transformative paradigm for solving complex differential equations, their reliance on backpropagation-based gradient descent and automatic differentiation (AD) imposes significant computational bottlenecks and severe non-convex optimization challenges. To overcome these fundamental limitations, we propose the \emph{Physics-Informed Stochastic Configuration Machine (PI-SCM)}, a novel backpropagation-free framework for both forward and inverse problems in differential equations. The core mathematical contribution lies in the analytical evaluation of local Jacobians for nonlinear differential operators, which facilitates a linearized representation of the physical loss and projects it into a unified, linearized algebraic subspace. This reformulation allows for the explicit determination of network weights via a sequence of linearized least-squares problems, effectively bypassing the iterative traps of traditional nonlinear optimizers. We develop a progressive algorithmic suite comprising localized construction (PI-SC-I), sliding-window correction (PI-SC-II), and global correction (PI-SC-III), and establish their universal approximation properties. Extensive experiments demonstrate that PI-SCM achieves robust parameter identification and substantially higher predictive accuracy than the PIELM baselines while accelerating training by orders of magnitude compared with standard PINNs. Our work provides a highly efficient and scalable foundation for next-generation, real-time Scientific Machine Learning applications.
	\end{abstract}

	\section{Introduction}

Nonlinear ordinary and partial differential equations (ODEs/PDEs) describe physical systems across science and engineering \citep{farlow1993partial}. Alongside classical numerical solvers, scientific machine learning integrates physical knowledge into learned approximations \citep{baker2019workshop}. Physics-Informed Neural Networks (PINNs) \citep{raissi2019physics} embed governing equations and initial/boundary conditions into the training loss, providing a mesh-free approach to solution approximation and parameter identification, including settings with sparse observations.

However, training PINNs using gradient-based optimizers, such as Adam and L-BFGS \citep{wang2023expert}, presents two main challenges. First, nonlinear physical residuals produce nonconvex objectives, and imbalanced gradients across loss terms can hinder convergence \citep{rathore2024challenges}. Second, evaluating derivatives through automatic differentiation (AD) (e.g., TensorFlow \citep{abadi2015tensorflow}, Pytorch \citep{paszke2017automatic}) and repeatedly backpropagating parameter gradients can incur substantial computational and memory costs \citep{cuomo2022scientific}. Together, these challenges motivate methods that incorporate physical constraints without backpropagation.

Randomized networks offer an alternative by fixing randomly configured hidden parameters and solving for output weights through linear least squares. Stochastic Configuration Networks (SCNs) \citep{wang2017stochastic} combine this analytical training mechanism with incremental node construction under a supervisory inequality, providing a universal approximation property under the stated conditions. This combination makes SCNs a natural starting point for efficient physics-informed learning. The key difficulty is that fixing hidden parameters does not make a nonlinear differential operator linear in the output weights. Physics-Informed Extreme Learning Machine (PIELM) \citep{dwivedi2020physics} obtains output weights through linear least squares for linear PDEs but cannot retain this formulation for nonlinear PDEs. Local Extreme Learning Machine (LocELM) \citep{dong2021local} and the neural network basis method (NNBM) \citep{huang2025neural} extend randomized networks to nonlinear PDEs without backpropagation but still require iterative nonlinear optimization, such as the Gauss--Newton iterations. Frozen-PINNs \citep{datar2026fast} combine randomized network features with least-squares initialization and adaptive ODE solvers for time-dependent PDEs, but this formulation cannot apply to nonlinear static PDEs. Physics-Informed Stochastic Configuration Network (PISCN) \citep{xu2025physics} relies on gradient descent for its optimization. Despite their different training strategies, these randomized network methods share a common limitation: physics-informed constraints do not guide their hidden parameter selection or node construction. This gap motivates a unified framework that incorporates physical constraints into both hidden node construction and output weight estimation while retaining analytical training for nonlinear differential equations. Further background and related work are provided in Appendix~\ref{app:extended-related-work}.

We address this challenge with the \textbf{Physics-Informed Stochastic Configuration Machine (PI-SCM)}. The physical residual is locally linearized, allowing the physics-informed and observational constraints to be assembled into a unified linear least-squares system. This formulation supports both unsupervised and semi-supervised forward learning and joint state and parameter estimation for inverse problems. Our contributions are threefold:

1. \textbf{Physics-informed node construction and analytical training.} We integrate physical, boundary, and data constraints into a unified residual that guides both node construction and output weight estimation. For each candidate node, the output weights are obtained by solving a single linear least-squares problem without AD, backpropagation, or an inner sequence of Gauss--Newton iterations.

2. \textbf{Progressive construction and correction.} PI-SC-I performs localized construction, PI-SC-II applies sliding-window correction, and PI-SC-III applies global correction, offering different trade-offs between accuracy and computational cost. Under the stated assumptions, the framework ensures monotonic residual reduction and satisfies the universal approximation property.

3. \textbf{Unified forward and inverse learning.} Jointly estimating physical parameters and network coefficients enables state reconstruction and parameter identification within the same formulation. Experiments on eight benchmark systems show that PI-SCM can achieve competitive or better predictive accuracy than deep PINN baselines while reducing training time by approximately $36$--$107\times$ across the evaluated forward and inverse tasks. 

\section{Preliminaries of PI-SCM and Problem Statement}\label{sec:preliminaries}
	
	\subsection{Network Architecture of PI-SCM}\label{sec:network_architecture}
	
	PI-SCM incrementally constructs a single hidden layer network. Given an input coordinate $\vx=(\evx_1,\evx_2,\dots,\evx_n)^\top$, the $k$-th hidden node output is $h_k(\vx)=g(\vw_k^\top\vx+b_k)$, where $g$ is the activation function, $\vw_k\in\R^n$ is the hidden weight vector, and $b_k\in\R$ is the hidden bias. With $L$ hidden nodes and output weights $\boldsymbol{\beta}_k\in\R^m$, the state approximation is $\vu_L(\vx)=\sum_{k=1}^{L}\boldsymbol{\beta}_k h_k(\vx)$. The hidden parameters $\vw_k$ and $b_k$ are randomly assigned within a problem-dependent adaptive range $[-\tau,\tau]$ and remain fixed, while the output weights are determined analytically using pseudoinverse solvers. This architecture serves as the foundation for solving the problems defined below.
	
	\subsection{Problem Statement}\label{sec:problem_statement}
	
	Consider a differential system of order $q\geq1$ on an open bounded domain $\Omega\subset\R^n$ with boundary $\partial\Omega$ and state $\vu:\Omega\to\R^m$. For a multi-index $\boldsymbol{\gamma}=(\gamma_1,\gamma_2,\dots,\gamma_n)\in\mathbb{N}_0^n$, let $|\boldsymbol{\gamma}|=\sum_{j=1}^n\gamma_j$ and $D^{\boldsymbol{\gamma}}\vu(\vx)=\frac{\partial^{|\boldsymbol{\gamma}|}\vu(\vx)}{\partial \evx_1^{\gamma_1}\partial \evx_2^{\gamma_2}\cdots\partial \evx_n^{\gamma_n}}$, with $D^{\vzero}\vu(\vx)=\vu(\vx)$. For any nonnegative integer $s$, define $\sI_s=\{\boldsymbol{\gamma}\in\mathbb{N}_0^n:|\boldsymbol{\gamma}|\leq s\}$ and the generalized differential state $\mathcal{D}^{(s)}\vu(\vx)=\{D^{\boldsymbol{\gamma}}\vu(\vx)\}_{\boldsymbol{\gamma}\in\sI_s}$. The differential system is
	\begin{equation}\label{eq:pde_forward}
		\mathbf{F}\!\left(\vx,\mathcal{D}^{(q)}\vu(\vx);\boldsymbol{\lambda}\right)=\vzero, \quad \forall\vx\in\Omega, \qquad
		\mathcal{B}\!\left(\vx,\mathcal{D}^{(q_b)}\vu(\vx)\right)=\vzero, \quad \forall\vx\in\partial\Omega,
	\end{equation}
	where $0\leq q_b\leq q$, $\mathbf{F}$ is the differential operator, $\mathcal{B}$ represents the boundary or initial constraint operator, and $\boldsymbol{\lambda}\in\R^{n_p}$ is the physical parameter vector. Let $\Omega_p\subset\Omega$, $\Omega_b\subset\partial\Omega$, and $\Omega_d\subset\Omega$ denote the finite sets of interior collocation points, boundary/initial points, and observational data points.
	
	\noindent\textbf{Forward problem.} Given $\boldsymbol{\lambda}$, the objective is to determine the   optimal set of network parameters $\{\vw_k,b_k,\boldsymbol{\beta}_k\}_{k=1}^L$ such that $\vu_L$ satisfies the governing equation and boundary/initial constraints in Eq.~(\ref{eq:pde_forward}) over $\Omega_p$ and $\Omega_b$, respectively, while matching available observations over $\Omega_d$.
	
	\noindent\textbf{Inverse problem.} When $\boldsymbol{\lambda}$ is unknown, the objective is to jointly determine $\{\vw_k,b_k,\boldsymbol{\beta}_k\}_{k=1}^L$ and identify $\boldsymbol{\lambda}$ such that $\vu_L$ satisfies the governing equation and boundary/initial constraints while fitting the observations over $\Omega_d$.
	
	\section{Physics-Informed Stochastic Configuration Machine}\label{sec:PISCM}
	\subsection{Forward Problem}\label{sec:algorithm-description}
	We begin with \textbf{PI-SC-I}, the localized construction algorithm. Suppose the network has generated $L-1$ hidden nodes, and the current state approximation is given by $\vu_{L-1}(\vx) = \sum_{k=1}^{L-1} \boldsymbol{\beta}_k h_k(\vx)$, where $\boldsymbol{\beta}_k \in \R^m$ denotes the output weight vector, $h_k(\vx) = g(\vw_k^\top \vx + b_k)$ represents the output of the $k$-th hidden node, $g(\cdot)$ is the activation function, and $\vw_k, b_k$ are the randomly configured input weights and bias, respectively. 
	
	At step $L$, we define the physical residual as $\ve_{L}^{p}:=\mathbf{F}\!\left(\vx,\mathcal{D}^{(q)}\vu_L\right)$ and the corresponding initial/boundary residual as $\ve_{L}^{b}:=\mathcal{B}\!\left(\vx,\mathcal{D}^{(q_b)}\vu_L\right)$. When evaluating the addition of the $L$-th node $h_L(\vx)$, we have $\vu_L(\vx) = \vu_{L-1}(\vx) + \boldsymbol{\beta}_L h_L(\vx)$. By examining the residual difference between step $L$ and step $L-1$, when the differential operator $\mathbf{F}$ is nonlinear, the physical residual increment is nonlinear with respect to the newly added output weight $\boldsymbol{\beta}_L$. This nonlinearity directly precludes the use of standard linear least squares methods for explicitly determining $\boldsymbol{\beta}_L$.
	
	To decouple this nonlinearity and project it into a linear subspace, we apply the first-order Taylor expansion around the previously established state $\mathcal{D}^{(q)}\vu_{L-1}$. By analytically evaluating the local Jacobians of the differential operator with respect to every derivative component up to order $q$, the physical residual is linearized as
	\begin{equation}\label{eq:Taylor}
		\ve_{L}^{p}= \mathbf{F}\!\left(\vx,\mathcal{D}^{(q)}\vu_L\right)\approx \mathbf{F}\!\left(\vx,\mathcal{D}^{(q)}\vu_{L-1}\right)
		+ \sum_{\boldsymbol{\gamma}\in\sI_q}\quad \mJ_{D^{\boldsymbol{\gamma}}\vu}(\vu_{L-1})
		\left(D^{\boldsymbol{\gamma}}\vu_L-D^{\boldsymbol{\gamma}}\vu_{L-1}\right),
	\end{equation}
	For each $\boldsymbol{\gamma}\in\sI_q$, the local Jacobian at the previous state is $\mJ_{D^{\boldsymbol{\gamma}}\vu}(\vu_{L-1})=\frac{\partial \mathbf{F}}{\partial(D^{\boldsymbol{\gamma}}\vu)}\big|_{\mathcal{D}^{(q)}\vu=\mathcal{D}^{(q)}\vu_{L-1}}$. In particular, the zero-order case $\boldsymbol{\gamma}=\vzero$ gives $\mJ_{D^{\vzero}\vu}=\mJ_{\vu}$. Because $\boldsymbol{\beta}_L$ is independent of $\vx$, the differential increment satisfies $D^{\boldsymbol{\gamma}}\vu_L-D^{\boldsymbol{\gamma}}\vu_{L-1}=\boldsymbol{\beta}_L D^{\boldsymbol{\gamma}}h_L(\vx)$ for $\boldsymbol{\gamma}\in\sI_q$. Substituting this relation into (\ref{eq:Taylor}) yields
	\begin{equation}\label{eq:ep_linear}
		\ve_{L}^{p} \approx \ve_{L-1}^{p} + \mA_L(\vx) \boldsymbol{\beta}_L,
	\end{equation}
	where
	\begin{equation}\label{eq:high_order_AL}
		\mA_L(\vx)
		=
		\sum_{\boldsymbol{\gamma}\in\sI_q}
		\mJ_{D^{\boldsymbol{\gamma}}\vu}(\vu_{L-1})
		D^{\boldsymbol{\gamma}}h_L(\vx).
	\end{equation}
	
	The nonlinear local residual evaluation is thus reduced to a linear mapping with respect to $\boldsymbol{\beta}_L$.	For the initial or boundary conditions, the operator can be formulated as $\mathcal{B}\!\left(\vx,\mathcal{D}^{(q_b)}\vu\right)=\sum_{\boldsymbol{\gamma}\in\sI_{q_b}}\mP_{\boldsymbol{\gamma}}(\vx)D^{\boldsymbol{\gamma}}\vu-\vg_{bc}(\vx)$, where $\mP_{\boldsymbol{\gamma}}(\vx)$ denotes the coefficient matrix and $\vg_{bc}(\vx)$ represents the prescribed boundary data. The boundary residual difference between steps $L$ and $L-1$ is $\ve_{L}^{b}-\ve_{L-1}^{b}=\left(\sum_{\boldsymbol{\gamma}\in\sI_{q_b}}\mP_{\boldsymbol{\gamma}}(\vx)D^{\boldsymbol{\gamma}}h_L(\vx)\right)\boldsymbol{\beta}_L$. This expression is linear with respect to $\boldsymbol{\beta}_L$, which yields
	\begin{equation}\label{eq:eb_linear}
		\ve_{L}^{b} = \ve_{L-1}^{b} + \mC_L(\vx) \boldsymbol{\beta}_L,
	\end{equation}
	where $\mC_L(\vx) = \sum_{\boldsymbol{\gamma}\in\sI_{q_b}}\mP_{\boldsymbol{\gamma}}(\vx)D^{\boldsymbol{\gamma}}h_L(\vx)$. 
	
	For semi-supervised learning, let $\ve_L^d:=\vu_{data}(\vx)-\vu_L(\vx)$ denote the observational residual, where $\vu_{data}(\vx)$ denotes the observations. The residual satisfies exactly $\ve_L^d=\ve_{L-1}^d+\mD_L(\vx)\boldsymbol{\beta}_L$, where $\mD_L(\vx)=-h_L(\vx)\mI_m$.
	
	Assuming that at the $L$-th step, we aim to simultaneously minimize $\ve_{L}^{p}$, $\ve_{L}^{b}$, and $\ve_{L}^{d}$, the comprehensive optimization objective for the PI-SC-I algorithm is formulated as the following linear least squares loss function:
	\begin{equation}\label{eq:opt_problem_semi}
		\begin{aligned}
			\mathcal{L}(\boldsymbol{\beta}_L)
			&= \frac{1}{N_p} \left\| \mA_L(\vx) \boldsymbol{\beta}_L + \ve_{L-1}^{p} \right\|_{\Omega_p}^2
			+ \frac{\omega_{bc}}{N_b} \left\| \mC_L(\vx) \boldsymbol{\beta}_L + \ve_{L-1}^{b} \right\|_{\Omega_b}^2 \\
			&\quad+ \frac{\omega_{data}}{N_d} \left\| \mD_L(\vx) \boldsymbol{\beta}_L + \ve_{L-1}^{d} \right\|_{\Omega_d}^2,
		\end{aligned}
	\end{equation}
	where $\omega_{bc}$ and $\omega_{data}$ are the penalty parameters balancing the initial/boundary conditions and the observational data against the physical domain. The notations $\|\cdot\|_{\Omega_p}^2$, $\|\cdot\|_{\Omega_b}^2$, and $\|\cdot\|_{\Omega_d}^2$ represent the standard discrete squared $\normltwo$ norms evaluated over $\Omega_p=\{\vx_i^p\}_{i=1}^{N_p}$, $\Omega_b=\{\vx_i^b\}_{i=1}^{N_b}$, and $\Omega_d=\{\vx_i^d\}_{i=1}^{N_d}$, respectively, where $N_p$, $N_b$, and $N_d$ are the corresponding numbers of physical, boundary, and data points. At the $L$-th step, we define
	\begin{equation}\label{eq:weighted_system}
			\tilde{\ve}_L
			=
			\begin{bmatrix}
				\dfrac{1}{\sqrt{N_p}}
				\underset{i=1,\ldots,N_p}{\operatorname{col}}
				\left(\ve_L^p(\vx_i^p)\right)
				\\
				\sqrt{\dfrac{\omega_{bc}}{N_b}}
				\underset{i=1,\ldots,N_b}{\operatorname{col}}
				\left(\ve_L^b(\vx_i^b)\right)
				\\
				\sqrt{\dfrac{\omega_{data}}{N_d}}
				\underset{i=1,\ldots,N_d}{\operatorname{col}}
				\left(\ve_L^d(\vx_i^d)\right)
			\end{bmatrix},
			\mM_L
			=
			\begin{bmatrix}
				\dfrac{1}{\sqrt{N_p}}
				\underset{i=1,\ldots,N_p}{\operatorname{col}}
				\left(\mA_L(\vx_i^p)\right)
				\\
				\sqrt{\dfrac{\omega_{bc}}{N_b}}
				\underset{i=1,\ldots,N_b}{\operatorname{col}}
				\left(\mC_L(\vx_i^b)\right)
				\\
				\sqrt{\dfrac{\omega_{data}}{N_d}}
				\underset{i=1,\ldots,N_d}{\operatorname{col}}
				\left(\mD_L(\vx_i^d)\right)
			\end{bmatrix},
			\vv_L
			=-\tilde{\ve}_{L-1},
	\end{equation}
	where $\operatorname{col}(\cdot)$ denotes the vertical concatenation operator. With these definitions, the objective in (\ref{eq:opt_problem_semi}) can be written as $\mathcal{L}(\boldsymbol{\beta}_L) = \left\|\mM_L\boldsymbol{\beta}_L-\vv_L\right\|^2$. Consequently, the minimum-norm solution of the linearized least squares problem is explicitly computed as $\widehat{\boldsymbol{\beta}}_L = \mM_L^{\dagger} \vv_L$. To control the Taylor truncation error of the true nonlinear residual, the adopted output-weight increment is defined as
	\begin{equation}\label{eq:analytical_solution}
		\boldsymbol{\beta}_L^* = \alpha \mM_L^{\dagger} \vv_L,
	\end{equation}
	where $\dagger$ denotes the Moore-Penrose pseudo-inverse operator, and $\alpha \in (0, 1]$ acts as a damping factor. Introducing this parameter guarantees the existence of a feasible parameter configuration that enforces strict monotonic descent of the true residual.
	
	To rigorously guarantee the convergence of the algorithm, we introduce a supervisory node selection mechanism. Using $\tilde{\ve}_L$ defined in (\ref{eq:weighted_system}), the comprehensive residual norm at the $L$-th step is given by $\mathcal{E}_L^2=\|\tilde{\ve}_L\|^2$. In the PI-SC-I algorithm, we establish a contraction evaluation metric $\xi_L = (r + \mu_L)\mathcal{E}_{L-1}^2 - \mathcal{E}_L^2$, where $r \in (0, 1)$ and $\mu_L = \frac{1-r}{L+1}$. During the generation of the $L$-th hidden node, we perform $T_{max}$ independent trials by randomly sampling the hidden parameters ($\vw_L, b_L$) and calculating their corresponding optimal output weights $\boldsymbol{\beta}_L^*$. Full candidate configurations satisfying $\xi_L > 0$ are deposited into a feasible pool. The algorithm ultimately selects the optimal parameter set that maximizes $\xi_L$, thereby completing the PI-SC-I localized update and enforcing the steepest residual descent.
	
	\begin{remark}
		Unlike Gauss--Newton methods, which repeatedly linearize a fixed nonlinear least-squares problem, PI-SCM integrates each linearization into node construction and requires only one linear least-squares solve for each candidate node. $\alpha$ and the supervisory node selection control the Taylor truncation error and enforce the residual descent required for the universal approximation property stated in Section~\ref{sec:universal-approximation}.
	\end{remark}
	
	In the original SCN framework \citep{wang2017stochastic}, a global evaluation of the output weights was designed after the local node selection to overcome the slow convergence caused by localized construction. Here, we adopt an analogous global correction, defining the \textbf{PI-SC-III} algorithm. Let $\mH_L(\vx) = [h_1(\vx)\mI_m, h_2(\vx)\mI_m, \dots, h_L(\vx)\mI_m]$ denote the global hidden layer output matrix, and the global output weight vector is denoted as $\boldsymbol{\beta}_{glob} = [\boldsymbol{\beta}_1^\top, \boldsymbol{\beta}_2^\top, \dots, \boldsymbol{\beta}_L^\top]^\top$. The globally updated state approximation is denoted as $\vu_L(\vx) = \mH_L(\vx) \boldsymbol{\beta}_{glob}$.
	
	Similarly, considering the nonlinearity of $\mathbf{F}$, we apply the first-order Taylor expansion around the locally constructed state $\mathcal{D}^{(q)}\vu_{loc}$. Specifically, let $\boldsymbol{\beta}_{loc}^* = [\boldsymbol{\beta}_1^\top, \boldsymbol{\beta}_2^\top, \dots, (\boldsymbol{\beta}_L^*)^\top]^\top$ denote the configuration achieved by the PI-SC-I algorithm at the current step, yielding the local state approximation $\vu_{loc}(\vx) = \mH_L(\vx) \boldsymbol{\beta}_{loc}^*$ and its corresponding physical residual $\ve_{loc}^{p}$.
	
	The local Jacobian matrices are then evaluated at the current localized configuration $\vu_{loc}(\vx)$. The physical residual is thereby linearized via these Jacobian operators as
	\begin{equation}\label{eq:global_ep_expanded}
			\ve_L^{p} \approx \ve_{loc}^{p}
			+ \sum_{\boldsymbol{\gamma}\in\sI_q}
			\mJ_{D^{\boldsymbol{\gamma}}\vu}(\vu_{loc})
			\left(D^{\boldsymbol{\gamma}}\mH_L(\vx)\boldsymbol{\beta}_{glob}
			-D^{\boldsymbol{\gamma}}\vu_{loc}(\vx)\right) = \mA_{glob}(\vx) \boldsymbol{\beta}_{glob} - \vb(\vx),
	\end{equation}
	where $\mA_{glob}(\vx)=\sum_{\boldsymbol{\gamma}\in\sI_q}\mJ_{D^{\boldsymbol{\gamma}}\vu}(\vu_{loc})D^{\boldsymbol{\gamma}}\mH_L(\vx)$, and $\vb(\vx)=\sum_{\boldsymbol{\gamma}\in\sI_q}\mJ_{D^{\boldsymbol{\gamma}}\vu}(\vu_{loc})D^{\boldsymbol{\gamma}}\vu_{loc}(\vx)-\ve_{loc}^{p}$. Following the same rationale, the initial/boundary condition residual is formulated as
	\begin{equation}\label{eq:global_eb}
			\ve_{L}^{b} = \sum_{\boldsymbol{\gamma}\in\sI_{q_b}}
			\mP_{\boldsymbol{\gamma}}(\vx)D^{\boldsymbol{\gamma}}\mH_L(\vx)\boldsymbol{\beta}_{glob}
			- \vg_{bc}(\vx) = \mC_{glob}(\vx) \boldsymbol{\beta}_{glob} - \vg_{bc}(\vx),
	\end{equation}
	where $\mC_{glob}(\vx) = \sum_{\boldsymbol{\gamma}\in\sI_{q_b}}\mP_{\boldsymbol{\gamma}}(\vx)D^{\boldsymbol{\gamma}}\mH_L(\vx)$. The data residual is formulated as:
	\begin{equation}\label{eq:global_ed}
		\ve_{L}^{d} = \vu_{data}(\vx) - \mH_L(\vx) \boldsymbol{\beta}_{glob}
		= \vu_{data}(\vx) - \mD_{glob}(\vx) \boldsymbol{\beta}_{glob},
	\end{equation}
	where $\mD_{glob}(\vx) = \mH_L(\vx)$.
	
	Building upon these global algebraic formulations, the comprehensive objective function for the PI-SC-III algorithm is constructed as a linear least squares problem:
	\begin{equation}\label{eq:opt_problem_global}
		\begin{split}
			\mathcal{L}_{glob}(\boldsymbol{\beta}_{glob}) &= \frac{1}{N_p} \left\| \mA_{glob}(\vx) \boldsymbol{\beta}_{glob} - \vb(\vx) \right\|_{\Omega_p}^2
			+ \frac{\omega_{bc}}{N_b} \left\| \mC_{glob}(\vx) \boldsymbol{\beta}_{glob} - \vg_{bc}(\vx) \right\|_{\Omega_b}^2 \\
			&\quad + \frac{\omega_{data}}{N_d} \left\| \mD_{glob}(\vx) \boldsymbol{\beta}_{glob} - \vu_{data}(\vx) \right\|_{\Omega_d}^2.
		\end{split}
	\end{equation}
	
	The three terms in (\ref{eq:opt_problem_global}) can be combined into a unified linear least squares system. The corresponding weighted global design matrix $\mM_{glob}$ and target vector $\vv_{glob}$ are given below:
	\begin{equation}\label{eq:global_matrices}
			\mM_{glob}
			=
			\begin{bmatrix}
				\dfrac{1}{\sqrt{N_p}}
				\underset{i=1,\ldots,N_p}{\operatorname{col}}
				\left(\mA_{glob}(\vx_i^p)\right)
				\\
				\sqrt{\dfrac{\omega_{bc}}{N_b}}
				\underset{i=1,\ldots,N_b}{\operatorname{col}}
				\left(\mC_{glob}(\vx_i^b)\right)
				\\
				\sqrt{\dfrac{\omega_{data}}{N_d}}
				\underset{i=1,\ldots,N_d}{\operatorname{col}}
				\left(\mD_{glob}(\vx_i^d)\right)
			\end{bmatrix},
			\vv_{glob}
			=
			\begin{bmatrix}
				\dfrac{1}{\sqrt{N_p}}
				\underset{i=1,\ldots,N_p}{\operatorname{col}}
				\left(\vb(\vx_i^p)\right)
				\\
				\sqrt{\dfrac{\omega_{bc}}{N_b}}
				\underset{i=1,\ldots,N_b}{\operatorname{col}}
				\left(\vg_{bc}(\vx_i^b)\right)
				\\
				\sqrt{\dfrac{\omega_{data}}{N_d}}
				\underset{i=1,\ldots,N_d}{\operatorname{col}}
				\left(\vu_{data}(\vx_i^d)\right)
			\end{bmatrix}.
	\end{equation}
	
	Accordingly, the objective in (\ref{eq:opt_problem_global}) can be written as $\mathcal{L}_{glob}(\boldsymbol{\beta}_{glob})=\left\|\mM_{glob}\boldsymbol{\beta}_{glob}-\vv_{glob}\right\|^2$. The output weights in PI-SC-III are obtained from the global pseudo-inverse:
	\begin{equation}\label{eq:global_analytical_solution}
		\tilde{\boldsymbol{\beta}}_{glob}^* = \mM_{glob}^{\dagger} \vv_{glob}.
	\end{equation}
	
	To control the Taylor truncation error of the true nonlinear residual and guarantee monotonic convergence under PI-SC-III, we introduce a discrete line search mechanism. We define the global search direction as $\Delta \boldsymbol{\beta} = \tilde{\boldsymbol{\beta}}_{glob}^* - \boldsymbol{\beta}_{loc}^*$. The output weights are given by $\boldsymbol{\beta}(\alpha) = \boldsymbol{\beta}_{loc}^* + \alpha \Delta \boldsymbol{\beta}$, with $\alpha \in [0,1]$.
	
	Let $\mathcal{E}_L^2(\alpha)$ denote the comprehensive true residual norm evaluated with the candidate weights $\boldsymbol{\beta}(\alpha)$. The algorithm computes in parallel a discrete set of candidate step sizes, denoted as $\sA \subset [0, 1]$, which includes the boundary limits $1.0$ and $0.0$. The best scaling factor $\alpha^*$ is selected as $\alpha^* = \mathop{\arg\min}_{\alpha \in \sA} \mathcal{E}_L^2(\alpha)$. The configuration is then adopted as the final updated weights $\boldsymbol{\beta}^* = \boldsymbol{\beta}(\alpha^*)$.
	
	Although PI-SC-III improves accuracy, its global correction becomes increasingly expensive as the network width $L$ grows. \textbf{PI-SC-II} follows the same correction principle but restricts its scope through a sliding window \citep{wang2017stochastic}. When $L\leq W$, PI-SC-II coincides with the global correction. When $L>W$, it freezes the first $L-W$ nodes and corrects only the most recent $W$ nodes. The active window fits the physical, boundary, and observational discrepancies left by the frozen block using the same linearization as PI-SC-III. PI-SC-II is then evaluated using the same discrete line search as PI-SC-III. Because the corresponding design matrix contains at most $mW$ columns, the correction cost is controlled by $W$ rather than the total network width $L$. The detailed derivation of PI-SC-II for forward problems is provided in Appendix~\ref{app:forward-window-correction}.
	
	It is worth noting that utilizing either PI-SC-II or PI-SC-III after each node selection (PI-SC-I) is not necessary. To reduce computational overhead, a correction interval $L_g$ is introduced such that either update is performed every $L_g$ nodes.

	\subsection{Inverse Problem}\label{sec:inverse problem}
	
	For inverse problems, the localized construction in PI-SC-I jointly estimates the output weight of the newly added node and the physical parameter increment $\Delta\boldsymbol{\lambda}_L
	=\boldsymbol{\lambda}_L-\boldsymbol{\lambda}_{L-1}$. Let
	$\mJ_{\lambda}(\vx;\vu_{L-1}, \boldsymbol{\lambda}_{L-1}) = \frac{\partial \mathbf{F}}{\partial \boldsymbol{\lambda}} \Big|_{\mathcal{D}^{(q)}\vu=\mathcal{D}^{(q)}\vu_{L-1}, \boldsymbol{\lambda}=\boldsymbol{\lambda}_{L-1}}$. Together with the state linearization represented by $\mA_L$ in Eq.~(\ref{eq:high_order_AL}), the joint linearized physical equation is
	\begin{equation}\label{eq:inverse_local_linearized}
		\left[
			\mA_L(\vx) \quad
			\mJ_{\lambda}
			(\vx;\vu_{L-1},\boldsymbol{\lambda}_{L-1})
		\right]
		\begin{bmatrix}
				\boldsymbol{\beta}_L \\
				\Delta\boldsymbol{\lambda}_L
		\end{bmatrix}
		\approx -\ve_{L-1}^{p}.
	\end{equation}
	Here, $\vtheta_L=[\boldsymbol{\beta}_L^\top,\Delta\boldsymbol{\lambda}_L^\top]^\top$. Because the boundary and observational constraints do not depend directly on $\boldsymbol{\lambda}$, they remain linear in $\vtheta_L$. Combining these constraints with the linearized physical equation yields a linear least-squares objective for estimating the state and physical parameters.
	
	The state and parameter candidate is damped and filtered by the established supervisory node selection mechanism used in the forward problem. PI-SC-II and PI-SC-III use the same method to jointly correct the output weights and physical parameters. The complete derivations for inverse problems are provided in Appendix~\ref{app:inverse-global-window-derivations}. The detailed training procedures are presented in Appendix~\ref{app:training-algorithms}.
	
	\subsection{Hidden Parameter Scale Selection}\label{sec:parameter selection}
	
	The candidate range $\Upsilon = \{\tau_1, \tau_2, \dots, \tau_{end}\}$ controls expressivity and numerical stability through $\rvw_L \sim U(-\tau, \tau)$ and $\rb_L \sim U(-\tau, \tau)$. For $h_L(\vx) = g(\vw_L^\top \vx + b_L)$ and $\boldsymbol{\gamma}\in\sI_q$, the chain rule gives $D^{\boldsymbol{\gamma}}h_L(\vx) = g^{(|\boldsymbol{\gamma}|)}(\vw_L^\top \vx+b_L)\prod_{j=1}^{n}\evw_{L,j}^{\gamma_j}$, where $\evw_{L,j}$ is the $j$-th hidden weight and $g^{(|\boldsymbol{\gamma}|)}$ is the corresponding activation derivative. Substitution into (\ref{eq:high_order_AL}) yields
	\begin{equation}\label{eq:AL_expanded}
		\mA_L(\vx)
		=
		\sum_{\boldsymbol{\gamma}\in\sI_q}
		\mJ_{D^{\boldsymbol{\gamma}}\vu}(\vu_{L-1})
		g^{(|\boldsymbol{\gamma}|)}(\vw_L^\top \vx+b_L)
		\prod_{j=1}^{n}\evw_{L,j}^{\gamma_j}.
	\end{equation}
	
	The weight factor $\prod_{j=1}^{n}\evw_{L,j}^{\gamma_j}$ scales as $\mathcal{O}(\tau^{|\boldsymbol{\gamma}|})$. Large $\tau$ can therefore amplify higher-order derivative contributions and worsen the conditioning of $\mM_L$ in (\ref{eq:weighted_system}). Conversely, small $\tau$ limits expressivity and can slow convergence. Therefore, we use $\Upsilon = a \times \{1, 2, 4, 8, 16, 32, 64\}$, with $a>0$ chosen to balance state and derivative contributions. PI-SCM searches these ranges in ascending order, expanding the range only when the current one fails to yield an accepted residual decrease.
	
\subsection{Universal Approximation Property}\label{sec:universal-approximation}

We now state the assumptions and universal approximation property for PI-SCM in forward problems. Detailed arguments including proofs are provided in Appendix~\ref{app:universal-approximation}.

\begin{assumption}
	\label{assumption:Lipschitz_continuous}
	Let $\vz = \underset{\boldsymbol{\gamma}\in\sI_q}{\operatorname{col}}\left(D^{\boldsymbol{\gamma}}\vu\right) = \operatorname{col}\!\left(\mathcal{D}^{(q)}\vu\right)$. The activation function satisfies $g\in C^q(\R)$. The generalized state $\vz(\vx)$ is uniformly bounded for all $\vx\in\Omega_p\cup\Omega_b\cup\Omega_d$. Define $\vz_{\mathrm{col}}(\vu) = \underset{\vx\in\Omega_p\cup\Omega_b\cup\Omega_d}{\operatorname{col}}\left(\vz(\vx)\right)$. Let $\sK$ be the closure of the reachable values of $\vz_{\mathrm{col}}(\vu)$. The differential operator $\mathbf{F}$ is assumed to be twice continuously Fr\'echet differentiable ($C^2$) with respect to $\vz$ on an open neighborhood of the local states represented in $\sK$.
\end{assumption}

\begin{assumption}
	\label{assumption:residual_completeness}
	Let $\Gamma$ be the admissible hidden-node family, and let $\mM(\vz_{\mathrm{col}},h)$ denote the weighted design matrix defined by Eq.~(\ref{eq:weighted_system}) at state $\vz_{\mathrm{col}}$ for candidate node $h$, so that $\mM_L(h)=\mM\bigl(\vz_{\mathrm{col}}(\vu_{L-1}),h\bigr)$. There exist a finite integer $J$ and fixed reference nodes $h_1,\ldots,h_J\in\Gamma$ such that, for every $\vz_{\mathrm{col}}\in\sK$, the block matrix
	\begin{equation}\label{eq:reference_design_matrix}
		\mM_{\mathrm{ref}}(\vz_{\mathrm{col}}) = [\mM(\vz_{\mathrm{col}},h_1),\ldots,\mM(\vz_{\mathrm{col}},h_J)]
	\end{equation}
	has full row rank. Moreover, for each fixed $j$, the rank of $\mM(\vz_{\mathrm{col}},h_j)$ is constant over $\sK$.
\end{assumption}

Under these assumptions, we now establish the universal approximation property of PI-SC-I.

\begin{theorem}\label{theorem:PI-SCM-I}
	Suppose that Assumptions~\ref{assumption:Lipschitz_continuous} and~\ref{assumption:residual_completeness} hold. There exist constants $\kappa\in(0,1]$ and $\bar{\alpha}\in(0,1]$ such that, for any $r\in(1-\kappa\bar{\alpha},1)$ and any nonnegative sequence $\{\mu_L\}$ satisfying $\lim_{L\to+\infty} \mu_L = 0$ and $\mu_L\leq1-r$, there exists a candidate hidden node $h_L\in\Gamma$ at each step $L$ such that, for any $\alpha\in((1-r)/\kappa,\bar{\alpha}]$, the output weight configuration $\boldsymbol{\beta}_L(\alpha)=\alpha\mM_L(h_L)^\dagger\vv_L$ guarantees that the following inequality holds:
	\begin{equation}\label{eq:theorem_inequality_nonlinear}
		\mathcal{E}_{L-1}^2-\mathcal{E}_L^2\geq\delta_L.
	\end{equation}
	where $\delta_L=(1-r-\mu_L)\mathcal{E}_{L-1}^2$. By accepting the node and the output weights that satisfy (\ref{eq:theorem_inequality_nonlinear}), the sequence of $\{\mathcal{E}_L^2\}$ monotonically decreases, satisfying $\lim_{L \to +\infty} \mathcal{E}_L^2 = 0$.
\end{theorem}

\begin{remark}
	Explicitly computing $\kappa$ and $\bar{\alpha}$ is unnecessary. Instead, $\alpha$ and $r$ are adaptively selected using the supervisory node selection mechanism. Both $\alpha$ and $r$ are initialized at the first values of their respective ordered lists, $\sA_1$ and $\sS_r$, and each advanced to its next value only when candidate nodes fail to satisfy the node selection.
\end{remark}

Let $\mathcal{E}_L^2(\boldsymbol{\beta})=\|\tilde{\ve}_L(\boldsymbol{\beta})\|^2$ and $\mathcal{E}_{L, loc}^2=\mathcal{E}_L^2(\boldsymbol{\beta}_{loc}^*)$. $\tilde{\boldsymbol{\beta}}_{win}^*$ is defined in (\ref{eq:window_analytical_solution}). We now extend the universal approximation property to PI-SC-III and PI-SC-II.

\begin{theorem}\label{theorem:PI-SCM-III}
	Suppose the conditions of Theorem \ref{theorem:PI-SCM-I} hold. We define the global search direction as $\Delta \boldsymbol{\beta} = \tilde{\boldsymbol{\beta}}_{glob}^* - \boldsymbol{\beta}_{loc}^*$. Then, there exists a constant $\bar{\alpha} \in [0, 1]$ such that for any $\alpha \in [0, \bar{\alpha}]$, the output weight configuration $\boldsymbol{\beta} = \boldsymbol{\beta}_{loc}^* + \alpha \Delta \boldsymbol{\beta}$ guarantees that the following inequality holds:
	\begin{equation}\label{eq:PI-SCM-IIIinequality}
		\mathcal{E}_L^2(\boldsymbol{\beta}) \leq \mathcal{E}_{L, loc}^2.
	\end{equation}
	
	By accepting the configuration that satisfies this inequality, the sequence of $\{\mathcal{E}_L^2\}$ monotonically decreases, satisfying $\lim_{L \to +\infty} \mathcal{E}_L^2 = 0$.
\end{theorem}

\begin{theorem}\label{theorem:PI-SCM-II}
	Suppose the conditions of Theorem \ref{theorem:PI-SCM-I} hold. We define the sliding-window search direction as $\Delta \boldsymbol{\beta} = \tilde{\boldsymbol{\beta}}_{win}^* - \boldsymbol{\beta}_{loc}^*$. Then, there exists constant $\bar{\alpha} \in [0, 1]$ such that for any $\alpha \in [0, \bar{\alpha}]$, the output weight configuration $\boldsymbol{\beta} = \boldsymbol{\beta}_{loc}^* + \alpha \Delta \boldsymbol{\beta}$ guarantees that the following inequality holds:
	\begin{equation}\label{eq:PI-SCM-IIinequality}
		\mathcal{E}_L^2(\boldsymbol{\beta}) \leq \mathcal{E}_{L, loc}^2.
	\end{equation}
	
	By accepting the configuration that satisfies this inequality, the sequence of $\{\mathcal{E}_L^2\}$ monotonically decreases, satisfying $\lim_{L \to +\infty} \mathcal{E}_L^2 = 0$.
\end{theorem}

	\section{Experimental Results}\label{sec:experiments}
	
	This section evaluates PI-SCM on forward problems involving the Van der Pol oscillator, the two-dimensional Helmholtz equation, the Allen--Cahn equation, the five-dimensional reaction--diffusion equation, the two-dimensional Bratu equation, the Kovasznay and Poiseuille flows, and nonlinear diffusion on a tree-shaped domain, together with inverse problems involving the Van der Pol, Helmholtz, and Bratu equations.
	
	For the forward and inverse results in
	Tables~\ref{tab:forward_performance} and~\ref{tab:inverse_performance}, all tabulated values are reported as mean $\pm$ standard deviation over 10 independent trials for Frozen-PINN and 100 independent trials for all other methods. State RMSE is evaluated on 10,000 randomly sampled test coordinates. The governing equations and experimental setup of the benchmark systems are provided in Appendix~\ref{app:supplementary-vdp}--\ref{app:supplementary-nonlinear-diffusion}. The accuracy and efficiency study, inverse experiments with noisy observations, activation function study, and ablation studies are reported in Appendices~\ref{app:accuracy-efficiency-study}, \ref{app:noisy-inverse-study}, \ref{app:activation-study}, and \ref{app:ablation-study}, respectively.
	
	\subsection{Forward Problems: Results and Analysis}
	
	Table~\ref{tab:forward_performance} reports training time and RMSE under the unsupervised and semi-supervised paradigms.

	\begin{table}[!t]
		\renewcommand{\arraystretch}{0.78}
		\setlength{\tabcolsep}{3pt}
		\setlength{\aboverulesep}{0.10ex}
		\setlength{\belowrulesep}{0.10ex}
		\centering
		\scriptsize
		\caption{Forward results across the eight benchmark systems.}
		\label{tab:forward_performance}
		\label{tab:vdp_performance}
		\label{tab:helmholtz_performance}
		\label{tab:ac_performance}
		\label{tab:reaction_diffusion_performance}
		\label{tab:bratu_performance}
		\label{tab:kovasznay_performance}
		\begin{tabular}{@{}lllcc@{}}
			\toprule
			System & Algorithm & Paradigm & Time (s) & RMSE \\
			\midrule
			Van der Pol oscillator
			& PINN~\citep{raissi2019physics} ($1\times300$) & Unsupervised & $21.3\pm0.57$ & $3.69e{-1}\pm9.18e{-2}$ \\
			& & Semi-supervised & $27.5\pm0.89$ & $8.05e{-2}\pm4.02e{-3}$ \\
			& PINN ($4\times64$) & Unsupervised & $40.0\pm1.02$ & $5.49e{-3}\pm3.03e{-3}$ \\
			& & Semi-supervised & $50.8\pm1.08$ & $1.34e{-3}\pm5.22e{-4}$ \\
			& PISCN~\citep{xu2025physics} & Unsupervised & -- & -- \\
			& & Semi-supervised & $13.8\pm0.37$ & $7.58e{-2}\pm2.59e{-2}$ \\
			& PIELM(NNBM)~\citep{huang2025neural} & Unsupervised & $0.50\pm0.10$ & $1.49e{-1}\pm4.63e{-1}$ \\
			& & Semi-supervised & $0.43\pm0.03$ & $1.84e{-3}\pm4.24e{-3}$ \\
			\cmidrule(lr){2-5}
			& PI-SC-I & Unsupervised & $0.52\pm0.01$ & $4.04e{-1}\pm1.48e{-1}$ \\
			& & Semi-supervised & $0.60\pm0.02$ & $1.16e{-1}\pm3.18e{-2}$ \\
			& PI-SC-II & Unsupervised & $0.70\pm0.02$ & $5.33e{-3}\pm8.44e{-3}$ \\
			& & Semi-supervised & $0.81\pm0.02$ & $2.17e{-3}\pm1.10e{-3}$ \\
			& PI-SC-III & Unsupervised & $0.70\pm0.02$ & $3.04e{-3}\pm6.71e{-3}$ \\
			& & Semi-supervised & $0.81\pm0.02$ & $5.03e{-4}\pm1.29e{-4}$ \\
			\midrule
			Helmholtz equation
			& PINN ($1\times400$) & Unsupervised & $40.5\pm1.23$ & $2.62e{-2}\pm1.97e{-2}$ \\
			& & Semi-supervised & $53.2\pm1.51$ & $1.13e{-2}\pm4.50e{-3}$ \\
			& PINN ($4\times64$) & Unsupervised & $75.6\pm1.11$ & $3.28e{-2}\pm1.43e{-2}$ \\
			& & Semi-supervised & $105\pm1.38$ & $1.44e{-2}\pm5.06e{-3}$ \\
			& PIELM~\citep{dwivedi2020physics} & Unsupervised & $0.02\pm0.00$ & $9.09e{-3}\pm3.83e{-3}$ \\
			& & Semi-supervised & $0.02\pm0.00$ & $6.61e{-3}\pm3.84e{-3}$ \\
			\cmidrule(lr){2-5}
			& PI-SC-I & Unsupervised & $0.57\pm0.01$ & $1.91e{-1}\pm9.82e{-2}$ \\
			& & Semi-supervised & $0.66\pm0.01$ & $1.57e{-1}\pm7.25e{-2}$ \\
			& PI-SC-II & Unsupervised & $1.17\pm0.06$ & $6.40e{-3}\pm2.21e{-3}$ \\
			& & Semi-supervised & $1.32\pm0.07$ & $6.65e{-3}\pm2.65e{-3}$ \\
			& PI-SC-III & Unsupervised & $1.21\pm0.07$ & $1.38e{-6}\pm1.93e{-6}$ \\
			& & Semi-supervised & $1.35\pm0.06$ & $1.07e{-6}\pm4.80e{-7}$ \\
			\midrule
			Allen--Cahn equation
			& VS-PINN~\citep{ko2025vs} ($1\times500$) & Unsupervised & $31.8\pm0.79$ & $5.40e{-1}\pm2.63e{-2}$ \\
			& & Semi-supervised & $36.2\pm0.94$ & $4.65e{-1}\pm4.67e{-2}$ \\
			& VS-PINN ($4\times64$) & Unsupervised & $65.0\pm1.06$ & $5.23e{-1}\pm4.14e{-2}$ \\
			& & Semi-supervised & $71.6\pm1.35$ & $7.75e{-2}\pm1.06e{-1}$ \\
			& PIELM(NNBM) & Unsupervised & $1.18\pm0.18$ & $3.48e{-2}\pm1.15e{-1}$ \\
			& & Semi-supervised & $1.06\pm0.30$ & $3.12e{-2}\pm7.43e{-2}$ \\
			& {Frozen-PINN-elm~\citep{datar2026fast}} & {Unsupervised} & {$0.15\pm0.01$} & {$2.99e{-3}\pm3.30e{-4}$} \\
			& & {Semi-supervised} & {$0.25\pm0.01$} & {$3.98e{-3}\pm4.76e{-4}$} \\
			& {Frozen-PINN-swim~\citep{datar2026fast}} & {Unsupervised} & {$0.20\pm0.04$} & {$1.11e{-3}\pm1.10e{-3}$} \\
			& & {Semi-supervised} & {$0.38\pm0.03$} & {$1.52e{-3}\pm1.72e{-3}$} \\
			\cmidrule(lr){2-5}
			& PI-SC-I & Unsupervised & $1.59\pm0.04$ & $5.28e{-1}\pm3.50e{-3}$ \\
			& & Semi-supervised & $1.73\pm0.03$ & $3.16e{-1}\pm2.47e{-2}$ \\
			& PI-SC-II & Unsupervised & $1.83\pm0.08$ & $6.45e{-2}\pm5.95e{-2}$ \\
			& & Semi-supervised & $1.96\pm0.04$ & $9.94e{-3}\pm4.39e{-3}$ \\
			& PI-SC-III & Unsupervised & $1.82\pm0.08$ & $1.63e{-2}\pm5.47e{-2}$ \\
			& & Semi-supervised & $2.01\pm0.08$ & $1.57e{-3}\pm2.66e{-4}$ \\
			\midrule
			Reaction--diffusion equation
			& PINN ($1\times1000$) & Unsupervised & $382\pm4.30$ & $1.62e{-2}\pm2.41e{-3}$ \\
			& & Semi-supervised & $395\pm4.05$ & $1.31e{-2}\pm1.62e{-3}$ \\
			& PINN ($4\times64$) & Unsupervised & $379\pm2.71$ & $5.12e{-3}\pm3.38e{-3}$ \\
			& & Semi-supervised & $398\pm6.83$ & $5.19e{-3}\pm3.32e{-3}$ \\
			& PIELM(NNBM) & Unsupervised & $1.61\pm0.03$ & $1.29e{-2}\pm1.35e{-3}$ \\
			& & Semi-supervised & $1.62\pm0.02$ & $1.29e{-2}\pm1.24e{-3}$ \\
			& {Frozen-PINN-elm} & {Unsupervised} & {$53.5\pm0.25$} & {$2.36e{-2}\pm2.48e{-3}$} \\
			& & {Semi-supervised} & {$56.2\pm0.71$} & {$2.33e{-2}\pm2.21e{-3}$} \\
			& {Frozen-PINN-swim} & {Unsupervised} & {$28.2\pm1.31$} & {$7.87e{-3}\pm1.56e{-3}$} \\
			& & {Semi-supervised} & {$37.3\pm1.00$} & {$7.88e{-3}\pm1.55e{-3}$} \\
			& {Frozen-PINN-swim (with projection)} & {Unsupervised} & {$17.5\pm3.34$} & {$6.57e{-5}\pm5.38e{-9}$} \\
			& & {Semi-supervised} & {$21.2\pm1.23$} & {$6.55e{-5}\pm2.56e{-8}$} \\
			\cmidrule(lr){2-5}
			& PI-SC-I & Unsupervised & $5.17\pm0.16$ & $1.01e{-1}\pm6.46e{-3}$ \\
			& & Semi-supervised & $5.43\pm0.06$ & $9.35e{-2}\pm5.80e{-3}$ \\
			& PI-SC-II & Unsupervised & $5.47\pm0.09$ & $2.30e{-2}\pm1.74e{-3}$ \\
			& & Semi-supervised & $5.76\pm0.09$ & $2.23e{-2}\pm1.53e{-3}$ \\
			& PI-SC-III & Unsupervised & $5.84\pm0.10$ & $6.48e{-3}\pm4.80e{-4}$ \\
			& & Semi-supervised & $6.28\pm0.13$ & $6.39e{-3}\pm4.63e{-4}$ \\
			\midrule
			Bratu equation
			& PINN ($1\times300$) & Unsupervised & $86.5\pm1.24$ & $3.64e{-4}\pm1.42e{-4}$ \\
			& & Semi-supervised & $95.8\pm1.31$ & $4.10e{-4}\pm1.23e{-4}$ \\
			& PINN ($4\times64$) & Unsupervised & $133\pm1.11$ & $1.42e{-4}\pm1.11e{-4}$ \\
			& & Semi-supervised & $148\pm1.32$ & $1.15e{-4}\pm8.73e{-5}$ \\
			& PIELM(NNBM) & Unsupervised & $0.13\pm0.00$ & $9.84e{-6}\pm8.88e{-8}$ \\
			& & Semi-supervised & $0.20\pm0.00$ & $7.17e{-6}\pm8.76e{-8}$ \\
			\cmidrule(lr){2-5}
			& PI-SC-I & Unsupervised & $1.25\pm0.06$ & $1.83e{-2}\pm4.13e{-3}$ \\
			& & Semi-supervised & $1.31\pm0.04$ & $1.54e{-2}\pm3.47e{-3}$ \\
			& PI-SC-II & Unsupervised & $1.27\pm0.07$ & $1.03e{-5}\pm4.33e{-7}$ \\
			& & Semi-supervised & $1.39\pm0.07$ & $8.03e{-6}\pm5.06e{-7}$ \\
			& PI-SC-III & Unsupervised & $1.27\pm0.05$ & $9.16e{-6}\pm1.56e{-7}$ \\
			& & Semi-supervised & $1.39\pm0.06$ & $6.59e{-6}\pm1.79e{-7}$ \\
			\midrule
			{Kovasznay flow}
			& {PINN ($1\times400$)} & {Unsupervised} & {$80.0\pm1.84$} & {$9.31e{-2}\pm1.89e{-2}$} \\
			& & {Semi-supervised} & {$99.3\pm2.56$} & {$5.95e{-2}\pm1.39e{-2}$} \\
			& {PINN ($4\times64$)} & {Unsupervised} & {$123\pm1.31$} & {$1.52e{-3}\pm2.56e{-4}$} \\
			& & {Semi-supervised} & {$158\pm7.11$} & {$5.56e{-4}\pm7.34e{-5}$} \\
			& {PIELM(NNBM)} & {Unsupervised} & {$3.76\pm0.64$} & {$9.49e{-11}\pm7.11e{-11}$} \\
			& & {Semi-supervised} & {$3.96\pm0.73$} & {$5.78e{-11}\pm6.35e{-11}$} \\
			\cmidrule(lr){2-5}
			& {PI-SC-I} & {Unsupervised} & {$1.20\pm0.05$} & {$5.34e{-2}\pm1.11e{-2}$} \\
			& & {Semi-supervised} & {$1.39\pm0.06$} & {$1.95e{-2}\pm7.21e{-3}$} \\
			& {PI-SC-II} & {Unsupervised} & {$1.89\pm0.06$} & {$4.24e{-8}\pm4.43e{-8}$} \\
			& & {Semi-supervised} & {$2.02\pm0.05$} & {$6.85e{-8}\pm5.86e{-8}$} \\
			& {PI-SC-III} & {Unsupervised} & {$2.08\pm0.07$} & {$5.30e{-13}\pm1.11e{-12}$} \\
			& & {Semi-supervised} & {$2.27\pm0.05$} & {$4.97e{-13}\pm1.03e{-12}$} \\
			\bottomrule
		\end{tabular}
	\end{table}
	
	\begin{table}[!t]
		\tabletopcaptrue
		\phantomsection
		\renewcommand{\arraystretch}{0.78}
		\setlength{\tabcolsep}{3pt}
		\setlength{\aboverulesep}{0.10ex}
		\setlength{\belowrulesep}{0.10ex}
		\centering
		\scriptsize
		\captcont{Forward results across the eight benchmark systems (continued).}
		\label{tab:poiseuille_performance}
		\label{tab:nonlinear_diffusion_performance}
		\begin{tabular}{@{}lllcc@{}}
			\toprule
			System & Algorithm & Paradigm & Time (s) & RMSE \\
			\midrule
			{Poiseuille flow}
			& {PINN ($1\times200$)} & {Unsupervised} & {$70.1\pm0.21$} & {$3.36e{-1}\pm3.44e{-2}$} \\
			& & {Semi-supervised} & {$87.5\pm0.66$} & {$6.95e{-3}\pm2.21e{-3}$} \\
			& {PINN ($4\times64$)} & {Unsupervised} & {$108\pm1.80$} & {$2.61e{-1}\pm2.13e{-2}$} \\
			& & {Semi-supervised} & {$136\pm2.02$} & {$6.13e{-3}\pm3.24e{-3}$} \\
			& {PIELM(NNBM)} & {Unsupervised} & {$0.55\pm0.02$} & {$3.43e{-8}\pm1.08e{-8}$} \\
			& & {Semi-supervised} & {$0.71\pm0.99$} & {$3.89e{-10}\pm1.27e{-10}$} \\
			\cmidrule(lr){2-5}
			& {PI-SC-I} & {Unsupervised} & {$1.09\pm0.02$} & {$5.22e{-1}\pm2.06e{-2}$} \\
			& & {Semi-supervised} & {$1.27\pm0.03$} & {$7.40e{-2}\pm1.59e{-2}$} \\
			& {PI-SC-II} & {Unsupervised} & {$1.37\pm0.02$} & {$5.06e{-7}\pm1.64e{-6}$} \\
			& & {Semi-supervised} & {$1.61\pm0.02$} & {$8.46e{-10}\pm1.78e{-9}$} \\
			& {PI-SC-III} & {Unsupervised} & {$1.68\pm0.06$} & {$6.59e{-14}\pm4.59e{-14}$} \\
			& & {Semi-supervised} & {$1.73\pm0.04$} & {$2.92e{-14}\pm1.06e{-14}$} \\
			\midrule
			{Nonlinear diffusion}
			& {PINN ($1\times600$)} & {Unsupervised} & {$138\pm2.61$} & {$1.43e{-1}\pm2.88e{-3}$} \\
			& & {Semi-supervised} & {$171\pm2.59$} & {$2.69e{-2}\pm4.86e{-3}$} \\
			& {PINN ($4\times64$)} & {Unsupervised} & {$160\pm1.36$} & {$1.52e{-2}\pm1.29e{-2}$} \\
			& & {Semi-supervised} & {$197\pm3.16$} & {$1.33e{-2}\pm1.16e{-2}$} \\
			& {PIELM(NNBM)} & {Unsupervised} & {$0.86\pm0.07$} & {$1.62e{-6}\pm7.93e{-7}$} \\
			& & {Semi-supervised} & {$0.89\pm0.45$} & {$6.07e{-7}\pm2.79e{-7}$} \\
			& {Frozen-PINN-swim} & {Unsupervised} & {$404\pm106$} & {$9.08e{-5}\pm1.86e{-4}$} \\
			& & {Semi-supervised} & {$399\pm100$} & {$9.00e{-5}\pm1.83e{-4}$} \\
			\cmidrule(lr){2-5}
			& {PI-SC-I} & {Unsupervised} & {$1.38\pm0.05$} & {$2.16e{-1}\pm2.68e{-2}$} \\
			& & {Semi-supervised} & {$1.65\pm0.03$} & {$1.42e{-1}\pm3.98e{-2}$} \\
			& {PI-SC-II} & {Unsupervised} & {$1.54\pm0.02$} & {$5.79e{-4}\pm2.00e{-4}$} \\
			& & {Semi-supervised} & {$1.82\pm0.03$} & {$2.72e{-4}\pm1.13e{-4}$} \\
			& {PI-SC-III} & {Unsupervised} & {$1.85\pm0.08$} & {$1.34e{-6}\pm1.33e{-6}$} \\
			& & {Semi-supervised} & {$2.01\pm0.79$} & {$5.13e{-7}\pm2.33e{-7}$} \\
			\bottomrule
		\end{tabular}
	\end{table}
	
	Table~\ref{tab:forward_performance} shows that PI-SCM provides a strong combination of predictive accuracy and computational efficiency. PI-SC-I has the lowest construction cost, PI-SC-II balances accuracy and efficiency, and PI-SC-III generally achieves the highest accuracy.
	
	We use $L_g=20$ for Van der Pol, $L_g=10$ for Helmholtz, Allen--Cahn, reaction--diffusion, and Poiseuille, and $L_g=30$ for Bratu, Kovasznay, and nonlinear diffusion. At these correction intervals, PI-SC-III trains approximately $36$ to $107$ times faster than the corresponding deep PINNs. It reduces RMSE by factors of $1.8$ to $2.7$ on Van der Pol, $32$ to $49$ on Allen--Cahn, $16$ to $17$ on Bratu, and $1.1\times10^4$ to $2.6\times10^4$ on nonlinear diffusion. The reductions exceed four orders of magnitude on Helmholtz and reach approximately nine orders on Kovasznay and eleven to twelve orders on Poiseuille. On reaction--diffusion, it achieves comparable RMSE while training approximately $63$ to $65$ times faster.
	
	Compared with the PIELM baselines, PI-SC-III at these intervals lowers RMSE on Van der Pol, Helmholtz, Allen--Cahn, reaction--diffusion, Kovasznay, Poiseuille, and nonlinear diffusion by factors ranging from approximately $1.2$ to $5.2\times10^{5}$. On Bratu, it reduces the unsupervised and semi-supervised RMSEs from $9.84\times10^{-6}$ to $9.16\times10^{-6}$ and from $7.17\times10^{-6}$ to $6.59\times10^{-6}$, respectively. On nonlinear diffusion, it reduces the unsupervised and semi-supervised RMSEs from $1.62\times10^{-6}$ to $1.34\times10^{-6}$ and from $6.07\times10^{-7}$ to $5.13\times10^{-7}$, respectively. These results demonstrate the benefit of using physics-informed constraints to guide both node construction and output weight estimation.
	
\begingroup
	Compared with the Frozen-PINN baselines, PI-SC-III reduces RMSE on nonlinear diffusion by factors of approximately $68$ to $175$ and trains approximately $199$ to $218$ times faster. On reaction--diffusion, it trains approximately $3.0$ to $9.2$ times faster than the three Frozen-PINN baselines and achieves lower RMSE than Frozen-PINN-elm and Frozen-PINN-swim. Frozen-PINN-swim with projection achieves RMSE near $6.6\times10^{-5}$, approximately two orders of magnitude lower than PI-SC-III. The projection restricts basis variation to the spatial directions on which the solution depends, thereby improving approximation accuracy. On Allen--Cahn, Frozen-PINN trains faster and achieves predictive accuracy comparable to that of PI-SC-III. The semi-supervised Frozen-PINN settings provide limited accuracy gains from the 100 additional observations.
	
	Overall, the forward experiments show that physics-informed node construction and analytical output weight updates enable PI-SCM to approximate solutions across both static and time-dependent systems. PI-SC-III combines substantial training time savings relative to PINNs with improved accuracy over the PIELM baselines.
	
\endgroup
	\subsection{Inverse Problems: Results and Analysis}

	Table~\ref{tab:inverse_performance} reports parameter identification error, training time, and state RMSE.

	\begin{table}[!t]
		\renewcommand{\arraystretch}{0.78}
		\setlength{\tabcolsep}{3pt}
		\setlength{\aboverulesep}{0.10ex}
		\setlength{\belowrulesep}{0.10ex}
		\centering
		\scriptsize
		\caption{Inverse results across the Van der Pol, two-dimensional Helmholtz, and Bratu systems, where the identified parameters are $\mu$, $k^2$, and $\lambda$, respectively.}
		\label{tab:inverse_performance}
		\label{tab:vdp_inverse}
		\label{tab:helmholtz_inverse_performance}
		\label{tab:bratu_inverse_performance}
		\begin{tabular}{@{}llccc@{}}
			\toprule
			System & Algorithm & $|\lambda_{pred}-\lambda_{true}|$ & Time (s) & State RMSE \\
			\midrule
			Van der Pol oscillator
			& PINN ($1\times300$) & $2.88e{-1}\pm1.09e{-1}$ & $22.0\pm0.51$ & $5.85e{-2}\pm1.35e{-2}$ \\
			& PINN ($4\times64$) & $2.35e{-4}\pm2.01e{-4}$ & $51.8\pm0.94$ & $9.51e{-4}\pm3.42e{-4}$ \\
			\cmidrule(lr){2-5}
			& PI-SC-I & $6.72e{-1}\pm5.12e{-2}$ & $0.80\pm0.02$ & $1.40e{-1}\pm2.27e{-2}$ \\
			& PI-SC-II & $5.41e{-3}\pm1.29e{-3}$ & $0.99\pm0.05$ & $1.10e{-3}\pm1.97e{-4}$ \\
			& PI-SC-III & $3.57e{-3}\pm1.14e{-3}$ & $1.01\pm0.02$ & $7.37e{-4}\pm1.80e{-4}$ \\
			\midrule
			Helmholtz equation
			& PINN ($1\times400$) & $1.09e{+1}\pm2.87e{+0}$ & $50.6\pm1.19$ & $2.82e{-1}\pm1.27e{-1}$ \\
			& PINN ($4\times64$) & $6.72e{-1}\pm6.27e{-1}$ & $101\pm2.13$ & $2.36e{-2}\pm1.22e{-2}$ \\
			\cmidrule(lr){2-5}
			& PI-SC-I & $3.72e{+1}\pm1.93e{+1}$ & $1.49\pm0.03$ & $2.17e{-1}\pm9.90e{-2}$ \\
			& PI-SC-II & $1.98e{+0}\pm2.44e{+0}$ & $1.92\pm0.06$ & $7.85e{-3}\pm6.54e{-3}$ \\
			& PI-SC-III & $5.07e{-4}\pm7.64e{-4}$ & $2.02\pm0.08$ & $3.84e{-5}\pm4.88e{-5}$ \\
			\midrule
			Bratu equation
			& PINN ($1\times300$) & $7.30e{-2}\pm3.50e{-2}$ & $97.1\pm1.08$ & $2.64e{-3}\pm1.29e{-3}$ \\
			& PINN ($4\times64$) & $2.05e{-3}\pm7.78e{-3}$ & $145\pm3.73$ & $1.42e{-4}\pm2.35e{-4}$ \\
			\cmidrule(lr){2-5}
			& PI-SC-I & $7.01e{-1}\pm1.60e{-2}$ & $1.58\pm0.05$ & $1.90e{-2}\pm5.27e{-4}$ \\
			& PI-SC-II & $1.27e{-4}\pm4.10e{-6}$ & $1.69\pm0.04$ & $7.86e{-6}\pm6.09e{-7}$ \\
			& PI-SC-III & $1.31e{-4}\pm1.85e{-6}$ & $1.66\pm0.04$ & $6.50e{-6}\pm1.58e{-7}$ \\
			\bottomrule
		\end{tabular}
	\end{table}
	
	The inverse results in Table~\ref{tab:inverse_performance} show the same trend as in the forward experiments. PI-SC-I trains fastest, PI-SC-II balances speed and accuracy, and PI-SC-III generally achieves the highest accuracy.
	
	Across all three inverse benchmarks, PI-SC-III achieves lower state RMSE than the deep PINN at substantially lower training cost. It also achieves lower parameter identification errors on Helmholtz and Bratu. These results demonstrate that PI-SCM retains high predictive accuracy and computational efficiency in inverse problems.
	
	\section{Conclusion}\label{sec:conclusion}	
	This paper introduced PI-SCM, a backpropagation-free framework for forward and inverse problems involving nonlinear differential equations. Physical, boundary, and observational constraints guide both hidden node construction and output weight estimation, embedding physics-informed learning throughout the constructive training process. By coupling residual linearization with progressive node construction, PI-SCM removes the need for an inner nonlinear optimization loop, thereby providing the computational foundation for its efficiency.
	
	Within the PI-SCM framework, PI-SC-I provided the fastest training but was limited in accuracy, PI-SC-III achieved the highest accuracy at greater computational cost, and PI-SC-II balanced accuracy and efficiency. In particular, PI-SC-III achieved errors comparable to or lower than those of gradient-based PINNs while reducing training time by orders of magnitude, and attained substantially higher accuracy compare with the PIELM baselines. These results highlight the potential of PI-SCM to serve as a powerful and scalable foundation for real-time scientific machine learning applications.
	
	Future work should examine scalability to higher dimensional geometries, more severe multiscale stiffness, and adaptive selection of the candidate range and update scope. The modular design of PI-SCM provides an extensible algorithmic foundation for incorporating complementary advances from the randomized network, PINN, and SCN communities. Promising directions include adaptive sampling \citep{wu2023comprehensive}, domain decomposition \citep{dong2021local}, the constrained expressions of boundary conditions \citep{schiassi2021extreme}, the enforcement of temporal causality \citep{datar2026fast}, and deep or adaptive SCN constructions \citep{wang2018deep,felicetti2022deep,felicetti2022deepoptimized}. These integrations could extend the current architecture while preserving analytical training and numerical stability.

	\bibliography{references}
	\bibliographystyle{plainnat}

	\clearpage
\appendix
\section*{APPENDIX}
\section{Extended Background and Related Work}\label{app:extended-related-work}

This appendix expands the background summarized in the Introduction, covering applications of physics-informed learning, the development of randomized networks and SCNs, and a related least-squares approach to dynamical-system learning.

\noindent\textbf{Physics-Informed Learning and Applications.} Scientific machine learning integrates physical knowledge with data-driven approximation \citep{baker2019workshop}. PINNs impose governing equations through the training loss \citep{raissi2019physics} and have been applied across several scientific and engineering domains. \citet{schiassi2022physics} used PINNs to learn state--costate dynamics satisfying optimality conditions for orbital transfers. \citet{shukla2020physics} studied surface-breaking cracks in metal plates, while \citet{berkhahn2022physics} modeled COVID-19 infection and hospitalization scenarios. \citet{diao2023solving} addressed boundary conditions in multi-material solid-mechanics problems. These applications demonstrate the broad applicability of PINNs across scientific and engineering domains.

Beyond these applications, Adaptive activation functions \citep{jagtap2020adaptive}, adaptive loss balancing \citep{wang2021understanding,xiang2022self}, and second-order optimization for gradient alignment \citep{wang2026gradient} have been developed to improve PINN convergence and training accuracy. Least-square finite difference approximations \citep{xiao2024least} replace AD for spatial derivatives to reduce computational costs.

\begingroup
\noindent\textbf{Randomized Networks and Stochastic Configuration Networks.} Randomized networks separate hidden parameter assignment from output weight estimation. Randomly generated hidden nodes provide nonlinear basis functions, and the output layer combines their outputs through linear coefficients. Once the hidden parameters are fixed, these coefficients can be estimated through linear least squares. Early examples include Random Vector Functional Link Networks (RVFLNs) \citep{pao1992functional} and random-weight feedforward networks \citep{schmidt1992feed}. Extreme Learning Machines (ELMs) \citep{huang2006extreme} further developed this approach by combining randomly assigned hidden parameters with analytical output weight computation.

Incremental Random Weight Neural Networks (IRWNNs) \citep{liu2002constructive} construct networks node by node, allowing approximation capacity to grow during training. Stochastic Configuration Networks (SCNs) \citep{wang2017stochastic} incrementally construct networks by sampling hidden parameters within adaptive ranges. A supervisory inequality selects nodes that sufficiently reduce the current residual, supporting residual convergence and universal approximation under the stated assumptions.
SC-I, SC-II, and SC-III estimate output weights for the newly added node, a sliding window of recent nodes, and all nodes, respectively, through analytical least-squares updates.

Subsequent work extends SCNs through deep architectures \citep{wang2018deep}, genetic-algorithm-based optimization \citep{li2018soft}, and Monte Carlo tree search strategies \citep{felicetti2022deep,felicetti2022deepoptimized}. Applications include robotic grasping recognition \citep{pan2020fpga}, greengage grading \citep{li2019greengage}, glass-insulator state recognition \citep{li2022stochastic}, plastic-bottle recognition \citep{chen2022research}, and photovoltaic generation forecasting with missing inputs \citep{liu2020pv}. These developments demonstrate the widespread application of SCNs.

\endgroup
\noindent\textbf{Related Least-Squares Learning Approaches.} Operator-theoretic approaches provide another route to efficient dynamical-system learning. \citet{wu2026least} proposed a parallelizable multi-step extended dynamic mode decomposition (EDMD) method that decomposes least-squares identification across prediction horizons and state coordinates. Although this approach enables efficient offline modeling, the cited formulation remains data-driven and does not impose physical constraints.

\section{Additional Algorithmic Details}\label{app:algorithmic-details}

\subsection{PI-SC-II Derivations for Forward Problems}\label{app:forward-window-correction}

PI-SC-II follows the correction principle of PI-SC-III while restricting the active scope to a window of $W$ hidden nodes. When $L\leq W$, it uses the global correction. The following formulation considers $W<L$, for which the first $L-W$ nodes are frozen and only the most recent $W$ nodes are corrected. We partition the hidden layer matrix and output weights as $\mH_L(\vx)=[\mH_{pre}(\vx),\mH_{win}(\vx)]$ and $\boldsymbol{\beta}_{glob}=[\boldsymbol{\beta}_{pre}^{\top},\boldsymbol{\beta}_{win}^{\top}]^{\top}$, respectively. The frozen state is $\vu_{L-W}(\vx)=\mH_{pre}(\vx)\boldsymbol{\beta}_{pre}$.

The active window fits the discrepancies left by the frozen state. The residual targets and design matrices are defined below.
\begin{equation}\label{eq:window_components}
	\begin{alignedat}{2}
		\vr_{L-W}^{p}(\vx)
		&= \vb(\vx) - \sum_{\boldsymbol{\gamma}\in\sI_q}
		\mJ_{D^{\boldsymbol{\gamma}}\vu}(\vu_{loc})D^{\boldsymbol{\gamma}}\vu_{L-W}(\vx),
		&\quad \mA_W(\vx)
		&= \sum_{\boldsymbol{\gamma}\in\sI_q}
		\mJ_{D^{\boldsymbol{\gamma}}\vu}(\vu_{loc})D^{\boldsymbol{\gamma}}\mH_{win}(\vx), \\
		\vr_{L-W}^{b}(\vx)
		&= \vg_{bc}(\vx) - \sum_{\boldsymbol{\gamma}\in\sI_{q_b}}
		\mP_{\boldsymbol{\gamma}}(\vx)D^{\boldsymbol{\gamma}}\vu_{L-W}(\vx),
		&\quad \mC_W(\vx)
		&= \sum_{\boldsymbol{\gamma}\in\sI_{q_b}}
		\mP_{\boldsymbol{\gamma}}(\vx)D^{\boldsymbol{\gamma}}\mH_{win}(\vx), \\
		\vr_{L-W}^{d}(\vx)
		&= \vu_{data}(\vx) - \vu_{L-W}(\vx),
		&\quad \mD_W(\vx)
		&= \mH_{win}(\vx).
	\end{alignedat}
\end{equation}

The window correction is formulated as the linear least squares objective
\begin{equation}\label{eq:opt_problem_window}
	\begin{split}
		\mathcal{L}_{win}(\boldsymbol{\beta}_{win}) &= \frac{1}{N_p} \left\| \mA_W(\vx) \boldsymbol{\beta}_{win} - \vr_{L-W}^{p}(\vx) \right\|_{\Omega_p}^2
		+ \frac{\omega_{bc}}{N_b} \left\| \mC_W(\vx) \boldsymbol{\beta}_{win} - \vr_{L-W}^{b}(\vx) \right\|_{\Omega_b}^2 \\
		&\quad+ \frac{\omega_{data}}{N_d} \left\| \mD_W(\vx) \boldsymbol{\beta}_{win} - \vr_{L-W}^{d}(\vx) \right\|_{\Omega_d}^2.
	\end{split}
\end{equation}

The three terms in (\ref{eq:opt_problem_window}) give the weighted design matrix $\mM_W$ and target vector $\vv_W$:
\begin{equation}\label{eq:window_matrices}
	\mM_W
	=
	\begin{bmatrix}
		\dfrac{1}{\sqrt{N_p}}
		\underset{i=1,\ldots,N_p}{\operatorname{col}}
		\left(\mA_W(\vx_i^p)\right)
		\\
		\sqrt{\dfrac{\omega_{bc}}{N_b}}
		\underset{i=1,\ldots,N_b}{\operatorname{col}}
		\left(\mC_W(\vx_i^b)\right)
		\\
		\sqrt{\dfrac{\omega_{data}}{N_d}}
		\underset{i=1,\ldots,N_d}{\operatorname{col}}
		\left(\mD_W(\vx_i^d)\right)
	\end{bmatrix},
	\vv_W
	=
	\begin{bmatrix}
		\dfrac{1}{\sqrt{N_p}}
		\underset{i=1,\ldots,N_p}{\operatorname{col}}
		\left(\vr_{L-W}^{p}(\vx_i^p)\right)
		\\
		\sqrt{\dfrac{\omega_{bc}}{N_b}}
		\underset{i=1,\ldots,N_b}{\operatorname{col}}
		\left(\vr_{L-W}^{b}(\vx_i^b)\right)
		\\
		\sqrt{\dfrac{\omega_{data}}{N_d}}
		\underset{i=1,\ldots,N_d}{\operatorname{col}}
		\left(\vr_{L-W}^{d}(\vx_i^d)\right)
	\end{bmatrix}.
\end{equation}

Thus, $\mathcal{L}_{win}(\boldsymbol{\beta}_{win})=\|\mM_W\boldsymbol{\beta}_{win}-\vv_W\|_2^2$. The active weights are computed using pseudo inverse and reassembled with the frozen weights as
\begin{equation}\label{eq:window_analytical_solution}
	\boldsymbol{\beta}_{win}^* = \mM_W^{\dagger} \vv_W, \qquad
	\tilde{\boldsymbol{\beta}}_{win}^* =
	\begin{bmatrix}
		\boldsymbol{\beta}_{pre}\\
		\boldsymbol{\beta}_{win}^*
	\end{bmatrix}.
\end{equation}

The matrix $\mM_W$ contains at most $mW$ columns, independent of the total network width $L$. To control the Taylor truncation error, PI-SC-II applies the same discrete line search as PI-SC-III, with $\Delta\boldsymbol{\beta}=\tilde{\boldsymbol{\beta}}_{win}^*-\boldsymbol{\beta}_{loc}^*$ and $\boldsymbol{\beta}(\alpha)=\boldsymbol{\beta}_{loc}^*+\alpha\Delta\boldsymbol{\beta}$. The selected step size $\alpha^*$ gives the final weights $\boldsymbol{\beta}^*=\boldsymbol{\beta}(\alpha^*)$.

\subsection{Inverse Problem Derivations}\label{app:inverse-global-window-derivations}

\paragraph{Inverse PI-SC-I.}
At step $L$, PI-SC-I linearizes the physical residual simultaneously with respect to the generalized state and the physical parameters around $(\mathcal{D}^{(q)}\vu_{L-1},\boldsymbol{\lambda}_{L-1})$:
\begin{equation}\label{eq:inverse_local_taylor}
	\begin{aligned}
		\ve_L^p
		\approx \ve_{L-1}^p
		+\sum_{\boldsymbol{\gamma}\in\sI_q}
		\mJ_{D^{\boldsymbol{\gamma}}\vu}(\vu_{L-1})
		\left(D^{\boldsymbol{\gamma}}\vu_L-D^{\boldsymbol{\gamma}}\vu_{L-1}\right) 
		+\mJ_{\lambda}(\vx;\vu_{L-1},\boldsymbol{\lambda}_{L-1})
		\left(\boldsymbol{\lambda}_L-\boldsymbol{\lambda}_{L-1}\right).
	\end{aligned}
\end{equation}
Using $D^{\boldsymbol{\gamma}}\vu_L-D^{\boldsymbol{\gamma}}\vu_{L-1}
=\boldsymbol{\beta}_L D^{\boldsymbol{\gamma}}h_L(\vx)$ and
$\Delta\boldsymbol{\lambda}_L=\boldsymbol{\lambda}_L-\boldsymbol{\lambda}_{L-1}$ reduces (\ref{eq:inverse_local_taylor}) to the augmented physical system in (\ref{eq:inverse_local_linearized}). The corresponding augmented design matrices are $\mA_{L,aug}(\vx)=[\mA_L(\vx),\mJ_{\lambda}(\vx;\vu_{L-1},\boldsymbol{\lambda}_{L-1})]$, $\mC_{L,aug}=[\mC_L,\mathbf{0}]$, and $\mD_{L,aug}=[\mD_L,\mathbf{0}]$. The state and parameters candidate is then obtained from
\begin{equation}\label{eq:inverse_local_ls}
	\begin{aligned}
		\mathcal{L}(\vtheta_L)
		&= \frac{1}{N_p} \|\mA_{L,aug}(\vx) \vtheta_L + \ve_{L-1}^p\|_{\Omega_p}^2
		+ \frac{\omega_{bc}}{N_b} \|\mC_{L,aug}(\vx) \vtheta_L + \ve_{L-1}^b\|_{\Omega_b}^2 \\
		&\quad + \frac{\omega_{data}}{N_d} \|\mD_{L,aug}(\vx) \vtheta_L + \ve_{L-1}^d\|_{\Omega_d}^2.
	\end{aligned}
\end{equation}

Solving this linear least-squares objective yields the candidate $\widehat{\vtheta}_L$. As in forward PI-SC-I, the candidate is damped as $\vtheta_L^*=\alpha\widehat{\vtheta}_L$ and accepted through the supervisory node selection mechanism. The accepted local configuration is
$\vu_{loc}=\vu_{L-1}+\boldsymbol{\beta}_L^*h_L$ and
$\boldsymbol{\lambda}_{loc}=\boldsymbol{\lambda}_{L-1}+\Delta\boldsymbol{\lambda}_L^*$.

\paragraph{Inverse PI-SC-III.}
Updating $\boldsymbol{\lambda}$ strictly through localized node additions can restrict parameter-identification accuracy. PI-SC-III therefore jointly refines all output weights and physical parameters around the local prediction $(\mathcal{D}^{(q)}\vu_{loc},\boldsymbol{\lambda}_{loc})$. Let $\Delta\boldsymbol{\lambda}_{glob}=\boldsymbol{\lambda}-\boldsymbol{\lambda}_{loc}$ and define the global augmented vector $\vtheta_{glob}=[\boldsymbol{\beta}_{glob}^{\top},\Delta\boldsymbol{\lambda}_{glob}^{\top}]^{\top}$. Using the forward PI-SC-III formulation, the global linearized physical equation becomes
\begin{equation}
	\mA_{glob}(\vx) \boldsymbol{\beta}_{glob}
	+ \mJ_{\lambda}(\vx;\vu_{loc},\boldsymbol{\lambda}_{loc})
	\Delta\boldsymbol{\lambda}_{glob}
	\approx \vb(\vx),
\end{equation}
where $\mA_{glob}$ and $\vb$ are defined in Section~\ref{sec:algorithm-description}, and
$\mJ_{\lambda}(\vx;\vu_{loc},\boldsymbol{\lambda}_{loc})
=\frac{\partial\mathbf{F}}{\partial\boldsymbol{\lambda}}
\big|_{\mathcal{D}^{(q)}\vu=\mathcal{D}^{(q)}\vu_{loc},
\boldsymbol{\lambda}=\boldsymbol{\lambda}_{loc}}$.

Let $\mJ_{\lambda,loc}(\vx)=\mJ_{\lambda}(\vx;\vu_{loc},\boldsymbol{\lambda}_{loc})$. The augmented matrices are
$\mA_{glob,aug}(\vx)=[\mA_{glob}(\vx),\mJ_{\lambda,loc}(\vx)]$,
$\mC_{glob,aug}(\vx)=[\mC_{glob}(\vx),\mathbf{0}]$, and
$\mD_{glob,aug}(\vx)=[\mD_{glob}(\vx),\mathbf{0}]$.
The global joint update is obtained from
\begin{equation}\label{eq:inverse_global_ls}
	\begin{aligned}
		\mathcal{L}_{glob}(\vtheta_{glob})
		&= \frac{1}{N_p} \|\mA_{glob,aug}(\vx)\vtheta_{glob}-\vb(\vx)\|_{\Omega_p}^2
		+ \frac{\omega_{bc}}{N_b} \|\mC_{glob,aug}(\vx)\vtheta_{glob}-\vg_{bc}(\vx)\|_{\Omega_b}^2 \\
		&\quad + \frac{\omega_{data}}{N_d} \|\mD_{glob,aug}(\vx)\vtheta_{glob}-\vu_{data}(\vx)\|_{\Omega_d}^2.
	\end{aligned}
\end{equation}

The corresponding solution is
$\tilde{\vtheta}_{glob}^*
=[(\tilde{\boldsymbol{\beta}}_{glob}^*)^\top,
(\Delta\tilde{\boldsymbol{\lambda}}_{glob}^*)^\top]^\top$.

\paragraph{Inverse PI-SC-II.}
PI-SC-II reduces the cost of full-batch joint inversion by restricting the correction to the most recent $W$ hidden nodes. Define $\Delta\boldsymbol{\lambda}_{win}=\boldsymbol{\lambda}-\boldsymbol{\lambda}_{loc}$ and
$\vtheta_{win}=[\boldsymbol{\beta}_{win}^{\top},\Delta\boldsymbol{\lambda}_{win}^{\top}]^{\top}$.
Using $\mA_W$, $\mC_W$, and $\mD_W$ from (\ref{eq:window_components}), the corresponding augmented matrices are
$\mA_{W,aug}(\vx)=[\mA_W(\vx),\mJ_{\lambda,loc}(\vx)]$,
$\mC_{W,aug}(\vx)=[\mC_W(\vx),\mathbf{0}]$, and
$\mD_{W,aug}(\vx)=[\mD_W(\vx),\mathbf{0}]$.
The sliding-window joint update is therefore formulated as
\begin{equation}\label{eq:inverse_window_ls}
	\begin{aligned}
		\mathcal{L}_{win}(\vtheta_{win})
		&= \frac{1}{N_p} \|\mA_{W,aug}(\vx)\vtheta_{win}-\vr_{L-W}^p(\vx)\|_{\Omega_p}^2
		+ \frac{\omega_{bc}}{N_b} \|\mC_{W,aug}(\vx)\vtheta_{win}-\vr_{L-W}^b(\vx)\|_{\Omega_b}^2 \\
		&\quad + \frac{\omega_{data}}{N_d} \|\mD_{W,aug}(\vx)\vtheta_{win}-\vr_{L-W}^d(\vx)\|_{\Omega_d}^2.
	\end{aligned}
\end{equation}

The corresponding solution is
$\vtheta_{win}^*
=[(\boldsymbol{\beta}_{win}^*)^\top,
(\Delta\tilde{\boldsymbol{\lambda}}_{win}^*)^\top]^\top$.

\paragraph{Discrete line search.}
Inverse PI-SC-II and PI-SC-III apply the discrete line search to the absolute joint configuration $\boldsymbol{\Theta}=[\boldsymbol{\beta}^{\top},\boldsymbol{\lambda}^{\top}]^{\top}$. Let
$\boldsymbol{\Theta}_{loc}^*=[\boldsymbol{\beta}_{loc}^{*\top},\boldsymbol{\lambda}_{loc}^{\top}]^{\top}$.
The candidate absolute configuration is reassembled as
\begin{equation}\label{eq:Theta}
	\tilde{\boldsymbol{\Theta}}^* =
	\begin{cases}
		\begin{bmatrix}
			\tilde{\boldsymbol{\beta}}_{glob}^* \\
			\boldsymbol{\lambda}_{loc}+\Delta\tilde{\boldsymbol{\lambda}}_{glob}^*
		\end{bmatrix},
		& \text{for PI-SC-III}, \\[1.5em]
		\begin{bmatrix}
			\boldsymbol{\beta}_{pre} \\
			\boldsymbol{\beta}_{win}^* \\
			\boldsymbol{\lambda}_{loc}+\Delta\tilde{\boldsymbol{\lambda}}_{win}^*
		\end{bmatrix},
		& \text{for PI-SC-II}.
	\end{cases}
\end{equation}
The search path and optimal step size are
\begin{equation}
	\boldsymbol{\Theta}(\alpha)
	=\boldsymbol{\Theta}_{loc}^*
	+\alpha(\tilde{\boldsymbol{\Theta}}^*-\boldsymbol{\Theta}_{loc}^*),
\end{equation}
and
\begin{equation}
	\alpha^*
	=\mathop{\arg\min}_{\alpha\in\sA}\mathcal{E}_L^2(\alpha),
\end{equation}
respectively. The final joint update is
$\boldsymbol{\Theta}^*=\boldsymbol{\Theta}(\alpha^*)$.

\subsection{Detailed PI-SCM Training Algorithms}\label{app:training-algorithms}

The complete pseudocode for the three PI-SCM algorithms is provided below. Algorithm~\ref{alg:PI-SC-I} gives the localized construction used by all variants. Algorithms~\ref{alg:PI-SC-II} and~\ref{alg:PI-SC-III} augment this construction with sliding-window and global corrections, respectively.

\begin{algorithm}[!htbp]
	\caption{Training procedure for PI-SC-I}
	\label{alg:PI-SC-I}
	\begin{algorithmic}[1]
		\renewcommand{\algorithmicrequire}{\textbf{Input:}}
		\renewcommand{\algorithmicensure}{\textbf{Output:}}
		\REQUIRE Given coordinates $\vx = (\evx_1, \dots, \evx_n)^\top$ partitioned into collocation sets $\Omega_p, \Omega_b, \Omega_d$; Sparse observational target $\vu_{data}$; The differential operator $\mathbf{F}$ and initial/boundary operator $\mathcal{B}$; Initialized $L_{max}, T_{max}$; Error tolerance $\epsilon$; Three sets of scalars $\Upsilon = \{\tau_{1}, \dots, \tau_{end}\},\sS_r = \{r_{1}, \dots, r_{end}\},\sA_1 = \{\alpha_{1,1}:\Delta \alpha_1:\alpha_{1,end}\}$; Penalty weights $\omega_{bc}, \omega_{data}$.
		\ENSURE Hidden parameters $\{\vw_L^*, b_L^*\}_{L=1}^{L_{end}}$, output weights $\{\boldsymbol{\beta}_L^*\}_{L=1}^{L_{end}}$ (and identified physical parameter $\boldsymbol{\lambda}^*$ for inverse tasks).

		\STATE Initialize $L=1, \vu_0 = \vzero$. If inverse problem, initialize $\boldsymbol{\lambda}_0$.
		\STATE Evaluate initial residuals $\ve_0^p, \ve_0^b, \ve_0^d$ and comprehensive norm $\mathcal{E}_0^2$.

		\WHILE{$L \leq L_{max}$ \textbf{and} $\mathcal{E}^2_{L-1} > \epsilon$}
		\STATE Initialize two empty sets $\sP$ and $\Omega_1$.\label{algorithm:begin PI-SC-I}
		\FOR{$\alpha_1 \in \sA_1$}
		\FOR{$\tau \in \Upsilon$}\label{algorithm:beginfortau}
		\FOR{$k = 1, 2, \dots, T_{max}$}
		\STATE Sample $\rvw_L \sim U(-\tau, \tau)$ and $\rb_L \sim U(-\tau, \tau)$, and denote their realizations by $\vw_L$ and $b_L$, respectively.

		\STATE For forward tasks, assemble $\mM_L, \vv_L$ by (\ref{eq:weighted_system}) and compute $\boldsymbol{\beta}_L^*$ by (\ref{eq:analytical_solution}). For inverse tasks, solve (\ref{eq:inverse_local_ls}) to obtain $\widehat{\vtheta}_L$, and set $\vtheta_L^*=\alpha_1\widehat{\vtheta}_L$.
		\STATE Extract $\boldsymbol{\beta}_L^*$ (and $\Delta \boldsymbol{\lambda}^*$ for inverse tasks).

		\STATE Update $\vu_L = \vu_{L-1} + \boldsymbol{\beta}_L^* h_L$.
		\STATE Evaluate $\mathcal{E}_L^2$.
		\STATE Calculate $\xi_L = (r + \mu_L)\mathcal{E}_{L-1}^2 - \mathcal{E}_L^2$, where $\mu_L = (1-r)/(L+1)$.

		\IF{$\xi_L > 0$}
		\STATE Save $\vw_L, b_L, \boldsymbol{\beta}_L^*$ (and $\Delta\boldsymbol{\lambda}^*$ for inverse tasks) in $\sP$, $\xi_L$ in $\Omega_1$.
		\ENDIF
		\ENDFOR
		\IF{$\sP$ is not empty}
		\STATE Find $\vw_L, b_L, \boldsymbol{\beta}_L^*$ (and $\Delta\boldsymbol{\lambda}^*$ for inverse tasks) that maximize $\xi_L$ in $\Omega_1$.
		\STATE \textbf{Break} (go to Step \ref{algorithm:endfor}).
		\ELSE
		\STATE Update $r$ to the subsequent value in $\sS_r$.
		\STATE \textbf{Continue}
		\ENDIF
		\ENDFOR
		\ENDFOR\label{algorithm:endpisci}
		\STATE\label{algorithm:endfor} Update $\vu_L$, $\ve_L^p, \ve_L^b, \ve_L^d$, $\mathcal{E}_L^2$ (and set $\boldsymbol\lambda_L^*=\boldsymbol\lambda_{L-1}+\Delta\boldsymbol\lambda_L^*$).
		\STATE $L \leftarrow L + 1$.
		\ENDWHILE

		\RETURN $\{\vw_L^*, b_L^*\}_{L=1}^{L_{end}}$, $\{\boldsymbol{\beta}_L^*\}_{L=1}^{L_{end}}$ (and $\boldsymbol{\lambda}^*$ for inverse tasks).
	\end{algorithmic}
\end{algorithm}

\begin{algorithm}[!htbp]
	\caption{Training procedure for PI-SC-II}
	\label{alg:PI-SC-II}
	\begin{algorithmic}[1]
		\renewcommand{\algorithmicrequire}{\textbf{Input:}}
		\renewcommand{\algorithmicensure}{\textbf{Output:}}
		\REQUIRE The inputs of Algorithm \ref{alg:PI-SC-I}; Window size $W < L_{max}$; correction interval $L_g < L_{max}$; $\sA_2 = \{\alpha_{2,1}:\Delta \alpha_2:\alpha_{2,end}\}$.
		\ENSURE Hidden parameters $\{\vw_L^*, b_L^*\}_{L=1}^{L_{end}}$, output weights $\{\boldsymbol{\beta}_L^*\}_{L=1}^{L_{end}}$ (and identified parameter $\boldsymbol{\lambda}^*$ for inverse tasks).

		\STATE Initialize $L=1, \vu_0 = \vzero$. If inverse problem, initialize $\boldsymbol{\lambda}_0$.
		\STATE Evaluate initial residuals $\ve_0^p, \ve_0^b, \ve_0^d$ and comprehensive norm $\mathcal{E}_0^2$.

		\WHILE{$L \leq L_{max}$ \textbf{and} $\mathcal{E}^2_{L-1} > \epsilon$}

		\STATE Execute Step \ref{algorithm:begin PI-SC-I}-\ref{algorithm:endpisci} of Algorithm \ref{alg:PI-SC-I} to secure $\vw_L^*, b_L^*$ and $\boldsymbol{\beta}_{loc}^*$ (or $\boldsymbol{\Theta}_{loc}^*$ for inverse tasks).
		\IF{$L \% L_g=0$}
		\IF{$L \le W$}
		\STATE For forward tasks, assemble $\mM_{glob}, \vv_{glob}$ by (\ref{eq:global_matrices}) and compute $\tilde{\boldsymbol{\beta}}^*$ by (\ref{eq:global_analytical_solution}). For inverse tasks, solve (\ref{eq:inverse_global_ls}) to obtain $\tilde{\vtheta}_{glob}^*$.
		\ELSE
		\STATE Retrieve $\boldsymbol{\beta}_{pre} = [\boldsymbol\beta_1^{*\top},\ldots,
		\boldsymbol\beta_{L-W}^{*\top}]^\top$.
		\STATE Calculate $\vr_{L-W}^p, \vr_{L-W}^b, \vr_{L-W}^d$ by (\ref{eq:window_components}).
		\STATE For forward tasks, assemble $\mM_W, \vv_W$ by (\ref{eq:window_matrices}) and compute $\boldsymbol{\beta}_{win}^*$ by (\ref{eq:window_analytical_solution}).
		\STATE For inverse tasks, solve (\ref{eq:inverse_window_ls}) to obtain $\vtheta_{win}^*$.
		\STATE For forward tasks, let $\tilde{\boldsymbol{\beta}}^* = [\boldsymbol{\beta}_{pre}^\top, (\boldsymbol{\beta}_{win}^*)^\top]^\top$.
		\ENDIF
		
		\STATE Initialize an empty set $\Omega_2$.
		\STATE Obtain $\Delta \boldsymbol{\beta} = \tilde{\boldsymbol{\beta}}^* - \boldsymbol{\beta}_{loc}^*$ (or obtain $\tilde{\boldsymbol{\Theta}}^*$ by (\ref{eq:Theta}), $\Delta \boldsymbol{\Theta} = \tilde{\boldsymbol{\Theta}}^* - \boldsymbol{\Theta}_{loc}^*$ for inverse tasks).
		\FOR{$\alpha_2 \in \sA_2$}
		\STATE Calculate $\boldsymbol{\beta}(\alpha_2) = \boldsymbol{\beta}_{loc}^* + \alpha_2 \Delta \boldsymbol{\beta}$ (or $\boldsymbol{\Theta}(\alpha_2)$ for inverse tasks) and corresponding $\mathcal{E}_L^2(\alpha_2)$. Save $\mathcal{E}_L^2(\alpha_2)$ to $\Omega_2$.
		\ENDFOR
		\STATE Find $\alpha^*_2$ that minimizes $\mathcal{E}_L^2(\alpha_2)$ in $\Omega_2$ and adopt $\boldsymbol{\beta}^* = \boldsymbol{\beta}(\alpha^*_2)$ (or $\boldsymbol{\Theta}^* = \boldsymbol{\Theta}(\alpha^*_2)$ for inverse tasks).
		\ENDIF
		\STATE Extract $\boldsymbol{\beta}^*$ (and $\boldsymbol{\lambda}^*$ for inverse tasks).
		\STATE Update $\vu_L$, $\ve_L^p, \ve_L^b, \ve_L^d$, $\mathcal{E}_L^2$ (and set $\boldsymbol\lambda_L^*=\boldsymbol{\lambda}^*$).
		\STATE $L \leftarrow L + 1$.

		\ENDWHILE

		\RETURN $\{\vw_L^*, b_L^*\}_{L=1}^{L_{end}}$, $\{\boldsymbol{\beta}_L^*\}_{L=1}^{L_{end}}$ (and $\boldsymbol{\lambda}^*$ for inverse tasks).
	\end{algorithmic}
\end{algorithm}

\begin{algorithm}[!htbp]
	\caption{Training procedure for PI-SC-III}
	\label{alg:PI-SC-III}
	\begin{algorithmic}[1]
		\renewcommand{\algorithmicrequire}{\textbf{Input:}}
		\renewcommand{\algorithmicensure}{\textbf{Output:}}
		\REQUIRE The inputs of Algorithm \ref{alg:PI-SC-I}; correction interval $L_g < L_{max}$; $\sA_2 = \{\alpha_{2,min}:\Delta \alpha_2:\alpha_{2,max}\}$.
		\ENSURE Hidden parameters $\{\vw_L^*, b_L^*\}_{L=1}^{L_{end}}$, output weights $\{\boldsymbol{\beta}_L^*\}_{L=1}^{L_{end}}$ (and identified parameter $\boldsymbol{\lambda}^*$ for inverse tasks).

		\STATE Initialize $L=1, \vu_0 = \vzero$. If inverse problem, initialize $\boldsymbol{\lambda}_0$.
		\STATE Evaluate initial residuals $\ve_0^p, \ve_0^b, \ve_0^d$ and comprehensive norm $\mathcal{E}_0^2$.

		\WHILE{$L \leq L_{max}$ \textbf{and} $\mathcal{E}^2_{L-1} > \epsilon$}

		\STATE Execute Step \ref{algorithm:begin PI-SC-I}-\ref{algorithm:endpisci} of Algorithm \ref{alg:PI-SC-I} to secure $\vw_L^*, b_L^*$ and $\boldsymbol{\beta}_{loc}^*$ (or $\boldsymbol{\Theta}_{loc}^*$ for inverse tasks).
		\IF{$L \% L_g=0$}
		\STATE For forward tasks, assemble $\mM_{glob}, \vv_{glob}$ by (\ref{eq:global_matrices}) and compute $\tilde{\boldsymbol{\beta}}_{glob}^*$ by (\ref{eq:global_analytical_solution}). For inverse tasks, solve (\ref{eq:inverse_global_ls}) to obtain $\tilde{\vtheta}_{glob}^*$.

		\STATE Initialize an empty set $\Omega_2$.
		\STATE Obtain $\Delta \boldsymbol{\beta} = \tilde{\boldsymbol{\beta}}_{glob}^* - \boldsymbol{\beta}_{loc}^*$ (or obtain $\tilde{\boldsymbol{\Theta}}^*$ by (\ref{eq:Theta}), $\Delta \boldsymbol{\Theta} = \tilde{\boldsymbol{\Theta}}^* - \boldsymbol{\Theta}_{loc}^*$ for inverse tasks).
		\FOR{$\alpha_2 \in \sA_2$}
		\STATE Calculate $\boldsymbol{\beta}(\alpha_2) = \boldsymbol{\beta}_{loc}^* + \alpha_2 \Delta \boldsymbol{\beta}$ (or $\boldsymbol{\Theta}(\alpha_2)$ for inverse tasks) and corresponding $\mathcal{E}_L^2(\alpha_2)$. Save $\mathcal{E}_L^2(\alpha_2)$ to $\Omega_2$.
		\ENDFOR
		\STATE Find $\alpha^*_2$ that minimizes $\mathcal{E}_L^2(\alpha_2)$ in $\Omega_2$ and adopt $\boldsymbol{\beta}^* = \boldsymbol{\beta}(\alpha^*_2)$ (or $\boldsymbol{\Theta}^* = \boldsymbol{\Theta}(\alpha^*_2)$ for inverse tasks).
		\ENDIF
		\STATE Extract $\boldsymbol{\beta}^*$ (and $\boldsymbol{\lambda}^*$ for inverse tasks).
		\STATE Update $\vu_L$, $\ve_L^p, \ve_L^b, \ve_L^d$, $\mathcal{E}_L^2$ (and set $\boldsymbol\lambda_L^*=\boldsymbol{\lambda}^*$).
		\STATE $L \leftarrow L + 1$.

		\ENDWHILE

		\RETURN $\{\vw_L^*, b_L^*\}_{L=1}^{L_{end}}$, $\{\boldsymbol{\beta}_L^*\}_{L=1}^{L_{end}}$ (and $\boldsymbol{\lambda}^*$ for inverse tasks).
	\end{algorithmic}
\end{algorithm}

\section{Proofs of the Universal Approximation Property}\label{app:universal-approximation}

\subsection{Proofs of the Universal Approximation Theorems}\label{app:universal-approximation-proofs}

This section provides the detailed arguments for the theorems stated in Section~\ref{sec:universal-approximation}. The proofs below focus on forward problems. For inverse problems, the joint state and parameter increment $\vtheta_L$ defined in Section~\ref{sec:inverse problem} yields an augmented linearized system of the same form, allowing the convergence argument to carry over when the corresponding conditions hold.

As the convergence of the PI-SCM framework relies on the stability of incremental node additions, we initially prove the universal approximation property for the PI-SC-I algorithm, which serves as the convergent baseline for the subsequent variants. We recall $\tilde{\ve}_L$, $\mathcal{E}_L^2$, $\mM_L$, and $\vv_L=-\tilde{\ve}_{L-1}$ established in Section~\ref{sec:algorithm-description}.

By applying the first-order Taylor expansion to $\tilde{\ve}_L$, we have $\tilde{\ve}_L = \tilde{\ve}_{L-1} + \mM_L \boldsymbol{\beta}_L + \vr(\boldsymbol{\beta}_L)=\mM_L\boldsymbol{\beta}_L-\vv_L+\vr(\boldsymbol{\beta}_L)$. The Taylor truncation error $\vr(\boldsymbol{\beta}_L)$ originates exclusively from the nonlinear physical residual. The true nonlinear residual norm at the $L$-th step is thus given by:
\begin{equation}\label{eq:true_residual_expansion}
	\mathcal{E}_L^2
	=
	\|\mM_L\boldsymbol{\beta}_L-\vv_L\|^2
	+\mathcal{R}_{Taylor}(\boldsymbol{\beta}_L),
\end{equation}
where $\mathcal{R}_{Taylor}(\boldsymbol{\beta}_L)$ is defined as:
\begin{equation}\label{eq:remainder_definition}
	\mathcal{R}_{Taylor}(\boldsymbol{\beta}_L) = 2(\mM_L\boldsymbol{\beta}_L-\vv_L)^\top \vr(\boldsymbol{\beta}_L) + \|\vr(\boldsymbol{\beta}_L)\|^2.
\end{equation}

As shown in (\ref{eq:remainder_definition}), $\mathcal{R}_{Taylor}(\boldsymbol{\beta}_L)$ represents the omitted second-order Taylor remainder. Under the assumptions stated in Section~\ref{sec:universal-approximation}, we establish the following lemma.

\begin{lemma}
	\label{lemma:uniform_residual_capture}
	Suppose Assumptions~\ref{assumption:Lipschitz_continuous} and~\ref{assumption:residual_completeness} hold. There exist constants $\kappa\in(0,1]$ and $D>0$ such that whenever $\vv_L\neq\vzero$, there exists a candidate node $h_L\in\Gamma$ that satisfies
	\begin{equation}\label{eq:uniform_residual_capture}
		\|\mP_L(h_L)\vv_L\|^2 \geq \kappa\|\vv_L\|^2,
	\end{equation}
	and
	\begin{equation}\label{eq:directional_stability}
		\|\mM_L(h_L)^\dagger\vv_L\|^2 \leq D\|\mP_L(h_L)\vv_L\|^2,
	\end{equation}
	where
	\begin{equation}\label{eq:candidate_projector}
		\mP_L(h) = \mM_L(h)\mM_L(h)^\dagger.
	\end{equation}
\end{lemma}

\begin{proof}
	Since the collocation state has fixed finite dimension, the uniform boundedness in Assumption~\ref{assumption:Lipschitz_continuous} makes $\sK$ compact. The operator regularity also implies that $\mM(\vz_{\mathrm{col}},h_j)$ is continuous in $\vz_{\mathrm{col}}$ for each fixed reference node.
	
	The full row rank condition in Assumption~\ref{assumption:residual_completeness}, together with continuity and compactness, yields
	\begin{align}
		a_0 &:= \min_{\vz_{\mathrm{col}}\in\sK}\lambda_{\min}\left(\mM_{\mathrm{ref}}(\vz_{\mathrm{col}})\mM_{\mathrm{ref}}(\vz_{\mathrm{col}})^\top\right)>0,\label{eq:reference_frame_lower_bound}\\
		B_0 &:= \max_{\vz_{\mathrm{col}}\in\sK,\,1\leq j\leq J}\|\mM(\vz_{\mathrm{col}},h_j)\|_2^2<\infty.\label{eq:reference_design_upper_bound}
	\end{align}
	
	For each nonzero reference block, constant rank ensures that its smallest positive singular value is continuous and positive on $\sK$, and therefore has a positive minimum. Zero-rank blocks have zero pseudoinverses. Consequently,
	\begin{equation}\label{eq:reference_pseudoinverse_bound}
		D := \max_{\vz_{\mathrm{col}}\in\sK,\,1\leq j\leq J}\|\mM(\vz_{\mathrm{col}},h_j)^\dagger\|_2^2<\infty.
	\end{equation}
	
	Fix a reachable state at step $L$ with $\vv_L\neq\vzero$. By the definition of $a_0$,
	\begin{equation}\label{eq:reference_correlation_sum}
		\sum_{j=1}^{J}\|\mM_L(h_j)^\top\vv_L\|^2 \geq a_0\|\vv_L\|^2.
	\end{equation}
	
	Hence there exists an index $j_*$ such that
	\begin{equation}\label{eq:reference_candidate_correlation}
		\|\mM_L(h_{j_*})^\top\vv_L\|^2 \geq \frac{a_0}{J}\|\vv_L\|^2.
	\end{equation}
	
	Choose $h_L=h_{j_*}$. Since $\mP_L(h_L)$ is the orthogonal projector onto the column space of $\mM_L(h_L)$, we have
	\begin{equation}\label{eq:uniform_capture_proof}
			\frac{a_0}{J}\|\vv_L\|^2 \leq \|\mM_L(h_L)^\top\vv_L\|^2=\|\mM_L(h_L)^\top\mP_L(h_L)\vv_L\|^2\leq B_0\|\mP_L(h_L)\vv_L\|^2.
	\end{equation}
	
	Thus the first inequality holds with $\kappa=a_0/(JB_0)$. Since $\sum_{j=1}^{J}\|\mM_L(h_j)^\top\vv_L\|^2\leq JB_0\|\vv_L\|^2$, Eq.~(\ref{eq:reference_correlation_sum}) gives $a_0\leq JB_0$ and hence $\kappa\in(0,1]$.
	
	For the same candidate node, $\mM_L(h_L)^\dagger\mP_L(h_L)=\mM_L(h_L)^\dagger$ gives
	\begin{equation}\label{eq:directional_stability_proof}
			\|\mM_L(h_L)^\dagger\vv_L\|^2 = \|\mM_L(h_L)^\dagger\mP_L(h_L)\vv_L\|^2 \leq D\|\mP_L(h_L)\vv_L\|^2.
	\end{equation}
	
	This completes the proof.
\end{proof}

We now prove Theorem~\ref{theorem:PI-SCM-I}.

\begin{proof}
	If $\vv_L=\vzero$, residual convergence has already been achieved. Otherwise, by Lemma~\ref{lemma:uniform_residual_capture}, choose a node $h_L$ satisfying Eqs.~(\ref{eq:uniform_residual_capture}) and~(\ref{eq:directional_stability}).
	
	Recalling $\mP_L(h_L)$ from Eq.~\ref{eq:candidate_projector}, it satisfies $\mP_L(h_L)^\top=\mP_L(h_L)$ and $\mP_L(h_L)^2=\mP_L(h_L)$. Hence, $\mM_L(h_L)\boldsymbol{\beta}_L(\alpha)-\vv_L=(\alpha\mP_L(h_L)-\mI)\vv_L$. Substituting the relation into Eq.~(\ref{eq:true_residual_expansion}), we define $\Delta \mathcal{E}^2(\alpha)$ as the reduction in the squared norm of the true residual:
	\begin{equation}\label{eq:proof_energy_descent}
			\Delta\mathcal{E}^2(\alpha)
			=\mathcal{E}_{L-1}^2-\mathcal{E}_L^2=\alpha(2-\alpha)\|\mP_L(h_L)\vv_L\|^2
			-\mathcal{R}_{Taylor}\!\left(\alpha\mM_L(h_L)^\dagger\vv_L\right).
	\end{equation}
	
	The values and derivatives up to order $q$ of $h_L$ are uniformly bounded. Eq.~(\ref{eq:directional_stability}) and the boundedness of the generalized residuals ensure that $\mM_L(h_L)^\dagger\vv_L$ is uniformly bounded. By compactness of $\sK$, there exists $\bar{\alpha}\in(0,1]$ such that all trial segments with $0\leq\alpha\leq\bar{\alpha}$ lie in a common compact subset of the $C^2$ neighborhood in Assumption~\ref{assumption:Lipschitz_continuous}. Taylor's theorem and Eq.~(\ref{eq:remainder_definition}), together with the preceding bounds, therefore yield a uniform constant $\bar{C}>0$ such that $|\mathcal{R}_{Taylor}(\alpha\mM_L(h_L)^\dagger\vv_L)|\leq\alpha^2\bar{C}\|\mM_L(h_L)^\dagger\vv_L\|^2$ for $0\leq\alpha\leq\bar{\alpha}$. Applying this bound to Eq.~(\ref{eq:proof_energy_descent}) gives
	\begin{equation}\label{eq:proof_energy_descent_lower_bound}
		\Delta\mathcal{E}^2(\alpha)\geq\alpha(2-\alpha)\|\mP_L(h_L)\vv_L\|^2-\alpha^2\bar{C}\|\mM_L(h_L)^\dagger\vv_L\|^2.
	\end{equation}
	
	Let $D$ be the constant in Lemma~\ref{lemma:uniform_residual_capture}. We reduce $\bar{\alpha}$ so that $\bar{\alpha}\leq(1+\bar{C}D)^{-1}$. Using Eqs.~(\ref{eq:directional_stability}) and~(\ref{eq:uniform_residual_capture}), together with $\|\vv_L\|^2=\mathcal{E}_{L-1}^2$, for any $\alpha\in((1-r)/\kappa,\bar{\alpha}]$, we obtain
	\begin{equation}\label{eq:proof_uniform_descent}
		\begin{aligned}
			\Delta\mathcal{E}^2(\alpha)
			&\geq\alpha[2-\alpha(1+\bar{C}D)]\|\mP_L(h_L)\vv_L\|^2\geq\alpha\|\mP_L(h_L)\vv_L\|^2\\
			&\geq\alpha\kappa\mathcal{E}_{L-1}^2>(1-r)\mathcal{E}_{L-1}^2\geq\delta_L.
		\end{aligned}
	\end{equation}

	Consequently, we obtain
	\begin{equation}\label{eq:contraction_mapping}
		\mathcal{E}_L^2
		\leq\mathcal{E}_{L-1}^2-\delta_L
		=(r+\mu_L)\mathcal{E}_{L-1}^2.
	\end{equation}
	
	Note that $\lim_{L \to +\infty} \mu_L = 0$ and $r<1$. By utilizing (\ref{eq:contraction_mapping}), it follows that $\lim_{L \to +\infty} \mathcal{E}_L^2 = 0$.
\end{proof}

We next prove Theorem~\ref{theorem:PI-SCM-III}. Recalling the optimal localized configuration $\boldsymbol{\beta}_{loc}^*$ and the global configuration $\tilde{\boldsymbol{\beta}}_{glob}^*$ established in Section \ref{sec:algorithm-description}, let $\mathcal{E}_L^2(\boldsymbol{\beta}) = \|\tilde{\ve}_L(\boldsymbol{\beta})\|^2$ denote the comprehensive true residual norm evaluated at any candidate configuration $\boldsymbol{\beta}$. Consequently, the local residual norm is explicitly defined as $\mathcal{E}_{L, loc}^2 = \mathcal{E}_L^2(\boldsymbol{\beta}_{loc}^*)$.

\begin{proof}
	If $\mM_{glob}\tilde{\boldsymbol{\beta}}_{glob}^* \neq \mM_{glob}\boldsymbol{\beta}_{loc}^*$, $\Delta \boldsymbol{\beta}$ constitutes a strict descent direction for the linearized least-squares objective. Analogous to the proof of Theorem \ref{theorem:PI-SCM-I}, the first-order linear descent dominates the bounded second-order remainder for sufficiently small $\alpha > 0$. This guarantees the existence of a positive upper bound $\bar{\alpha} \in (0, 1]$ such that for all $\alpha \in (0, \bar{\alpha}]$, we have $\mathcal{E}_L^2(\boldsymbol{\beta}) < \mathcal{E}_{L, loc}^2$. 
	
	If $\mM_{glob}\tilde{\boldsymbol{\beta}}_{glob}^* = \mM_{glob}\boldsymbol{\beta}_{loc}^*$, taking $\bar{\alpha}=0$ satisfies the acceptance criterion. Since $\alpha=0$ inherently satisfies the equality $\mathcal{E}_L^2(\boldsymbol{\beta}) = \mathcal{E}_{L, loc}^2$, the inequality described in (\ref{eq:PI-SCM-IIIinequality}) holds for all $\alpha \in [0, \bar{\alpha}]$.
	
	By (\ref{eq:contraction_mapping}), $\mathcal{E}_{L, loc}^2 \leq (r + \mu_L)\mathcal{E}_{L-1}^2$. Thus, we have:
	\begin{equation}
		\mathcal{E}_L^2 \leq \mathcal{E}_{L, loc}^2 \leq (r + \mu_L)\mathcal{E}_{L-1}^2.
	\end{equation}
	
	Using the arguments established in the proof of Theorem \ref{theorem:PI-SCM-I}, we conclude that $\lim_{L \to +\infty} \mathcal{E}_L^2 = 0$. 
\end{proof}

Since PI-SC-II algorithm employs an identical discrete line search mechanism as PI-SC-III, for brevity, we omit the proof of Theorem~\ref{theorem:PI-SCM-II}.

\begingroup
\subsection{Sufficient Conditions for Assumption 2}
\label{app:residual-completeness}
\subsubsection{Collocation Features and Sufficient Conditions}

Let $\mathcal X=\Omega_p\cup\Omega_b\cup\Omega_d$ denote the set of distinct coordinates used to evaluate the residuals. We define
\begin{equation}
\label{eq:rc-node-features}
\boldsymbol{\phi}(h) =\underset{(\vx,\boldsymbol{\gamma})\in\mathcal X\times\sI_q} {\operatorname{col}}\bigl(D^{\boldsymbol{\gamma}}h(\vx)\bigr).
\end{equation}

For a node increment $\boldsymbol{\beta}h$, the state and derivative increments in this ordering are $(\boldsymbol{\phi}(h)\otimes\mI_m)\boldsymbol{\beta}$. Let $\mJ_{\mathrm{col}}(\vz_{\mathrm{col}})$ denote the linear map from these collocation increments to the weighted residual increments in Eq.~(\ref{eq:weighted_system}). Its physical rows are determined by the Jacobians in Eq.~(\ref{eq:high_order_AL}). Thus, $\mM(\vz_{\mathrm{col}},h) =\mJ_{\mathrm{col}}(\vz_{\mathrm{col}}) \bigl(\boldsymbol{\phi}(h)\otimes\mI_m\bigr)$. For $\mV=[\boldsymbol{\phi}(h_1),\ldots,\boldsymbol{\phi}(h_J)]$, we obtain
\begin{equation}
\label{eq:rc-reference-factorization}
\mM_{\mathrm{ref}}(\vz_{\mathrm{col}}) =\mJ_{\mathrm{col}}(\vz_{\mathrm{col}})(\mV\otimes\mI_m).
\end{equation}

The following lemma and theorem provide sufficient conditions for Assumption~\ref{assumption:residual_completeness}.

\begin{lemma}
\label{lem:rc-feature-spanning}
Let $\Gamma=\{\,\vx\mapsto g(\vw^\top\vx+b):(\vw,b)\in\R^{n+1}\,\}$ with $g\in C^q(\R)$. Assume that $\operatorname{span}\Gamma$ is dense in $C^q(\R^n)$. Then the vectors $\boldsymbol{\phi}(h)$ in Eq.~(\ref{eq:rc-node-features}), with $h\in\Gamma$, span the collocation feature space.
\end{lemma}

\begin{theorem}
\label{thm:rc-sufficient-conditions}
Suppose the density hypothesis of Lemma~\ref{lem:rc-feature-spanning} holds and the following conditions are satisfied.
\begin{enumerate}
    \item The matrix $\mJ_{\mathrm{col}}(\vz_{\mathrm{col}})$ has full row rank for every $\vz_{\mathrm{col}}\in\sK$.
    \item There is a fixed selection of rows of $\mM(\vz_{\mathrm{col}},h)$ whose submatrix $\mS(h)$ is independent of $\vz_{\mathrm{col}}$. The parameters $(\vw,b)$ for which $\mS(h)$ has full column rank $m$ form a dense subset of $\R^{n+1}$.
\end{enumerate}

Then Assumption~\ref{assumption:residual_completeness} holds for the fixed collocation system. In particular, the reference nodes can be chosen so that every single-node block has constant rank $m$ over $\sK$.
\end{theorem}

\begin{proof}
By Lemma~\ref{lem:rc-feature-spanning}, choose finitely many nodes whose matrix $\mV$ has full row rank. The feature vectors depend continuously on $(\vw,b)$ because $g\in C^q(\R)$, so sufficiently small perturbations of these parameters preserve full row rank of $\mV$. By the second condition, choose such perturbations so that every $\mS(h_j)$ has full column rank. Since $\mV\otimes\mI_m$ also has full row rank, Eq.~(\ref{eq:rc-reference-factorization}) implies, at every state in $\sK$, we have $\operatorname{rank}\mM_{\mathrm{ref}}(\vz_{\mathrm{col}}) =\operatorname{rank}\mJ_{\mathrm{col}}(\vz_{\mathrm{col}})$.

The first condition therefore implies full row rank of the concatenated matrix. Each selected $\mM(\vz_{\mathrm{col}},h_j)$ contains the submatrix $\mS(h_j)$ of rank $m$ and has exactly $m$ columns. Its rank is consequently $m$ for every $\vz_{\mathrm{col}}\in\sK$.
\end{proof}

\subsubsection{Example: Two-dimensional Bratu}

Consider Eqs.~(\ref{eq:bratu_original}) and~(\ref{eq:bratu_bc}) with the prescribed forward parameter $\lambda=1$. The physical row associated with $\vx_i^p\in\Omega_p$, before multiplication by $1/\sqrt{N_p}$, is
\begin{equation}
\label{eq:rc-bratu-physics-row}
h_{xx}(\vx_i^p)+h_{yy}(\vx_i^p) +\lambda e^{u(\vx_i^p)}h(\vx_i^p).
\end{equation}

The unweighted Dirichlet row is $h(\vx_i^b)$, and each observational row, when included, is $-h(\vx_i^d)$. Assume that the physical collocation coordinates are distinct and that each coordinate carries at most one prescribed value among the Dirichlet and observational rows. To isolate a physical row, set the $h_{xx}(\vx_i^p)$ coordinate in Eq.~(\ref{eq:rc-node-features}) equal to $1$ and all other feature coordinates equal to $0$. This changes only the chosen physical row, since $h(\vx_i^p)=h_{yy}(\vx_i^p)=0$ in this assignment. To isolate a Dirichlet or observational row, first set the corresponding value coordinate equal to $1$. If that coordinate also belongs to $\Omega_p$, set its $h_{xx}$ coordinate equal to $-\lambda e^{u(\vx_i^p)}$ to cancel its physical contribution. All other coordinates can be set to zero. Each residual coordinate can therefore be varied independently for every $\vz_{\mathrm{col}}\in\sK$.

Choose a Dirichlet collocation point $\vx_i^b$. Its row gives the state-independent scalar block $\mS(h)=\sqrt{\omega_{bc}/N_b}\,h(\vx_i^b)$. If the zero set of $g$ has empty interior, then, for each fixed $\vw$, every bias interval contains a $b$ for which $h(\vx_i^b)\neq0$. Thus the parameters giving rank $1$ are dense, and Theorem~\ref{thm:rc-sufficient-conditions} applies in both the unsupervised and semi-supervised settings under the stated density and collocation conditions.

\endgroup
\section{Supplementary Details on Experiments}\label{app:supplementary-experiments}

This appendix summarizes the common implementation and evaluation details used across all experiments, followed by the governing equations and problem-specific configurations of the benchmark systems, accuracy--efficiency study, study on inverse problems with noisy observations, activation-function study, and ablation studies.

\noindent\textbf{Baseline Selection.} We evaluate PI-SCM alongside PINNs \citep{raissi2019physics} and PIELM \citep{dwivedi2020physics} as representative frameworks for physics-informed learning, spanning gradient optimization and randomized networks with analytical output weight estimation. PIELM is used for the linear Helmholtz equation, while NNBM \citep{huang2025neural} extends the randomized network comparison to the nonlinear systems. For Allen--Cahn, reaction--diffusion, and nonlinear diffusion, we include Frozen-PINN \citep{datar2026fast}. We additionally include VS-PINN \citep{ko2025vs} for Allen--Cahn and PISCN \citep{xu2025physics} for semi-supervised Van der Pol. The inverse experiments retain PINNs as baselines for joint state and parameter estimation.

\noindent\textbf{Hardware Details.} All methods are implemented in Python and evaluated on a workstation equipped with an NVIDIA RTX 4060 Ti GPU, an Intel Core i5-14600KF CPU, and 16~GB of RAM.  PI-SCM, PIELM, NNBM, and Frozen-PINN were run on the CPU, whereas PINN, VS-PINN, and PISCN were run on the GPU.

\noindent\textbf{Evaluation Metric.} For each system, state accuracy is measured by the root mean square error (RMSE) over $N_{\mathrm{test}}$ test inputs:
\begin{equation}\label{eq:rmse}
	\mathrm{RMSE}
	=
	\sqrt{\frac{1}{N_{\mathrm{test}}}
		\sum_{i=1}^{N_{\mathrm{test}}}
		\left\|
		\vu_{\mathrm{pred}}(\boldsymbol{\zeta}_i)
		-
		\vu_{\mathrm{true}}(\boldsymbol{\zeta}_i)
		\right\|_2^2},
\end{equation}
where $\boldsymbol{\zeta}_i$ denotes the $i$th test input, and $\vu_{\mathrm{pred}}$ and $\vu_{\mathrm{true}}$ denote the predicted and reference states, respectively.

\subsection{Van der Pol Oscillator}\label{app:supplementary-vdp}

The Van der Pol oscillator \citep{hafeez2015analytical} is a nonlinear, non-conservative dynamical system used in circuit modeling, seismology, biological neuron modeling, and optimal control \citep{andersson2012dynamic}. We consider the controlled second-order equation
\begin{equation}\label{eq:vdp_second_order}
	u_{tt}-\mu(1-u^2)u_t+u-U=0, \qquad (t,U)\in[0,5]\times[0,1],
\end{equation}
where $\mu$ is the nonlinear damping coefficient and $U$ is the exogenous input. The initial conditions are $u(0,U)=-0.25$ and $u_t(0,U)=-2.0$. The value $\mu=1$ is prescribed in the forward task, whereas $\mu$ is identified in the inverse task.

\subsubsection{Forward Problems: Experimental Setup}\label{subsubsec:vdp_forward}\label{app:vdp-forward-setup}

We consider both unsupervised and semi-supervised settings. For the Van der Pol system, we sample $N_p=1000$ interior points and $N_i=100$ initial points, with an additional $N_d=100$ observations in the semi-supervised setting. The penalty weights are set to $\omega_{bc}=1$ and $\omega_{data}=1$. The constructive parameters are $\Upsilon=a\times\{1,2,4,8,16,32,64\}$ with $a = 15$, $T_{max}=50$, $L_{max}=300$, $W=200$, and the correction interval is set to $L_g=20$, with $\sA_1=\{1.0,0.9,\ldots,0.1\}$, $\sA_2=\{1.0,0.9,\ldots,0\}$, and $\sS_r=\{1-10^{-k}\mid k=1,\ldots,7\}$. PI-SCM uses the sine activation. 

The PINN baselines use shallow and deep architectures of $1\times300$ and $4\times64$, respectively. Each baseline uses the same physics collocation, initial, and observational points as PI-SCM. The PINN baselines employ the hyperbolic tangent activation as recommended by \citet{wang2023expert}, and are trained with Adam for 5,000 epochs using an initial learning rate of $10^{-2}$ and a staircase exponential decay schedule with a decay factor of $0.9$ applied every 100 epochs.

The NNBM baseline uses the same $1\times300$ architecture and the same physics collocation, initial, and observational points as PI-SCM. It employs the sine activation, and its hidden weights and biases are sampled as $\rvw \sim U(-15,15)$ and $\rb \sim U(-15,15)$, respectively. The output weights are optimized using Gauss--Newton iterations. Following the setting in \citet{huang2025neural}, the iterations terminate when the relative change in the linearized mean squared error falls below $10^{-2}$.

For the Van der Pol equation, we additionally compare with PISCN, which adopts the same constructive settings as PI-SCM and is subsequently optimized with Adam for 100 epochs using a learning rate of $5\times10^{-3}$ and a staircase exponential decay schedule with a decay factor of $0.9$ applied every 100 epochs. Since PISCN requires labeled data during construction, it is evaluated only in the semi-supervised setting. PISCN uses the sine activation.

\subsubsection{Forward Problems: Solution Visualization}\label{app:vdp-forward-visualization}

Figure~\ref{fig:vdp_forward_solution} shows the predicted state, absolute error, and representative cross-sections for the Van der Pol forward problem.

\begin{figure}[H]
	\centering
	\includegraphics[width=0.90\linewidth]{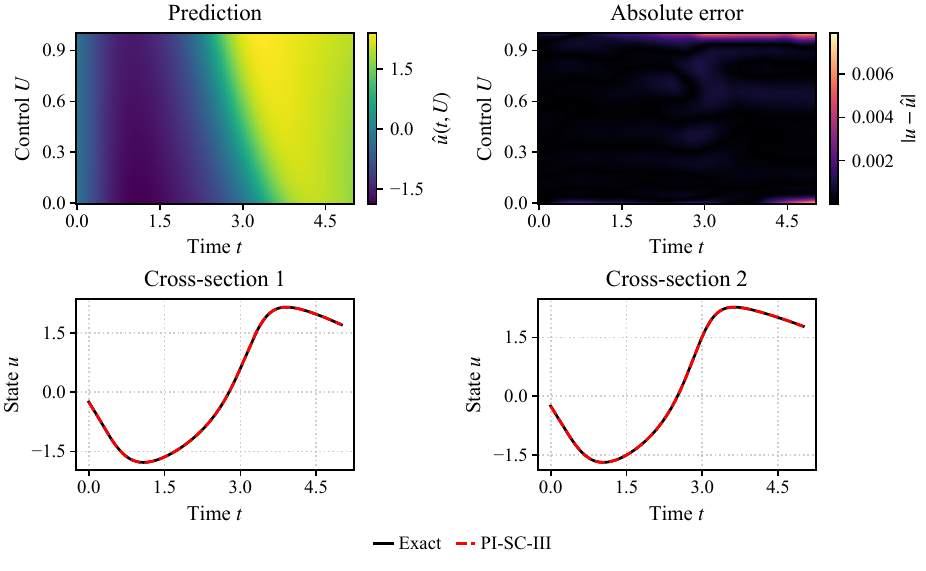}
	\caption{Semi-supervised PI-SC-III solution for the Van der Pol forward problem. The panels show the prediction, absolute pointwise error, and representative cross-sections at $U=0.3$ and $U=0.6$.}
	\label{fig:vdp_forward_solution}
\end{figure}

\subsubsection{Inverse Problems: Experimental Setup}\label{app:vdp-inverse-setup}

This inverse experiment jointly recovers the state and the scalar physical coefficient $\mu$, and compares the two PINN architectures with PI-SC-I, PI-SC-II, and PI-SC-III. The unknown physical parameter has a ground-truth value of $1.0$ and is initialized at $0.1$, with $\mu_{true}=1.0$ and $\mu^{(0)}=0.1$. The experiment uses $N_d=100$ observations and a correction interval of $L_g=20$, while all other sampling, construction, and baseline settings match the corresponding forward experiment.

\subsection{Two-dimensional Helmholtz Equation}\label{app:supplementary-helmholtz}

The Helmholtz equation is a canonical elliptic model of time-harmonic wave propagation. The benchmark is
\begin{equation}\label{eq:helmholtz_original}
	u_{xx}+u_{yy}+k^2u-q(x,y)=0, \qquad (x,y)\in[-1,1]\times[-1,1],
\end{equation}
with homogeneous Dirichlet conditions
\begin{equation}\label{eq:helmholtz_bc}
	u(-1,y)=u(1,y)=u(x,-1)=u(x,1)=0.
\end{equation}

A manufactured source is used:
\begin{equation}\label{eq:helmholtz_source}
	q(x,y)=-(a_1\pi)^2\sin(a_1\pi x)\sin(a_2\pi y)
	-(a_2\pi)^2\sin(a_1\pi x)\sin(a_2\pi y)
	+\sin(a_1\pi x)\sin(a_2\pi y),
\end{equation}
where $a_1=1$ and $a_2=4$. The value $k^2=1$ is prescribed in the forward task, whereas $k^2$ is identified in the inverse task.

The corresponding analytical solution is
\begin{equation}\label{eq:helmholtz_solution}
	u(x,y)=\sin(a_1\pi x)\sin(a_2\pi y).
\end{equation}

\subsubsection{Forward Problems: Experimental Setup}\label{app:helmholtz-forward-setup}
 
For the two-dimensional Helmholtz equation, we use $N_p=2000$ interior collocation points, $N_b=1000$ boundary points, $L_{max}=400$, and the correction interval is set to $L_g=10$. The remaining PI-SCM settings follow those of the Van der Pol experiment.

The PINN baselines use shallow and deep architectures of $1\times400$ and $4\times64$, respectively. Each baseline uses the same physics collocation, boundary, and observational points as PI-SCM. The activation and Adam training settings follow those of the Van der Pol experiment.

The PIELM baseline uses the same $1\times400$ architecture and the same physics collocation, boundary, and observational points as PI-SCM. It employs the sine activation, and its hidden weights and biases are sampled as $\rvw \sim U(-15,15)$ and $\rb \sim U(-15,15)$, respectively.

\subsubsection{Forward Problems: Solution Visualization}\label{app:helmholtz-forward-visualization}

Figure~\ref{fig:helmholtz_forward_solution} shows the predicted state, absolute error, and representative cross-sections for the Helmholtz forward problem.

\begin{figure}[H]
	\centering
	\includegraphics[width=0.90\linewidth]{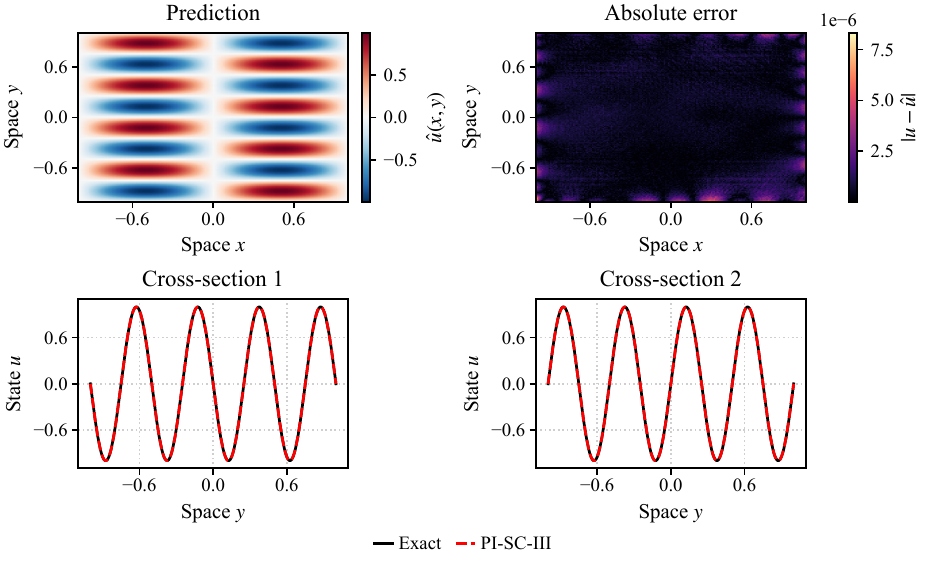}
	\caption{Semi-supervised PI-SC-III solution for the Helmholtz forward problem. The panels show the prediction, absolute pointwise error, and representative cross-sections at $x=-0.5$ and $x=0.5$.}
	\label{fig:helmholtz_forward_solution}
\end{figure}

\subsubsection{Inverse Problems: Experimental Setup}\label{app:helmholtz-inverse-setup}

This inverse experiment jointly recovers the state and the scalar physical parameter $k^2$, and compares the two PINN architectures with PI-SC-I, PI-SC-II, and PI-SC-III. The unknown physical parameter has a ground-truth value of $1.0$ and is initialized at $0.1$, with $k^2_{true}=1.0$ and $\left(k^2\right)^{(0)}=0.1$. The experiment uses $N_d=100$ observations and a correction interval of $L_g=10$, while all other sampling, construction, and baseline settings match the corresponding forward experiment. 

\subsection{Allen--Cahn Equation}\label{app:supplementary-allen-cahn}

The Allen--Cahn equation models phase separation and presents a challenging combination of a stiff reaction term and sharp moving interfaces. We consider
\begin{equation}\label{eq:ac_original}
	u_t-\epsilon u_{xx}+5u^3-5u=0, \qquad (x,t)\in[-1,1]\times[0,0.5],
\end{equation}
where $\epsilon=10^{-4}$, $u(x,0)=x^2\cos(\pi x)$, and the periodic boundary constraints are $u(-1,t)=u(1,t)$ and $u_x(-1,t)=u_x(1,t)$. Because the small diffusion coefficient produces a large scale disparity between the derivative and reaction terms, we define $X=\omega_nx$ with $\omega_n=1/\sqrt{\epsilon}=100$. The transformed equation is
\begin{equation}\label{eq:ac_second_order_scaled}
	\frac{\partial u}{\partial t}-\epsilon\omega_n^2\frac{\partial^2u}{\partial X^2}+5u^3-5u=0.
\end{equation}

Since $\epsilon\omega_n^2=1$, the coordinate stretching balances the state- and derivative-dependent Jacobian contributions and improves the conditioning of $\mM_L$ without changing the underlying physical solution.

\subsubsection{Forward Problems: Experimental Setup}\label{app:allen-cahn-forward-setup}

For the Allen--Cahn equation, we use $N_b=200$, $L_{max}=500$, and the correction interval is set to $L_g=10$, with $a=50$. The remaining PI-SCM settings follow those of the Van der Pol experiment.

 The corresponding variable-scaling physics-informed neural network (VS-PINN) baselines use shallow and deep architectures of $1\times500$ and $4\times64$, respectively. Each baseline uses the same physics collocation, initial/boundary, and observational points as PI-SCM. The activation and Adam training settings follow those of the Van der Pol experiment, with an initial learning rate of \(10^{-3}\).
 
The NNBM baseline uses the same $1\times500$ architecture and the same physics collocation, initial/boundary, and observational points as PI-SCM. It employs the sine activation, and its hidden weights and biases are sampled as $\rvw \sim U(-50,50)$ and $\rb \sim U(-50,50)$, respectively. The Gauss--Newton settings follow those of the Van der Pol experiment.

{The Frozen-PINN-elm and Frozen-PINN-swim baselines \citep{datar2026fast} use a hidden-layer width of 500, $N_p=2000$, $N_i=100$, $N_b=200$, and $N_d=100$. Both use fixed spatial hyperbolic tangent features with time-dependent output weights. Frozen-PINN-elm samples hidden weights independently from $\mathcal{N}(0,1)$ and biases from $U(-1,1)$. The relative cutoff for singular value decomposition (SVD) is $5\times10^{-11}$, and the initial output weights are fitted at the initial points with a relative least squares cutoff of $10^{-10}$. LSODA integrates the coefficient ODE with relative and absolute tolerances of $10^{-6}$.}

\subsubsection{Forward Problems: Solution Visualization}\label{app:allen-cahn-forward-visualization}

Figure~\ref{fig:allen_cahn_forward_solution} shows the predicted state, absolute error, and representative cross-sections for the Allen--Cahn forward problem.

\begin{figure}[H]
	\centering
	\includegraphics[width=0.90\linewidth]{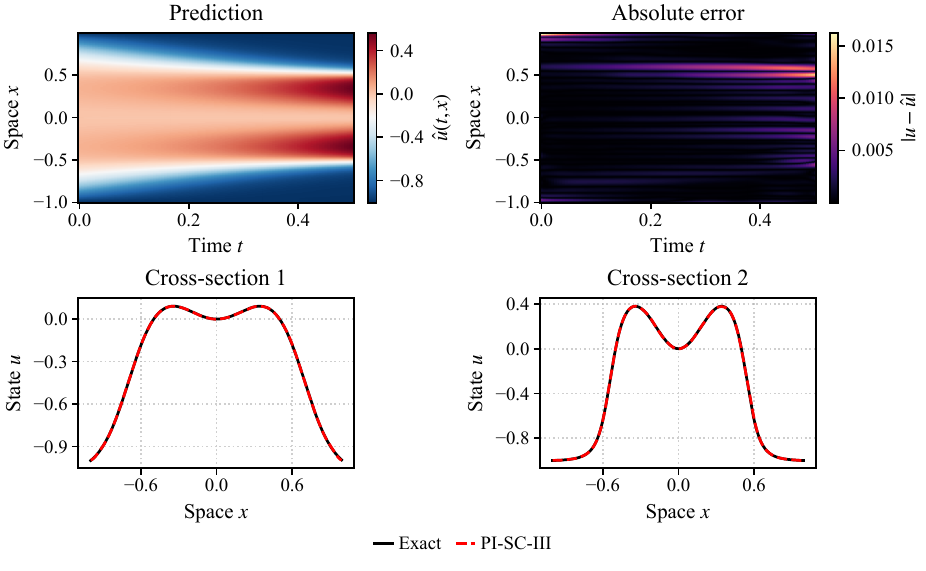}
	\caption{Semi-supervised PI-SC-III solution for the Allen--Cahn forward problem. The panels show the prediction, absolute pointwise error, and representative cross-sections at $t=0.1$ and $t=0.4$.}
	\label{fig:allen_cahn_forward_solution}
\end{figure}

\subsection{Five-dimensional Reaction--Diffusion Equation}\label{app:supplementary-reaction-diffusion}

The reaction--diffusion equation models spatial pattern formation arising from the interaction between diffusion and nonlinear reaction processes. We consider the five-dimensional nonlinear equation
\begin{equation}\label{eq:reaction_diffusion_original}
	u_t-\Delta u-u^2=f(\vx,t), \qquad (\vx,t)\in[-1,1]^5\times[0,1],
\end{equation}
where $\vx=(\evx_1,\ldots,\evx_5)^\top$ and $\Delta=\sum_{j=1}^{5}\partial^2/\partial \evx_j^2$ denotes the spatial Laplacian.

The corresponding analytical solution is
\begin{equation}\label{eq:reaction_diffusion_solution}
	u(\vx,t)=2\sin\!\left(\frac{\pi}{2}\evx_1\right)\cos\!\left(\frac{\pi}{2}\evx_2\right)e^{-t}.
\end{equation}

The initial condition, time-dependent Dirichlet boundary conditions,
and source term are derived from the analytical solution in
Eq.~(\ref{eq:reaction_diffusion_solution}).

\subsubsection{Forward Problems: Experimental Setup}\label{app:reaction-diffusion-forward-setup}

For the five-dimensional reaction--diffusion equation, we use $N_p=1000$, $N_i=250$, $N_b=2500$, set $L_{max}=1000$, and set $L_g=10$, with $a=1.5$. The remaining PI-SCM settings follow those of the Van der Pol experiment.

The PINN baselines use shallow and deep architectures of $1\times1000$ and $4\times64$, respectively. Each baseline uses the same physics collocation, initial/boundary, and observational points as PI-SCM. The activation and Adam training settings follow those of the Van der Pol experiment, with an initial learning rate of \(10^{-3}\).

The NNBM baseline uses the same $1\times1000$ architecture and the same physics collocation, initial/boundary, and observational points as PI-SCM. It employs the sine activation, and its hidden weights and biases are sampled as $\rvw \sim U(-1.5,1.5)$ and $\rb \sim U(-1.5,1.5)$, respectively. The Gauss--Newton settings follow those of the Van der Pol experiment.

{The Frozen-PINN baselines use a hidden-layer width of 1000 and the same numbers of physics collocation, initial/boundary, and observational points as PI-SCM. We evaluate Frozen-PINN-elm and Frozen-PINN-swim, including a SWIM variant with projection onto the span of the initial profile gradients at each candidate pair \citep{datar2026fast}. The feature activation and ELM sampling distributions follow those of the Allen--Cahn experiment. The relative SVD cutoff is $10^{-10}$. The boundary penalty parameter is $10^4$. RK45 is used for Frozen-PINN-elm and DOP853 for both SWIM variants, with relative and absolute tolerances of $10^{-4}$.}

\subsubsection{Forward Problems: Solution Visualization}\label{app:reaction-diffusion-forward-visualization}

Figure~\ref{fig:reaction_diffusion_forward_solution} shows the predicted state, absolute error, and representative cross-sections for the reaction--diffusion forward problem.

\begin{figure}[H]
	\centering
	\includegraphics[width=0.90\linewidth]{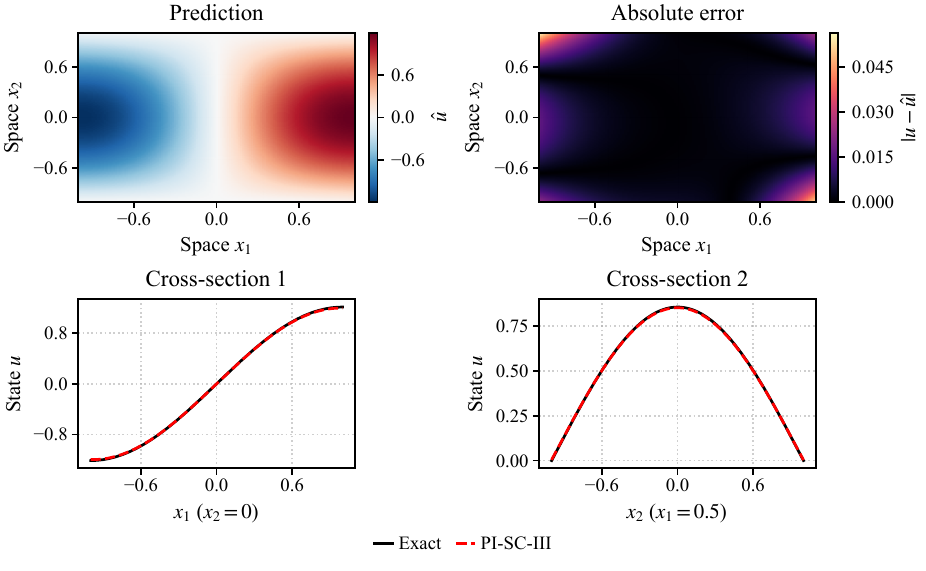}
	\caption{Semi-supervised PI-SC-III solution for the reaction--diffusion forward problem. The panels show the prediction and absolute pointwise error over the displayed $(\evx_1,\evx_2)$ slice, together with representative cross-sections at $\evx_2=0$ and $\evx_1=0.5$.}
	\label{fig:reaction_diffusion_forward_solution}
\end{figure}

\begingroup
\makeatletter
\let\revisionseccntformat\@seccntformat
\def\@seccntformat#1{{\revisionseccntformat{#1}}}
\makeatother
\subsubsection{{Forward Problem Involving All Five Spatial Variables}}\label{app:reaction-diffusion-all-variables-forward}

The analytical solution in Eq.~(\ref{eq:reaction_diffusion_solution}) depends on the first two spatial variables, $\evx_1$ and $\evx_2$. To further assess PI-SCM's ability to handle high-dimensional problems, we consider Eq.~(\ref{eq:reaction_diffusion_original}) on $(\vx,t)\in(0,1)^5\times[0,1]$, with the analytical solution
\begin{equation}\label{eq:reaction_diffusion_all_variables_solution}
	u^\ast(\vx,t)=e^t\prod_{j=1}^{5}\sin(\pi\evx_j).
\end{equation}

The initial condition, Dirichlet boundary conditions, and source term are derived from the analytical solution.

For this forward problem, we set $L_{max}=2000$, and the correction interval is set to $L_g=30$, with $a=3.0$. All other settings follow those in Appendix~\ref{app:reaction-diffusion-forward-setup}. All tabulated values are reported as mean $\pm$ standard deviation, and state predictive accuracy is evaluated on 10,000 randomly sampled test coordinates. Each setting is evaluated over 100 independent trials. Table~\ref{tab:reaction-diffusion-all-variables-forward} reports training time and RMSE under the unsupervised and semi-supervised paradigms.

\begin{table}[H]
	\renewcommand{\arraystretch}{0.78}
	\setlength{\tabcolsep}{3pt}
	\setlength{\aboverulesep}{0.10ex}
	\setlength{\belowrulesep}{0.10ex}
	\centering
	\scriptsize
	\caption{Forward results for the reaction--diffusion problem involving all five spatial variables.}
	\label{tab:reaction-diffusion-all-variables-forward}
	\color{black}
	\begin{tabular}{@{}llcc@{}}
		\toprule
		{Algorithm} & {Paradigm} & {Time (s)} & {RMSE} \\
		\midrule
		{PI-SC-I} & {Unsupervised} & {$22.4\pm0.67$} & {$2.25e{-1}\pm3.33e{-3}$} \\
		 & {Semi-supervised} & {$23.0\pm0.12$} & {$1.28e{-1}\pm1.61e{-2}$} \\
		{PI-SC-II} & {Unsupervised} & {$23.9\pm0.58$} & {$9.47e{-2}\pm6.31e{-3}$} \\
		 & {Semi-supervised} & {$24.5\pm1.27$} & {$4.54e{-2}\pm8.71e{-3}$} \\
		{PI-SC-III} & {Unsupervised} & {$43.0\pm2.27$} & {$6.30e{-3}\pm2.41e{-4}$} \\
		 & {Semi-supervised} & {$44.5\pm1.98$} & {$6.25e{-3}\pm2.66e{-4}$} \\
		\bottomrule
	\end{tabular}
\end{table}

PI-SC-III achieves a mean RMSE of approximately $6.3\times10^{-3}$ in both paradigms, further demonstrating the effectiveness of PI-SCM for high-dimensional inputs.

\endgroup
\subsection{Two-dimensional Bratu Equation}\label{app:supplementary-bratu}

The Bratu equation is a nonlinear elliptic boundary-value problem with an exponential reaction term. We consider
\begin{equation}\label{eq:bratu_original}
	u_{xx}+u_{yy}+\lambda e^u=0, \qquad (x,y)\in[0,1]\times[0,1],
\end{equation}
with homogeneous Dirichlet conditions
\begin{equation}\label{eq:bratu_bc}
	u(0,y)=u(1,y)=u(x,0)=u(x,1)=0.
\end{equation}

The value $\lambda=1$ is prescribed in the forward task, whereas $\lambda$ is identified in the inverse task.

\subsubsection{Forward Problems: Experimental Setup}\label{app:bratu-forward-setup}

For the two-dimensional Bratu equation, we use $N_p=2000$ interior collocation points, $N_b=1000$ boundary points, and $T_{max}=100$, and the correction interval is set to $L_g=30$. The remaining PI-SCM settings follow those of the Van der Pol experiment.

The PINN baselines use shallow and deep architectures of $1\times300$ and $4\times64$, respectively. Each baseline uses the same physics collocation, boundary, and observational points as PI-SCM. The activation and Adam training settings follow those of the Van der Pol experiment, with an initial learning rate of \(10^{-3}\).

The NNBM baseline uses the same $1\times300$ architecture and the same physics collocation, boundary, and observational points as PI-SCM. It employs the sine activation, and its hidden weights and biases are sampled as $\rvw \sim U(-15,15)$ and $\rb \sim U(-15,15)$, respectively. The Gauss--Newton settings follow those of the Van der Pol experiment.

\subsubsection{Forward Problems: Solution Visualization}\label{app:bratu-forward-visualization}

Figure~\ref{fig:bratu_forward_solution} shows the predicted state, absolute error, and centerline comparison for the Bratu forward problem.

\begin{figure}[H]
	\centering
	\includegraphics[width=0.90\linewidth]{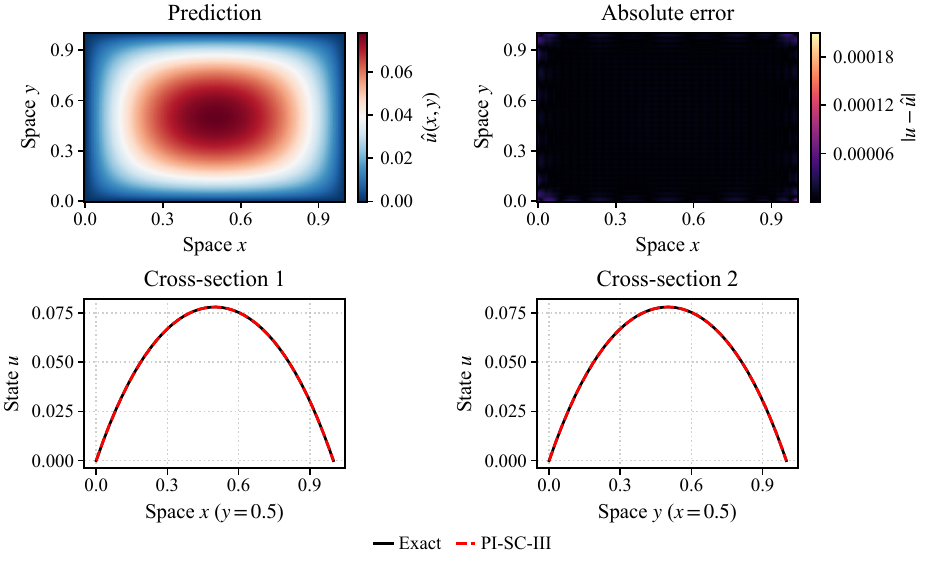}
	\caption{Semi-supervised PI-SC-III solution for the Bratu forward problem. The panels show the prediction and absolute pointwise error over the displayed $(x,y)$ slice, together with representative cross-sections at $y=0.5$ and $x=0.5$.}
	\label{fig:bratu_forward_solution}
\end{figure}

\subsubsection{Inverse Problems: Experimental Setup}\label{app:bratu-inverse-setup}

This inverse experiment jointly recovers the state and the physical parameter $\lambda$, and compares the two PINN architectures with PI-SC-I, PI-SC-II, and PI-SC-III. The unknown physical parameter has a ground-truth value of $1.0$ and is initialized at $0.1$. The experiment uses $N_d=100$ observations and a correction interval of $L_g=30$, while all other sampling, construction, and baseline settings match the corresponding forward experiment. 

\begingroup
\subsection{Navier--Stokes Equations}\label{app:supplementary-navier-stokes}

We consider two-dimensional steady incompressible flows with velocity components $(u,v)$ and pressure $p$. On a spatial domain $\Omega$, the governing Navier--Stokes equations are
\begin{equation}\label{eq:navier_stokes_governing}
	\begin{cases}
		uu_x+vu_y+p_x-\nu\Delta u=0,\\
		uv_x+vv_y+p_y-\nu\Delta v=0,\\
		u_x+v_y=0,
	\end{cases}
	\qquad (x,y)\in\Omega,
\end{equation}
where $\nu$ is the kinematic viscosity and $\Delta=\partial_{xx}+\partial_{yy}$ is the spatial Laplacian.

\paragraph{(a) Kovasznay flow.}\label{app:supplementary-kovasznay}

For the Kovasznay flow, Eq.~(\ref{eq:navier_stokes_governing}) is solved on $\Omega=[-0.5,1]\times[-0.5,1.5]$, with Reynolds number $\mathrm{Re}=40$ and kinematic viscosity $\nu=1/\mathrm{Re}=0.025$. Let
\begin{equation}\label{eq:kovasznay_decay}
	\lambda=\frac{1}{2\nu}-\sqrt{\frac{1}{4\nu^2}+4\pi^2}.
\end{equation}

The corresponding analytical solution is
\begin{equation}\label{eq:kovasznay_solution}
	\begin{aligned}
		u(x,y)&=1-e^{\lambda x}\cos(2\pi y),\\
		v(x,y)&=\frac{\lambda}{2\pi}e^{\lambda x}\sin(2\pi y),\\
		p(x,y)&=\frac{1}{2}\left(1-e^{2\lambda x}\right).
	\end{aligned}
\end{equation}

Dirichlet boundary conditions are imposed on all four sides of $\partial\Omega$, and the pressure gauge is fixed by $p(0,0)=0$. In the semi-supervised setting, observations of $(u,v,p)$ are obtained from the analytical solution.

\paragraph{(b) Poiseuille flow.}\label{app:supplementary-poiseuille}
For the Poiseuille flow, we solve Eq.~(\ref{eq:navier_stokes_governing}) on $\Omega=[0,2]\times[0,0.4]$ with the kinematic viscosity $\nu=0.005$. The boundary conditions are
\begin{equation}\label{eq:poiseuille_bc}
	\begin{aligned}
		p=1,\quad v=0, &\qquad x=0,\\
		p=0,\quad v=0, &\qquad x=2,\\
		u=v=0, &\qquad y\in\{0,0.4\}.
	\end{aligned}
\end{equation}

The corresponding analytical solution is
\begin{equation}\label{eq:poiseuille_solution}
	u(x,y)=50y(0.4-y),\qquad
	v(x,y)=0,\qquad
	p(x,y)=1-0.5x.
\end{equation}

In the semi-supervised setting, observations of $(u,v,p)$ are obtained from the analytical solution.

\subsubsection{Forward Problems: Experimental Setup}\label{app:navier-stokes-forward-setup}

\paragraph{(a) Kovasznay flow.}

For the Kovasznay flow, we use $N_b=400$ boundary points, $T_{max}=10$, $L_{max}=400$, and the correction interval is evaluated at $L_g=30$, with $a=10$. The remaining PI-SCM settings follow those of the Van der Pol experiment.

The PINN baselines use shallow and deep architectures of $1\times400$ and $4\times64$, respectively. Each baseline uses the same physics collocation, boundary, and observational points as PI-SCM. The activation and Adam training settings follow those of the Van der Pol experiment.

The NNBM baseline uses the same $1\times400$ architecture and the same physics collocation, boundary, and observational points as PI-SCM. It employs the sine activation, and its hidden weights and biases are sampled as $\rvw\sim U(-10,10)$ and $\rb\sim U(-10,10)$, respectively. The Gauss--Newton settings follow those of the Van der Pol experiment.

\paragraph{(b) Poiseuille flow.}

For the Poiseuille flow, we use $N_b=400$ boundary points, $L_{max}=200$, $W=100$, and the correction interval is evaluated at $L_g=10$, with $a=5$. The remaining PI-SCM settings follow those of the Van der Pol experiment.

The PINN baselines use shallow and deep architectures of $1\times200$ and $4\times64$, respectively. Each baseline uses the same physics collocation, boundary, and observational points as PI-SCM. The activation and Adam training settings follow those of the Van der Pol experiment, with an initial learning rate of \(10^{-3}\).

The NNBM baseline uses the same $1\times200$ architecture and the same physics collocation, boundary, and observational points as PI-SCM. It employs the sine activation, and its hidden weights and biases are sampled as $\rvw\sim U(-5,5)$ and $\rb\sim U(-5,5)$, respectively. The Gauss--Newton settings follow those of the Van der Pol experiment.

\subsubsection{Forward Problems: Solution Visualization}\label{app:navier-stokes-forward-visualization}

\paragraph{(a) Kovasznay flow.}

Figure~\ref{fig:kovasznay_forward_solution} shows the predicted velocity components and pressure, together with their absolute errors, for the Kovasznay forward problem.

\begin{figure}[H]
	\centering
	\includegraphics[width=0.90\linewidth]{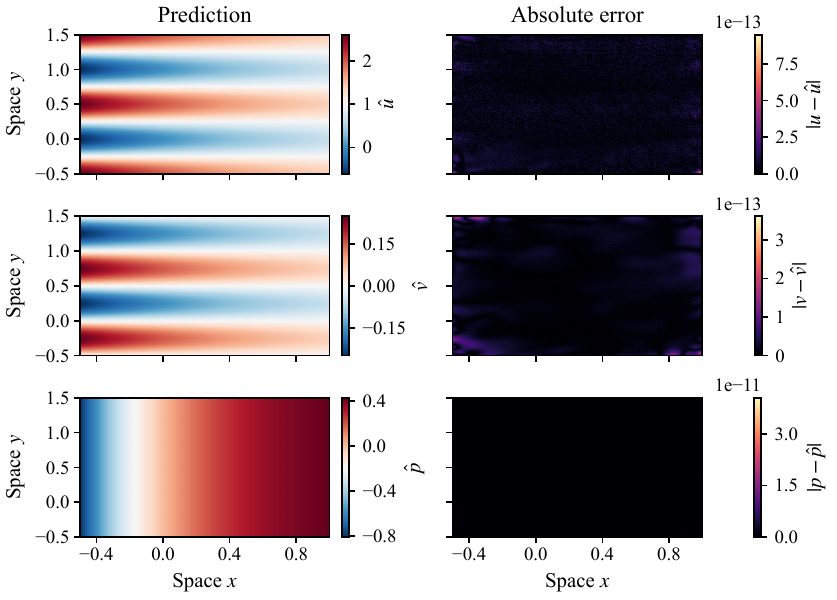}
	\par
	\caption{Semi-supervised PI-SC-III solution for the Kovasznay forward problem. The panels show the predictions and absolute pointwise errors for the velocity components $(u,v)$ and pressure $p$.}
	\label{fig:kovasznay_forward_solution}
\end{figure}

\paragraph{(b) Poiseuille flow.}

Figure~\ref{fig:poiseuille_forward_solution} shows the predicted velocity components and pressure, together with their absolute errors, for the Poiseuille forward problem.

\begin{figure}[H]
	\centering
	\includegraphics[width=0.90\linewidth]{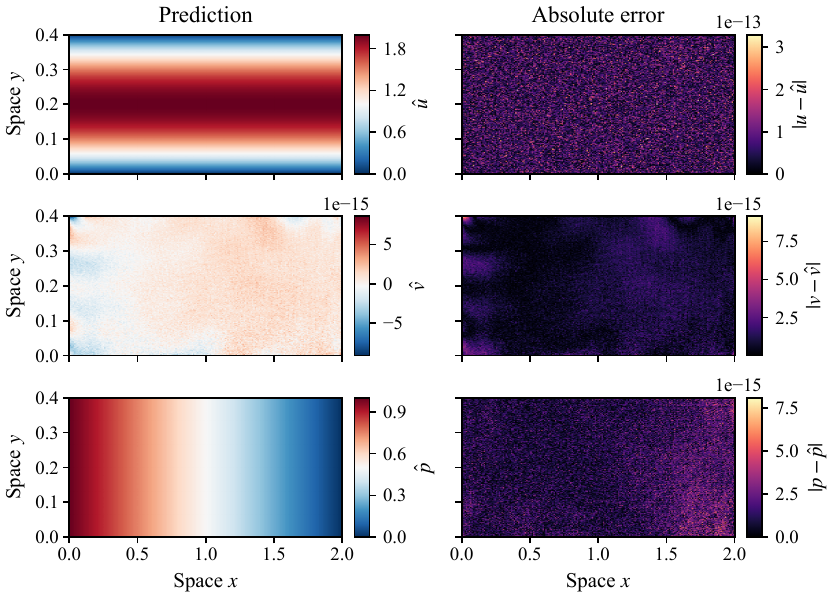}
	\par
	\caption{Semi-supervised PI-SC-III solution for the Poiseuille forward problem. The panels show the predictions and absolute pointwise errors for the velocity components $(u,v)$ and pressure $p$.}
	\label{fig:poiseuille_forward_solution}
\end{figure}

\subsection{Nonlinear Diffusion}\label{app:supplementary-nonlinear-diffusion}

We consider a two-dimensional nonlinear diffusion equation on a tree-shaped domain. The branched geometry is inspired by tree-like patterns produced during the controlled shaping of fluids in a multiport Hele--Shaw cell \citep{islam2017viscous}. The spatial domain $\Omega$ is represented by an 82-vertex polygon with bounding box $[0.65,0.9]^2$, and the time interval is $t\in[0,1]$.

The governing equation is
\begin{equation}\label{eq:nonlinear_diffusion_governing}
	u_t-u(u_{xx}+u_{yy})=f(t,x,y),
	\qquad (t,x,y)\in[0,1]\times\Omega.
\end{equation}

The corresponding analytical solution is
\begin{equation}\label{eq:nonlinear_diffusion_solution}
	u(t,x,y)=5e^{-t}\sin(\pi x)y^{-3}.
\end{equation}

The initial condition, Dirichlet boundary conditions, and source term are derived from the analytical solution in Eq.~(\ref{eq:nonlinear_diffusion_solution}).

In the semi-supervised setting, observations of $u$ are obtained from the analytical solution.

\subsubsection{Forward Problems: Experimental Setup}\label{app:nonlinear-diffusion-forward-setup}

For the nonlinear diffusion equation, we use $N_p=1500$ interior collocation points, $N_i=200$ initial points, $N_b=400$ boundary points, $L_{max}=600$, and the correction interval is evaluated at $L_g=30$, with $a=10$. The penalty weights are set to $\omega_{bc}=10^4$ and $\omega_{data}=10^4$. The remaining PI-SCM settings follow those of the Van der Pol experiment.

The PINN baselines use shallow and deep architectures of $1\times600$ and $4\times64$, respectively. Each baseline uses the same physics collocation, initial/boundary, and observational points as PI-SCM. The activation and Adam training settings follow those of the Van der Pol experiment.

The NNBM baseline uses the same $1\times600$ architecture and the same physics collocation, initial/boundary, and observational points as PI-SCM. It employs the sine activation, and its hidden weights and biases are sampled as $\rvw\sim U(-10,10)$ and $\rb\sim U(-10,10)$, respectively. The Gauss--Newton settings follow those of the Van der Pol experiment.

The Frozen-PINN-swim baseline uses a hidden-layer width of 600 and the same numbers of physics collocation, initial/boundary, and observational points as PI-SCM. The feature activation follows those of the Allen--Cahn experiment. The relative SVD cutoff is $10^{-8}$. The boundary penalty parameter is $10^5$. DOP853 integrates the coefficient ODE with relative and absolute tolerances of $10^{-6}$.

\subsubsection{Forward Problems: Solution Visualization}\label{app:nonlinear-diffusion-forward-visualization}

Figure~\ref{fig:nonlinear_diffusion_forward_solution} shows the predicted state and absolute error at $t=0.5$ for the nonlinear diffusion forward problem.

\begin{figure}[H]
	\centering
	\includegraphics[width=0.90\linewidth]{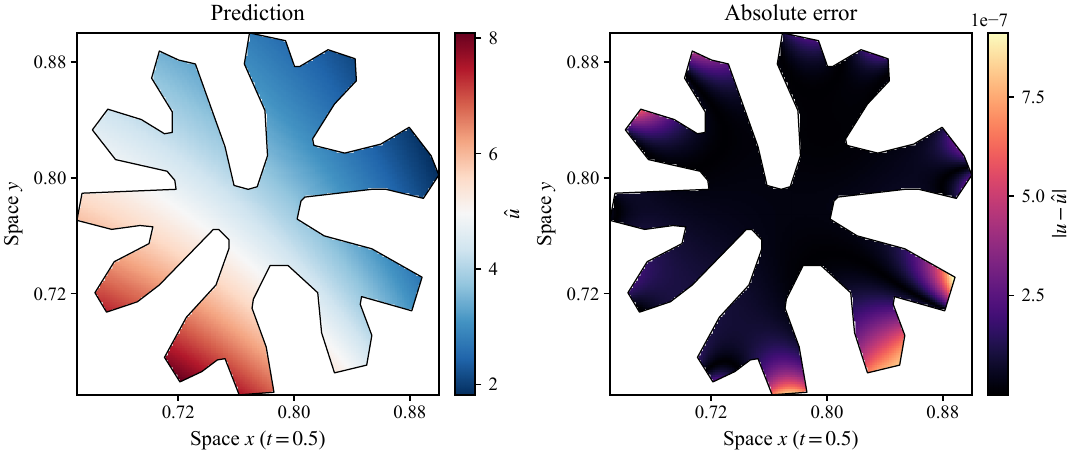}
	\par
	\caption{Semi-supervised PI-SC-III solution for the nonlinear diffusion forward problem. The panels show the prediction and absolute pointwise error at $t=0.5$.}
	\label{fig:nonlinear_diffusion_forward_solution}
\end{figure}

\subsection{Accuracy--Efficiency Study}\label{app:accuracy-efficiency-study}

\subsubsection{Experimental Setup}\label{app:accuracy-efficiency-setup}

We compare semi-supervised PI-SC-III, PIELM baselines, and PINNs on the eight forward benchmark systems using error-versus-time curves. And we compare semi-supervised PI-SC-III and PIELM baselines on the eight forward benchmark systems using error-versus-width curves. The PINNs use $4\times64$ and $8\times64$ architectures and are evaluated after $\{1000,2000,5000,10000,20000\}$ training epochs.

Table~\ref{tab:accuracy-efficiency-network-sizes} lists the evaluated network sizes. Across all eight benchmarks, each method has reached convergence at its maximum evaluated epoch or network size $L$. All other settings follow those of the corresponding forward problem in Appendix~\ref{app:vdp-forward-setup}--\ref{app:nonlinear-diffusion-forward-setup}.

\begin{table}[H]
	\renewcommand{\arraystretch}{0.78}
	\setlength{\tabcolsep}{3pt}
	\setlength{\aboverulesep}{0.10ex}
	\setlength{\belowrulesep}{0.10ex}
	\centering
	\scriptsize
	\caption{Evaluated network sizes for the accuracy--efficiency study. $L$ denotes the hidden-layer width. The notation $a:b:c$ specifies a sequence from $a$ to $c$ with step $b$.}
	\label{tab:accuracy-efficiency-network-sizes}
	\color{black}
	\begin{tabular}{@{}llll@{}}
		\toprule
		& \multicolumn{2}{c}{{Error-versus-time}} & {Error-versus-width} \\
		\cmidrule(lr){2-3}\cmidrule(l){4-4}
		{System} & {PI-SC-III $L$} & {NNBM/PIELM $L$} & {Shared $L$} \\
		\midrule
		{Van der Pol} & {$\{50,100:200:1100\}$} & {$\{100:200:1500\}$} & {$\{100:200:1100\}$} \\
		{Helmholtz} & {$\{100:200:1300\}$} & {$\{100:200:1700\}$} & {$\{100:200:1300\}$} \\
		{Allen--Cahn} & {$\{50,100:200:1100\}$} & {$\{100:200:2300\}$} & {$\{100:200:1100\}$} \\
		{Reaction--diffusion} & {$\{100:200:3300\}$} & {$\{100:200:3900\}$} & {$\{100:200:3300\}$} \\
		{Bratu} & {$\{50,100:200:1100\}$} & {$\{50,100:200:1100\}$} & {$\{100:200:1100\}$} \\
		{Kovasznay} & {$\{100:100:900\}$} & {$\{100:100:800\}$} & {$\{100:100:800\}$} \\
		{Poiseuille} & {$\{50,100:100:500\}$} & {$\{50,100:100:600\}$} & {$\{100:100:500\}$} \\
		{Nonlinear diffusion} & {$\{100:200:1700\}$} & {$\{100:200:1700\}$} & {$\{100:200:1700\}$} \\
		\bottomrule
	\end{tabular}
\end{table}

\subsubsection{Results and Analysis}\label{app:accuracy-efficiency-results}

State RMSE is evaluated on 10,000 randomly sampled test coordinates. Each setting is evaluated over 10 independent trials. Figure~\ref{fig:accuracy-efficiency-pareto} shows the error-versus-time curves across the eight benchmark systems. Figure~\ref{fig:accuracy-efficiency-matched-width} shows the error-versus-width curves across the eight benchmark systems.

\begin{figure}[H]
	\centering
	\includegraphics[width=0.7\linewidth]{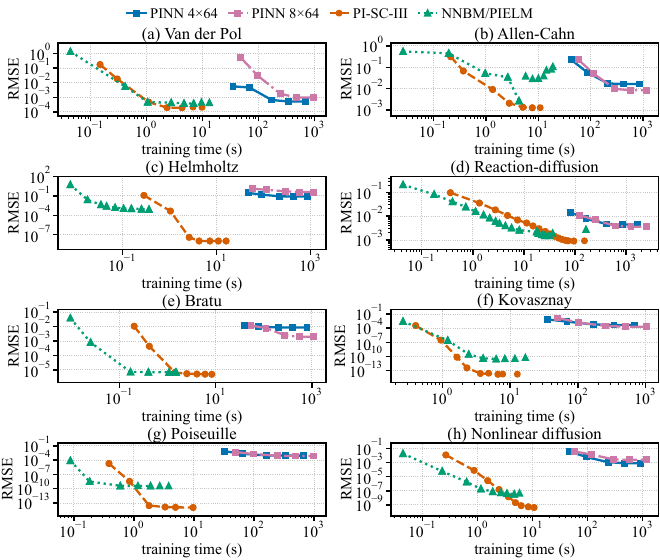}
	\par
	\caption{Error-versus-time curves across the eight benchmark systems. The curves connect the nondominated points within each method.}
	\label{fig:accuracy-efficiency-pareto}
\end{figure}

\begin{figure}[H]
	\centering
	\includegraphics[width=0.7\linewidth]{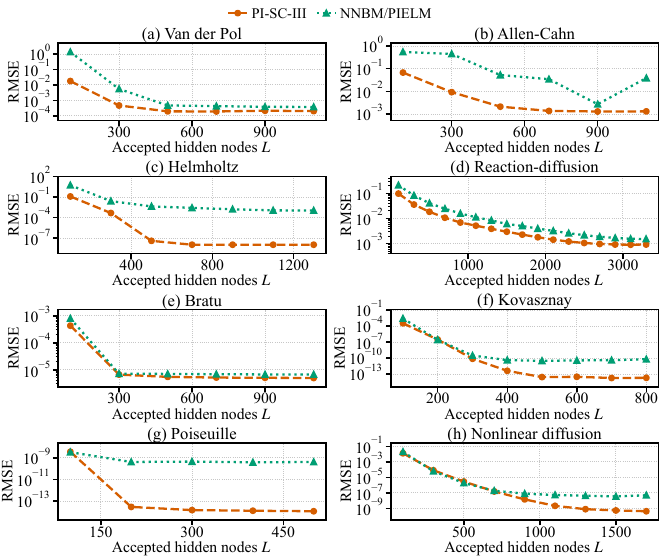}
	\par
	\caption{Error-versus-width curves across the eight benchmark systems.}
	\label{fig:accuracy-efficiency-matched-width}
\end{figure}

We use the lower mean RMSE attained by the two PINN architectures at convergence as the target accuracy for each benchmark. Comparing the corresponding PINN's mean convergence time with the shortest mean training time required by PI-SC-III to reach this accuracy yields speedups ranging from $62.7\times$ to $1317\times$. These results demonstrate that PI-SC-III can attain the accuracy of PINNs at substantially lower training cost.

The error-versus-time curves in Fig.~\ref{fig:accuracy-efficiency-pareto} show that, after both methods have converged, PI-SC-III reduces mean RMSE by factors ranging from approximately $1.40$ to $7.57\times10^4$ relative to the PIELM baselines, with the largest improvements on Helmholtz, Kovasznay, and Poiseuille. The error-versus-width curves in Fig.~\ref{fig:accuracy-efficiency-matched-width} further show that PI-SC-III generally attains lower RMSE at same network widths.

\subsection{Inverse Problems with Noisy Observations}\label{app:noisy-inverse-study}

\subsubsection{Experimental Setup}\label{app:noisy-inverse-setup}

We evaluate PI-SCM on the Van der Pol, Helmholtz, and Bratu inverse problems using noisy state observations. We consider $N_d\in\{100,500,1000\}$ observation points and noise levels $\epsilon\in\{0\%,1\%,5\%,10\%\}$. Additive Gaussian noise is introduced as
\begin{equation}\label{eq:inverse-observation-noise}
    X_{\mathrm{noisy}}=X_{\mathrm{true}}+\epsilon\,\operatorname{std}(X_{\mathrm{true}})\,Z,
    \qquad Z\sim\mathcal{N}(0,1),
\end{equation}
where $X_{\mathrm{true}}$ contains the clean state observations, $\operatorname{std}(X_{\mathrm{true}})$ is their standard deviation, and the entries of $Z$ are sampled independently. All other settings follow those of the corresponding inverse experiments in Appendices~\ref{app:vdp-inverse-setup}, \ref{app:helmholtz-inverse-setup}, and \ref{app:bratu-inverse-setup}.

\subsubsection{Results and Analysis}\label{app:noisy-inverse-results}

All tabulated values are reported as mean $\pm$ standard deviation, and state predictive accuracy is evaluated on 10,000 randomly sampled test coordinates. Each setting is evaluated over 100 independent trials. Tables~\ref{tab:noisy-inverse-vdp}--\ref{tab:noisy-inverse-bratu} report the parameter identification error, training time, and state RMSE across the tested observation counts and noise levels.

\begin{table}[H]
    \renewcommand{\arraystretch}{0.78}
    \setlength{\tabcolsep}{3pt}
    \setlength{\aboverulesep}{0.10ex}
    \setlength{\belowrulesep}{0.10ex}
    \centering
    \scriptsize
    \caption{Inverse results with additive Gaussian observation noise for the Van der Pol oscillator.}
    \label{tab:noisy-inverse-vdp}
    \color{black}
    \begin{tabular}{@{}lllccc@{}}
        \toprule
        {$N_d$} & {Algorithm} & {$\epsilon$} & {$|\mu_{\mathrm{pred}}-\mu_{\mathrm{true}}|$} & {Time (s)} & {State RMSE} \\
        \midrule
        {100} & {PI-SC-I} & {$0\%$} & {$6.72e{-1}\pm5.12e{-2}$} & {$0.80\pm0.02$} & {$1.40e{-1}\pm2.27e{-2}$} \\
         &  & {$1\%$} & {$6.73e{-1}\pm4.99e{-2}$} & {$0.80\pm0.01$} & {$1.40e{-1}\pm2.18e{-2}$} \\
         &  & {$5\%$} & {$6.78e{-1}\pm5.79e{-2}$} & {$0.80\pm0.02$} & {$1.43e{-1}\pm2.52e{-2}$} \\
         &  & {$10\%$} & {$6.88e{-1}\pm8.24e{-2}$} & {$0.81\pm0.03$} & {$1.48e{-1}\pm2.99e{-2}$} \\
        \cmidrule(lr){2-6}
         & {PI-SC-II} & {$0\%$} & {$5.41e{-3}\pm1.29e{-3}$} & {$0.99\pm0.05$} & {$1.10e{-3}\pm1.97e{-4}$} \\
         &  & {$1\%$} & {$1.12e{-2}\pm7.94e{-3}$} & {$1.00\pm0.02$} & {$3.98e{-3}\pm8.44e{-4}$} \\
         &  & {$5\%$} & {$4.98e{-2}\pm3.57e{-2}$} & {$0.98\pm0.03$} & {$1.92e{-2}\pm4.24e{-3}$} \\
         &  & {$10\%$} & {$1.03e{-1}\pm7.10e{-2}$} & {$0.97\pm0.03$} & {$3.92e{-2}\pm8.96e{-3}$} \\
        \cmidrule(lr){2-6}
         & {PI-SC-III} & {$0\%$} & {$3.57e{-3}\pm1.14e{-3}$} & {$1.01\pm0.02$} & {$7.37e{-4}\pm1.80e{-4}$} \\
         &  & {$1\%$} & {$1.09e{-2}\pm7.83e{-3}$} & {$1.01\pm0.03$} & {$4.41e{-3}\pm8.99e{-4}$} \\
         &  & {$5\%$} & {$5.07e{-2}\pm3.60e{-2}$} & {$0.99\pm0.02$} & {$2.18e{-2}\pm4.54e{-3}$} \\
         &  & {$10\%$} & {$1.05e{-1}\pm7.18e{-2}$} & {$0.99\pm0.02$} & {$4.48e{-2}\pm9.73e{-3}$} \\
        \midrule
        {500} & {PI-SC-I} & {$0\%$} & {$6.34e{-1}\pm7.82e{-2}$} & {$0.84\pm0.02$} & {$1.23e{-1}\pm2.73e{-2}$} \\
         &  & {$1\%$} & {$6.35e{-1}\pm8.10e{-2}$} & {$0.83\pm0.01$} & {$1.24e{-1}\pm2.84e{-2}$} \\
         &  & {$5\%$} & {$6.32e{-1}\pm7.49e{-2}$} & {$0.84\pm0.02$} & {$1.23e{-1}\pm2.59e{-2}$} \\
         &  & {$10\%$} & {$6.33e{-1}\pm7.17e{-2}$} & {$0.84\pm0.03$} & {$1.24e{-1}\pm2.51e{-2}$} \\
        \cmidrule(lr){2-6}
         & {PI-SC-II} & {$0\%$} & {$3.97e{-3}\pm8.89e{-4}$} & {$1.06\pm0.02$} & {$8.70e{-4}\pm1.62e{-4}$} \\
         &  & {$1\%$} & {$4.96e{-3}\pm3.53e{-3}$} & {$1.07\pm0.04$} & {$1.86e{-3}\pm3.72e{-4}$} \\
         &  & {$5\%$} & {$1.86e{-2}\pm1.38e{-2}$} & {$1.04\pm0.02$} & {$8.41e{-3}\pm1.91e{-3}$} \\
         &  & {$10\%$} & {$3.69e{-2}\pm2.74e{-2}$} & {$1.04\pm0.03$} & {$1.69e{-2}\pm3.90e{-3}$} \\
        \cmidrule(lr){2-6}
         & {PI-SC-III} & {$0\%$} & {$2.67e{-3}\pm8.43e{-4}$} & {$1.06\pm0.03$} & {$5.55e{-4}\pm1.18e{-4}$} \\
         &  & {$1\%$} & {$4.39e{-3}\pm3.15e{-3}$} & {$1.06\pm0.03$} & {$2.01e{-3}\pm3.69e{-4}$} \\
         &  & {$5\%$} & {$1.86e{-2}\pm1.40e{-2}$} & {$1.05\pm0.02$} & {$9.76e{-3}\pm1.93e{-3}$} \\
         &  & {$10\%$} & {$3.76e{-2}\pm2.81e{-2}$} & {$1.04\pm0.02$} & {$1.96e{-2}\pm3.92e{-3}$} \\
        \midrule
        {1000} & {PI-SC-I} & {$0\%$} & {$6.36e{-1}\pm7.16e{-2}$} & {$0.93\pm0.02$} & {$1.23e{-1}\pm2.47e{-2}$} \\
         &  & {$1\%$} & {$6.33e{-1}\pm6.02e{-2}$} & {$0.94\pm0.02$} & {$1.22e{-1}\pm2.03e{-2}$} \\
         &  & {$5\%$} & {$6.36e{-1}\pm7.31e{-2}$} & {$0.94\pm0.01$} & {$1.24e{-1}\pm2.61e{-2}$} \\
         &  & {$10\%$} & {$6.37e{-1}\pm8.49e{-2}$} & {$0.94\pm0.02$} & {$1.24e{-1}\pm2.89e{-2}$} \\
        \cmidrule(lr){2-6}
         & {PI-SC-II} & {$0\%$} & {$4.09e{-3}\pm1.10e{-3}$} & {$1.17\pm0.03$} & {$8.69e{-4}\pm1.24e{-4}$} \\
         &  & {$1\%$} & {$4.31e{-3}\pm2.96e{-3}$} & {$1.17\pm0.03$} & {$1.49e{-3}\pm3.16e{-4}$} \\
         &  & {$5\%$} & {$1.29e{-2}\pm1.11e{-2}$} & {$1.15\pm0.02$} & {$6.16e{-3}\pm1.22e{-3}$} \\
         &  & {$10\%$} & {$2.51e{-2}\pm2.20e{-2}$} & {$1.15\pm0.03$} & {$1.23e{-2}\pm2.48e{-3}$} \\
        \cmidrule(lr){2-6}
         & {PI-SC-III} & {$0\%$} & {$2.86e{-3}\pm7.98e{-4}$} & {$1.17\pm0.02$} & {$5.71e{-4}\pm1.01e{-4}$} \\
         &  & {$1\%$} & {$3.40e{-3}\pm2.55e{-3}$} & {$1.17\pm0.02$} & {$1.51e{-3}\pm2.62e{-4}$} \\
         &  & {$5\%$} & {$1.27e{-2}\pm1.11e{-2}$} & {$1.16\pm0.02$} & {$7.09e{-3}\pm1.23e{-3}$} \\
         &  & {$10\%$} & {$2.53e{-2}\pm2.22e{-2}$} & {$1.16\pm0.03$} & {$1.42e{-2}\pm2.54e{-3}$} \\
        \bottomrule
    \end{tabular}
\end{table}

\begin{table}[H]
    \renewcommand{\arraystretch}{0.78}
    \setlength{\tabcolsep}{3pt}
    \setlength{\aboverulesep}{0.10ex}
    \setlength{\belowrulesep}{0.10ex}
    \centering
    \scriptsize
    \caption{Inverse results with additive Gaussian observation noise for the Helmholtz equation.}
    \label{tab:noisy-inverse-helmholtz}
    \color{black}
    \begin{tabular}{@{}lllccc@{}}
        \toprule
        {$N_d$} & {Algorithm} & {$\epsilon$} & {$|k^2_{\mathrm{pred}}-k^2_{\mathrm{true}}|$} & {Time (s)} & {State RMSE} \\
        \midrule
        {100} & {PI-SC-I} & {$0\%$} & {$3.72e{+1}\pm1.93e{+1}$} & {$1.42\pm0.04$} & {$2.17e{-1}\pm9.90e{-2}$} \\
         &  & {$1\%$} & {$3.74e{+1}\pm1.97e{+1}$} & {$1.44\pm0.03$} & {$2.17e{-1}\pm9.11e{-2}$} \\
         &  & {$5\%$} & {$3.74e{+1}\pm1.97e{+1}$} & {$1.44\pm0.03$} & {$2.17e{-1}\pm9.05e{-2}$} \\
         &  & {$10\%$} & {$3.74e{+1}\pm1.97e{+1}$} & {$1.45\pm0.03$} & {$2.17e{-1}\pm9.04e{-2}$} \\
        \cmidrule(lr){2-6}
         & {PI-SC-II} & {$0\%$} & {$1.98e{+0}\pm2.44e{+0}$} & {$1.78\pm0.04$} & {$7.85e{-3}\pm6.54e{-3}$} \\
         &  & {$1\%$} & {$1.13e{+0}\pm1.14e{+0}$} & {$1.78\pm0.05$} & {$5.73e{-3}\pm3.12e{-3}$} \\
         &  & {$5\%$} & {$1.42e{+0}\pm9.99e{-1}$} & {$1.79\pm0.05$} & {$6.80e{-3}\pm2.49e{-3}$} \\
         &  & {$10\%$} & {$1.49e{+0}\pm1.06e{+0}$} & {$1.82\pm0.07$} & {$8.48e{-3}\pm2.42e{-3}$} \\
        \cmidrule(lr){2-6}
         & {PI-SC-III} & {$0\%$} & {$5.07e{-4}\pm7.64e{-4}$} & {$1.97\pm0.09$} & {$3.84e{-5}\pm4.88e{-5}$} \\
         &  & {$1\%$} & {$1.13e{-1}\pm8.35e{-2}$} & {$1.85\pm0.05$} & {$7.29e{-4}\pm2.15e{-4}$} \\
         &  & {$5\%$} & {$5.66e{-1}\pm4.26e{-1}$} & {$1.80\pm0.05$} & {$3.71e{-3}\pm1.08e{-3}$} \\
         &  & {$10\%$} & {$1.09e{+0}\pm8.87e{-1}$} & {$1.80\pm0.05$} & {$7.61e{-3}\pm2.42e{-3}$} \\
        \midrule
        {500} & {PI-SC-I} & {$0\%$} & {$3.73e{+1}\pm1.96e{+1}$} & {$1.51\pm0.03$} & {$2.14e{-1}\pm9.05e{-2}$} \\
         &  & {$1\%$} & {$3.73e{+1}\pm1.96e{+1}$} & {$1.51\pm0.04$} & {$2.13e{-1}\pm9.05e{-2}$} \\
         &  & {$5\%$} & {$3.73e{+1}\pm1.96e{+1}$} & {$1.50\pm0.03$} & {$2.11e{-1}\pm9.11e{-2}$} \\
         &  & {$10\%$} & {$3.73e{+1}\pm1.95e{+1}$} & {$1.50\pm0.03$} & {$2.12e{-1}\pm9.00e{-2}$} \\
        \cmidrule(lr){2-6}
         & {PI-SC-II} & {$0\%$} & {$9.92e{-1}\pm8.99e{-1}$} & {$1.95\pm0.06$} & {$5.22e{-3}\pm2.45e{-3}$} \\
         &  & {$1\%$} & {$1.06e{+0}\pm9.06e{-1}$} & {$1.91\pm0.06$} & {$5.48e{-3}\pm2.62e{-3}$} \\
         &  & {$5\%$} & {$9.63e{-1}\pm7.73e{-1}$} & {$1.91\pm0.05$} & {$5.39e{-3}\pm2.23e{-3}$} \\
         &  & {$10\%$} & {$1.50e{+0}\pm1.42e{+0}$} & {$1.90\pm0.05$} & {$7.04e{-3}\pm3.42e{-3}$} \\
        \cmidrule(lr){2-6}
         & {PI-SC-III} & {$0\%$} & {$7.61e{-5}\pm7.48e{-5}$} & {$2.00\pm0.04$} & {$1.17e{-5}\pm9.81e{-6}$} \\
         &  & {$1\%$} & {$5.31e{-2}\pm3.87e{-2}$} & {$1.98\pm0.04$} & {$3.42e{-4}\pm9.50e{-5}$} \\
         &  & {$5\%$} & {$2.66e{-1}\pm1.94e{-1}$} & {$1.93\pm0.06$} & {$1.71e{-3}\pm4.34e{-4}$} \\
         &  & {$10\%$} & {$5.53e{-1}\pm4.38e{-1}$} & {$1.91\pm0.05$} & {$3.53e{-3}\pm1.02e{-3}$} \\
        \midrule
        {1000} & {PI-SC-I} & {$0\%$} & {$3.74e{+1}\pm1.96e{+1}$} & {$1.67\pm0.02$} & {$2.16e{-1}\pm8.91e{-2}$} \\
         &  & {$1\%$} & {$3.74e{+1}\pm1.96e{+1}$} & {$1.67\pm0.03$} & {$2.16e{-1}\pm8.91e{-2}$} \\
         &  & {$5\%$} & {$3.74e{+1}\pm1.96e{+1}$} & {$1.66\pm0.02$} & {$2.15e{-1}\pm8.94e{-2}$} \\
         &  & {$10\%$} & {$3.74e{+1}\pm1.96e{+1}$} & {$1.66\pm0.03$} & {$2.14e{-1}\pm8.90e{-2}$} \\
        \cmidrule(lr){2-6}
         & {PI-SC-II} & {$0\%$} & {$1.08e{+0}\pm1.04e{+0}$} & {$2.05\pm0.04$} & {$5.65e{-3}\pm3.09e{-3}$} \\
         &  & {$1\%$} & {$1.13e{+0}\pm1.11e{+0}$} & {$2.05\pm0.05$} & {$5.62e{-3}\pm3.12e{-3}$} \\
         &  & {$5\%$} & {$9.83e{-1}\pm8.29e{-1}$} & {$2.06\pm0.05$} & {$5.36e{-3}\pm2.51e{-3}$} \\
         &  & {$10\%$} & {$1.40e{+0}\pm1.30e{+0}$} & {$2.06\pm0.05$} & {$6.58e{-3}\pm3.25e{-3}$} \\
        \cmidrule(lr){2-6}
         & {PI-SC-III} & {$0\%$} & {$8.08e{-5}\pm8.16e{-5}$} & {$2.13\pm0.03$} & {$1.23e{-5}\pm9.56e{-6}$} \\
         &  & {$1\%$} & {$4.28e{-2}\pm2.92e{-2}$} & {$2.12\pm0.04$} & {$2.52e{-4}\pm6.86e{-5}$} \\
         &  & {$5\%$} & {$2.08e{-1}\pm1.38e{-1}$} & {$2.07\pm0.04$} & {$1.24e{-3}\pm3.22e{-4}$} \\
         &  & {$10\%$} & {$4.13e{-1}\pm2.83e{-1}$} & {$2.05\pm0.04$} & {$2.53e{-3}\pm6.75e{-4}$} \\
        \bottomrule
    \end{tabular}
\end{table}

\begin{table}[H]
    \renewcommand{\arraystretch}{0.78}
    \setlength{\tabcolsep}{3pt}
    \setlength{\aboverulesep}{0.10ex}
    \setlength{\belowrulesep}{0.10ex}
    \centering
    \scriptsize
    \caption{Inverse results with additive Gaussian observation noise for the Bratu equation.}
    \label{tab:noisy-inverse-bratu}
    \color{black}
    \begin{tabular}{@{}lllccc@{}}
        \toprule
        {$N_d$} & {Algorithm} & {$\epsilon$} & {$|\lambda_{\mathrm{pred}}-\lambda_{\mathrm{true}}|$} & {Time (s)} & {State RMSE} \\
        \midrule
        {100} & {PI-SC-I} & {$0\%$} & {$6.99e{-1}\pm1.47e{-2}$} & {$1.52\pm0.03$} & {$1.91e{-2}\pm5.24e{-4}$} \\
         &  & {$1\%$} & {$6.99e{-1}\pm1.47e{-2}$} & {$1.54\pm0.07$} & {$1.91e{-2}\pm5.22e{-4}$} \\
         &  & {$5\%$} & {$6.99e{-1}\pm1.47e{-2}$} & {$1.52\pm0.02$} & {$1.91e{-2}\pm5.40e{-4}$} \\
         &  & {$10\%$} & {$6.97e{-1}\pm1.43e{-2}$} & {$1.53\pm0.02$} & {$1.90e{-2}\pm5.20e{-4}$} \\
        \cmidrule(lr){2-6}
         & {PI-SC-II} & {$0\%$} & {$1.09e{-4}\pm3.97e{-6}$} & {$1.61\pm0.03$} & {$8.02e{-6}\pm6.34e{-7}$} \\
         &  & {$1\%$} & {$4.48e{-4}\pm3.65e{-4}$} & {$1.62\pm0.04$} & {$3.16e{-5}\pm1.11e{-5}$} \\
         &  & {$5\%$} & {$2.14e{-3}\pm1.83e{-3}$} & {$1.61\pm0.02$} & {$1.52e{-4}\pm5.69e{-5}$} \\
         &  & {$10\%$} & {$4.27e{-3}\pm3.66e{-3}$} & {$1.61\pm0.02$} & {$3.04e{-4}\pm1.14e{-4}$} \\
        \cmidrule(lr){2-6}
         & {PI-SC-III} & {$0\%$} & {$1.09e{-4}\pm1.92e{-6}$} & {$1.62\pm0.05$} & {$6.47e{-6}\pm1.60e{-7}$} \\
         &  & {$1\%$} & {$4.47e{-4}\pm3.65e{-4}$} & {$1.63\pm0.05$} & {$3.13e{-5}\pm1.11e{-5}$} \\
         &  & {$5\%$} & {$2.14e{-3}\pm1.83e{-3}$} & {$1.70\pm0.62$} & {$1.53e{-4}\pm5.67e{-5}$} \\
         &  & {$10\%$} & {$4.27e{-3}\pm3.66e{-3}$} & {$1.63\pm0.06$} & {$3.06e{-4}\pm1.14e{-4}$} \\
        \midrule
        {500} & {PI-SC-I} & {$0\%$} & {$6.97e{-1}\pm1.41e{-2}$} & {$1.69\pm0.02$} & {$1.89e{-2}\pm4.95e{-4}$} \\
         &  & {$1\%$} & {$6.97e{-1}\pm1.44e{-2}$} & {$1.68\pm0.02$} & {$1.89e{-2}\pm4.92e{-4}$} \\
         &  & {$5\%$} & {$6.96e{-1}\pm1.47e{-2}$} & {$1.68\pm0.02$} & {$1.88e{-2}\pm5.21e{-4}$} \\
         &  & {$10\%$} & {$6.95e{-1}\pm1.48e{-2}$} & {$1.69\pm0.02$} & {$1.88e{-2}\pm5.29e{-4}$} \\
        \cmidrule(lr){2-6}
         & {PI-SC-II} & {$0\%$} & {$1.20e{-4}\pm3.56e{-6}$} & {$1.78\pm0.02$} & {$7.86e{-6}\pm6.83e{-7}$} \\
         &  & {$1\%$} & {$2.22e{-4}\pm1.60e{-4}$} & {$1.78\pm0.02$} & {$1.61e{-5}\pm4.35e{-6}$} \\
         &  & {$5\%$} & {$1.00e{-3}\pm7.52e{-4}$} & {$1.79\pm0.02$} & {$7.00e{-5}\pm2.41e{-5}$} \\
         &  & {$10\%$} & {$1.99e{-3}\pm1.52e{-3}$} & {$1.79\pm0.03$} & {$1.39e{-4}\pm4.85e{-5}$} \\
        \cmidrule(lr){2-6}
         & {PI-SC-III} & {$0\%$} & {$1.22e{-4}\pm1.52e{-6}$} & {$1.80\pm0.03$} & {$6.41e{-6}\pm1.65e{-7}$} \\
         &  & {$1\%$} & {$2.23e{-4}\pm1.61e{-4}$} & {$1.78\pm0.02$} & {$1.55e{-5}\pm4.46e{-6}$} \\
         &  & {$5\%$} & {$1.01e{-3}\pm7.52e{-4}$} & {$1.79\pm0.02$} & {$7.03e{-5}\pm2.41e{-5}$} \\
         &  & {$10\%$} & {$2.00e{-3}\pm1.52e{-3}$} & {$1.79\pm0.02$} & {$1.40e{-4}\pm4.83e{-5}$} \\
        \midrule
        {1000} & {PI-SC-I} & {$0\%$} & {$7.05e{-1}\pm1.40e{-2}$} & {$1.72\pm0.02$} & {$1.93e{-2}\pm5.06e{-4}$} \\
         &  & {$1\%$} & {$7.05e{-1}\pm1.40e{-2}$} & {$1.72\pm0.02$} & {$1.93e{-2}\pm5.06e{-4}$} \\
         &  & {$5\%$} & {$7.05e{-1}\pm1.39e{-2}$} & {$1.72\pm0.02$} & {$1.93e{-2}\pm5.18e{-4}$} \\
         &  & {$10\%$} & {$7.04e{-1}\pm1.43e{-2}$} & {$1.72\pm0.02$} & {$1.93e{-2}\pm5.27e{-4}$} \\
        \cmidrule(lr){2-6}
         & {PI-SC-II} & {$0\%$} & {$1.23e{-4}\pm3.15e{-6}$} & {$1.82\pm0.02$} & {$7.74e{-6}\pm5.86e{-7}$} \\
         &  & {$1\%$} & {$1.64e{-4}\pm1.26e{-4}$} & {$1.82\pm0.02$} & {$1.26e{-5}\pm2.26e{-6}$} \\
         &  & {$5\%$} & {$7.17e{-4}\pm5.05e{-4}$} & {$1.83\pm0.03$} & {$4.90e{-5}\pm1.37e{-5}$} \\
         &  & {$10\%$} & {$1.44e{-3}\pm1.01e{-3}$} & {$1.83\pm0.03$} & {$9.68e{-5}\pm2.77e{-5}$} \\
        \cmidrule(lr){2-6}
         & {PI-SC-III} & {$0\%$} & {$1.22e{-4}\pm1.28e{-6}$} & {$1.85\pm0.02$} & {$6.44e{-6}\pm1.61e{-7}$} \\
         &  & {$1\%$} & {$1.63e{-4}\pm1.25e{-4}$} & {$1.85\pm0.02$} & {$1.18e{-5}\pm2.34e{-6}$} \\
         &  & {$5\%$} & {$7.18e{-4}\pm5.06e{-4}$} & {$1.86\pm0.02$} & {$4.91e{-5}\pm1.37e{-5}$} \\
         &  & {$10\%$} & {$1.44e{-3}\pm1.01e{-3}$} & {$1.86\pm0.02$} & {$9.74e{-5}\pm2.76e{-5}$} \\
        \bottomrule
    \end{tabular}
\end{table}

Tables~\ref{tab:noisy-inverse-vdp}--\ref{tab:noisy-inverse-bratu} show that parameter identification errors generally increase with the observation noise level. For PI-SC-III at $\epsilon=10\%$, increasing $N_d$ from 100 to 1000 reduces the mean parameter error by factors of approximately $2.6$ to $4.2$ across the three systems. These results support the effectiveness of PI-SCM for identifying physical parameters from noisy observations and show that additional observations can mitigate the impact of noise on parameter recovery.

\subsection{Activation-Function Study}\label{app:activation-study}

\subsubsection{Experimental Setup}\label{app:activation-setup}

We investigate the influence of the activation function on PI-SCM using the semi-supervised Van der Pol, Helmholtz, Allen--Cahn, reaction--diffusion, Bratu, Kovasznay, Poiseuille, and nonlinear diffusion forward problems. PI-SC-III is evaluated with sine, hyperbolic tangent, and sigmoid activations. The $L_{max}$ and the scale $a$ of the candidate range $\Upsilon$ are listed in Table~\ref{tab:activation_performance}. For each system, $a$ is tuned separately to suit each activation function. The correction interval is set to $L_g=1$. All other settings follow those of the corresponding forward problem in Appendix~\ref{app:vdp-forward-setup}--\ref{app:nonlinear-diffusion-forward-setup}.

\subsubsection{Results}\label{app:activation-results}

All tabulated values are reported as mean $\pm$ standard deviation, and state predictive accuracy is evaluated on 10,000 randomly sampled test coordinates. Each activation is evaluated over 10 independent trials. Table~\ref{tab:activation_performance} summarizes the state RMSE across the eight benchmark systems, while Fig.~\ref{fig:activation_residual} shows the evolution of the average true residual $\mathcal{E}_L^2$ with the number of accepted hidden nodes.

\begin{table}[H]
	\renewcommand{\arraystretch}{0.78}
	\setlength{\tabcolsep}{3pt}
	\setlength{\aboverulesep}{0.10ex}
	\setlength{\belowrulesep}{0.10ex}
	\centering
	\scriptsize
	\caption{Semi-supervised PI-SC-III results with different activation functions.}
	\label{tab:activation_performance}
	\color{black}
	\begin{tabular}{@{}lclcc@{}}
		\toprule
		{System} & {$L_{max}$} & {Activation} & {$a$} & {RMSE} \\
		\midrule
		{Van der Pol} & {1000} & {Sine} & {15} & {$1.45e{-4}\pm7.76e{-5}$} \\
		 &  & {Tanh} & {9.5} & {$1.55e{-4}\pm3.31e{-5}$} \\
		 &  & {Sigmoid} & {20} & {$2.91e{-4}\pm1.83e{-4}$} \\
		\midrule
		{Helmholtz} & {600} & {Sine} & {15} & {$1.06e{-8}\pm1.47e{-11}$} \\
		 &  & {Tanh} & {1.5} & {$9.81e{-7}\pm3.84e{-7}$} \\
		 &  & {Sigmoid} & {3} & {$1.53e{-6}\pm8.23e{-7}$} \\
		\midrule
		{Allen--Cahn} & {1200} & {Sine} & {50} & {$1.17e{-3}\pm5.93e{-4}$} \\
		 &  & {Tanh} & {9.5} & {$2.39e{-4}\pm6.29e{-5}$} \\
		 &  & {Sigmoid} & {20} & {$1.24e{-3}\pm1.54e{-3}$} \\
		\midrule
		{Reaction--diffusion} & {1200} & {Sine} & {1.5} & {$3.82e{-3}\pm2.49e{-4}$} \\
		 &  & {Tanh} & {0.5} & {$8.56e{-3}\pm3.87e{-4}$} \\
		 &  & {Sigmoid} & {1} & {$9.19e{-3}\pm6.55e{-4}$} \\
		\midrule
		{Bratu} & {800} & {Sine} & {15} & {$5.15e{-6}\pm5.79e{-8}$} \\
		 &  & {Tanh} & {6} & {$4.70e{-6}\pm8.59e{-8}$} \\
		 &  & {Sigmoid} & {10.5} & {$5.21e{-6}\pm1.00e{-7}$} \\
		\midrule
		{Kovasznay} & {700} & {Sine} & {10} & {$8.86e{-15}\pm6.81e{-15}$} \\
		 &  & {Tanh} & {1.5} & {$5.09e{-10}\pm1.34e{-9}$} \\
		 &  & {Sigmoid} & {2.5} & {$2.79e{-10}\pm3.90e{-10}$} \\
		\midrule
		{Poiseuille} & {700} & {Sine} & {5} & {$1.24e{-14}\pm3.23e{-15}$} \\
		 &  & {Tanh} & {3} & {$1.07e{-14}\pm4.27e{-15}$} \\
		 &  & {Sigmoid} & {6.5} & {$2.76e{-15}\pm7.66e{-16}$} \\
		\midrule
		{Nonlinear diffusion} & {1200} & {Sine} & {10} & {$5.82e{-10}\pm2.43e{-10}$} \\
		 &  & {Tanh} & {2} & {$6.19e{-9}\pm2.54e{-9}$} \\
		 &  & {Sigmoid} & {3.5} & {$2.87e{-9}\pm2.07e{-9}$} \\
		\bottomrule
	\end{tabular}
\end{table}

\begin{figure}[H]
	\centering
	\includegraphics[width=0.7\linewidth]{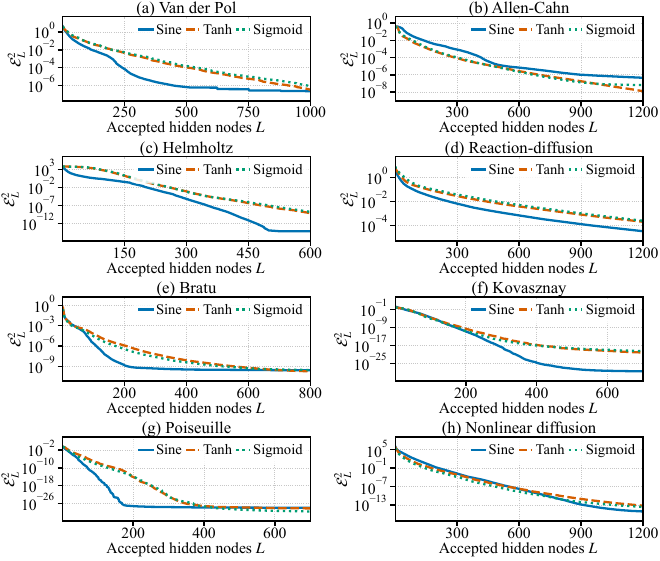}
	\par
	\caption{Evolution of $\mathcal{E}_L^2$ with the number of accepted hidden nodes for three activation functions across eight benchmark systems.}
	\label{fig:activation_residual}
\end{figure}

\subsubsection{Result Analysis}\label{app:activation-analysis}

Table~\ref{tab:activation_performance} shows that sine achieves the lowest mean RMSE on five of the eight benchmark systems. Figure~\ref{fig:activation_residual} also shows rapid residual reduction with sine on several systems. Tanh yields the lowest mean RMSE on Allen--Cahn and Bratu, while sigmoid performs best on Poiseuille. These results support sine as a practical default activation for PI-SCM.

These findings further demonstrate the effectiveness of PI-SCM across different activation functions, showing that its physics-informed construction and analytical update mechanisms can support accurate solutions using multiple choices of activation.

\endgroup
\subsection{Ablation Study}\label{app:ablation-study}

\subsubsection{Experimental Setup}\label{app:ablation-setup}

We conduct three ablation studies of semi-supervised PI-SC-III. The first two examine the number of interior collocation points $N_p$ and the correction interval $L_g$ on all eight benchmark systems. The third assesses the local linearization and discrete line search mechanisms on the seven systems listed in Table~\ref{tab:gauss-newton-ablation}.

For the collocation point ablation, we use $N_p\in\{1000,2000,4000\}$ for Van der Pol, Helmholtz, reaction--diffusion, and Bratu, $N_p\in\{500,1000,2000\}$ for Allen--Cahn, {Kovasznay, and Poiseuille}, and {$N_p\in\{750,1500,3000\}$ for nonlinear diffusion}. The correction interval is set to $L_g=1$ for all eight systems. All other settings follow those of the corresponding forward problem in Appendix~\ref{app:vdp-forward-setup}--\ref{app:nonlinear-diffusion-forward-setup}.

For the correction interval ablation, $L_g\in\{1,5,10,20,30,40,50\}$ for all eight systems. All other settings follow those of the corresponding forward problem in Appendix~\ref{app:vdp-forward-setup}--\ref{app:nonlinear-diffusion-forward-setup}.

\begingroup
To assess the local linearization and discrete line search mechanisms in nonlinear differential equations, we evaluate semi-supervised PI-SC-III on all benchmark systems except Helmholtz, replacing the single linearized least-squares updates and discrete line search with Gauss--Newton iterations for both local node construction and global correction. The local nonlinear least-squares problem is solved for the new node's output weights, initialized at zero, and the global problem is solved for all output weights, initialized from the local update. The correction interval is set to $L_g=1$, and all other settings follow those of the corresponding forward problem in Appendix~\ref{app:vdp-forward-setup}--\ref{app:nonlinear-diffusion-forward-setup}.

\endgroup
\subsubsection{Results}\label{app:ablation-results}

All tabulated values are reported as mean $\pm$ standard deviation over 100 independent trials. State RMSE is evaluated on 10,000 randomly sampled test coordinates.

\noindent\textbf{Number of interior collocation points.} Table~\ref{tab:collocation-count-ablation} reports the training time and state RMSE across the tested values of $N_p$.

\begin{table}[H]
	\renewcommand{\arraystretch}{0.78}
	\setlength{\tabcolsep}{3pt}
	\setlength{\aboverulesep}{0.10ex}
	\setlength{\belowrulesep}{0.10ex}
	\centering
	\scriptsize
	\caption{Sensitivity of semi-supervised PI-SC-III to the number of interior collocation points.}
	\label{tab:collocation-count-ablation}
	\begin{tabular}{@{}lccc@{}}
		\toprule
		System & $N_p$ & Time (s) & RMSE \\
		\midrule
		Van der Pol oscillator
		& 1000 & $1.61\pm0.04$ & $6.05e{-4}\pm1.66e{-4}$ \\
		& 2000 & $1.82\pm0.05$ & $4.96e{-4}\pm1.96e{-4}$ \\
		& 4000 & $2.33\pm0.03$ & $4.87e{-4}\pm2.18e{-4}$ \\
		\midrule
		Helmholtz equation
		& 1000 & $2.26\pm0.05$ & $1.06e{-6}\pm5.22e{-7}$ \\
		& 2000 & $2.62\pm0.04$ & $1.04e{-6}\pm7.75e{-7}$ \\
		& 4000 & $3.53\pm0.05$ & $1.12e{-6}\pm2.03e{-6}$ \\
		\midrule
		Allen--Cahn equation
		& 500 & $3.21\pm0.08$ & $1.78e{-3}\pm7.03e{-4}$ \\
		& 1000 & $3.54\pm0.06$ & $1.52e{-3}\pm2.66e{-4}$ \\
		& 2000 & $4.31\pm0.04$ & $2.16e{-3}\pm2.56e{-4}$ \\
		\midrule
		Reaction--diffusion equation
		& 1000 & $17.1\pm0.12$ & $6.08e{-3}\pm4.72e{-4}$ \\
		& 2000 & $19.7\pm0.09$ & $5.97e{-3}\pm4.25e{-4}$ \\
		& 4000 & $28.1\pm0.17$ & $5.70e{-3}\pm4.01e{-4}$ \\
		\midrule
		Bratu equation
		& 1000 & $2.48\pm0.09$ & $9.25e{-6}\pm5.15e{-6}$ \\
		& 2000 & $3.01\pm0.05$ & $7.01e{-6}\pm1.39e{-6}$ \\
		& 4000 & $4.53\pm0.09$ & $6.61e{-6}\pm1.07e{-6}$ \\
		\midrule
		{Kovasznay flow}
		& {500} & {$10.3\pm0.21$} & {$6.20e{-12}\pm9.47e{-12}$} \\
		& {1000} & {$16.9\pm0.26$} & {$4.66e{-13}\pm6.91e{-13}$} \\
		& {2000} & {$30.7\pm0.38$} & {$1.75e{-13}\pm3.66e{-13}$} \\
		\midrule
		{Poiseuille flow}
		& {500} & {$3.33\pm0.05$} & {$2.32e{-14}\pm1.14e{-14}$} \\
		& {1000} & {$4.53\pm0.04$} & {$2.12e{-14}\pm8.37e{-15}$} \\
		& {2000} & {$7.37\pm0.10$} & {$1.99e{-14}\pm6.44e{-15}$} \\
		\midrule
		{Nonlinear diffusion}
		& {750} & {$4.96\pm0.04$} & {$3.80e{-7}\pm1.57e{-7}$} \\
		& {1500} & {$6.37\pm0.06$} & {$5.45e{-7}\pm2.69e{-7}$} \\
		& {3000} & {$9.78\pm0.08$} & {$6.49e{-7}\pm2.43e{-7}$} \\
		\bottomrule
	\end{tabular}
\end{table}

\noindent\textbf{Correction interval.} Table~\ref{tab:correction-interval-ablation} reports the training time and state RMSE across the tested values of $L_g$.

\begin{table}[H]
	\renewcommand{\arraystretch}{0.78}
	\setlength{\tabcolsep}{3pt}
	\setlength{\aboverulesep}{0.10ex}
	\setlength{\belowrulesep}{0.10ex}
	\centering
	\scriptsize
	\caption{Sensitivity of semi-supervised PI-SC-III to the correction interval.}
	\label{tab:correction-interval-ablation}
	\begin{tabular}{@{}lccc@{}}
		\toprule
		System & $L_g$ & Time (s) & RMSE \\
		\midrule
		Van der Pol oscillator
		& 1  & $1.44\pm0.04$ & $6.28e{-4}\pm1.94e{-4}$ \\
		& 5  & $0.95\pm0.04$ & $6.15e{-4}\pm2.37e{-4}$ \\
		& 10 & $0.86\pm0.02$ & $5.86e{-4}\pm1.79e{-4}$ \\
		& 20 & $0.81\pm0.02$ & $5.07e{-4}\pm1.37e{-4}$ \\
		& 30 & $0.80\pm0.02$ & $4.53e{-4}\pm1.28e{-4}$ \\
		& 40 & $0.79\pm0.02$ & $4.61e{-4}\pm1.53e{-4}$ \\
		& 50 & $0.77\pm0.02$ & $4.18e{-4}\pm1.20e{-4}$ \\
		\midrule
		Helmholtz equation
		& 1  & $2.67\pm0.08$ & $1.19e{-6}\pm1.17e{-6}$ \\
		& 5  & $1.52\pm0.05$ & $1.51e{-6}\pm4.22e{-6}$ \\
		& 10 & $1.35\pm0.03$ & $1.16e{-6}\pm5.59e{-7}$ \\
		& 20 & $1.28\pm0.06$ & $1.68e{-6}\pm1.28e{-6}$ \\
		& 30 & $1.26\pm0.06$ & $2.10e{-6}\pm1.34e{-6}$ \\
		& 40 & $1.22\pm0.05$ & $2.59e{-6}\pm3.42e{-6}$ \\
		& 50 & $1.22\pm0.06$ & $4.17e{-6}\pm5.36e{-6}$ \\
		\midrule
		Allen--Cahn equation
		& 1  & $3.58\pm0.08$ & $1.65e{-3}\pm3.10e{-4}$ \\
		& 5  & $2.23\pm0.04$ & $1.60e{-3}\pm3.10e{-4}$ \\
		& 10 & $2.06\pm0.08$ & $1.58e{-3}\pm2.80e{-4}$ \\
		& 20 & $1.93\pm0.07$ & $1.72e{-3}\pm3.33e{-4}$ \\
		& 30 & $1.92\pm0.08$ & $2.08e{-3}\pm4.01e{-4}$ \\
		& 40 & $1.91\pm0.08$ & $2.40e{-3}\pm5.64e{-4}$ \\
		& 50 & $1.91\pm0.07$ & $2.54e{-3}\pm6.81e{-4}$ \\
		\midrule
		Reaction--diffusion equation
		& 1  & $17.5\pm0.15$ & $6.13e{-3}\pm4.42e{-4}$ \\
		& 5  & $7.50\pm0.07$  & $6.31e{-3}\pm5.15e{-4}$ \\
		& 10 & $6.33\pm0.12$  & $6.38e{-3}\pm4.18e{-4}$ \\
		& 20 & $5.89\pm0.12$  & $6.56e{-3}\pm5.47e{-4}$ \\
		& 30 & $5.65\pm0.10$  & $6.64e{-3}\pm6.02e{-4}$ \\
		& 40 & $5.57\pm0.09$  & $6.56e{-3}\pm5.02e{-4}$ \\
		& 50 & $5.60\pm0.12$  & $6.54e{-3}\pm5.30e{-4}$ \\
		\midrule
		Bratu equation
		& 1  & $2.98\pm0.10$ & $6.67e{-6}\pm1.54e{-7}$ \\
		& 5  & $1.67\pm0.09$ & $6.72e{-6}\pm1.68e{-7}$ \\
		& 10 & $1.50\pm0.08$ & $6.77e{-6}\pm1.77e{-7}$ \\
		& 20 & $1.47\pm0.06$ & $6.70e{-6}\pm1.77e{-7}$ \\
		& 30 & $1.42\pm0.07$ & $6.69e{-6}\pm1.95e{-7}$ \\
		& 40 & $1.42\pm0.07$ & $6.98e{-6}\pm1.72e{-7}$ \\
		& 50 & $1.40\pm0.07$ & $6.82e{-6}\pm2.01e{-7}$ \\
		\midrule
		{Kovasznay flow}
		& {1} & {$17.8\pm1.49$} & {$5.58e{-13}\pm8.07e{-13}$} \\
		& {5} & {$4.98\pm0.94$} & {$6.22e{-13}\pm1.87e{-12}$} \\
		& {10} & {$3.38\pm0.12$} & {$4.31e{-13}\pm4.37e{-13}$} \\
		& {20} & {$2.57\pm0.06$} & {$3.52e{-13}\pm5.59e{-13}$} \\
		& {30} & {$2.29\pm0.07$} & {$6.48e{-13}\pm1.56e{-12}$} \\
		& {40} & {$2.35\pm0.96$} & {$3.09e{-13}\pm5.08e{-13}$} \\
		& {50} & {$2.03\pm0.67$} & {$2.02e{-13}\pm2.06e{-13}$} \\
		\midrule
		{Poiseuille flow}
		& {1} & {$4.72\pm0.83$} & {$2.19e{-14}\pm9.29e{-15}$} \\
		& {5} & {$2.23\pm0.04$} & {$3.18e{-14}\pm1.32e{-14}$} \\
		& {10} & {$1.84\pm0.04$} & {$2.88e{-14}\pm1.09e{-14}$} \\
		& {20} & {$1.63\pm0.03$} & {$2.75e{-14}\pm1.12e{-14}$} \\
		& {30} & {$1.53\pm0.04$} & {$1.44e{-10}\pm6.40e{-10}$} \\
		& {40} & {$1.53\pm0.04$} & {$8.41e{-9}\pm2.83e{-8}$} \\
		& {50} & {$1.48\pm0.04$} & {$8.81e{-7}\pm1.50e{-6}$} \\
		\midrule
		{Nonlinear diffusion}
		& {1} & {$6.52\pm0.68$} & {$5.42e{-7}\pm2.61e{-7}$} \\
		& {5} & {$2.82\pm0.06$} & {$5.48e{-7}\pm2.29e{-7}$} \\
		& {10} & {$2.55\pm1.27$} & {$5.41e{-7}\pm2.51e{-7}$} \\
		& {20} & {$2.03\pm0.04$} & {$5.35e{-7}\pm2.39e{-7}$} \\
		& {30} & {$1.91\pm0.04$} & {$5.13e{-7}\pm2.33e{-7}$} \\
		& {40} & {$1.84\pm0.04$} & {$6.14e{-7}\pm3.06e{-7}$} \\
		& {50} & {$1.91\pm0.83$} & {$6.30e{-7}\pm2.64e{-7}$} \\
		\bottomrule
	\end{tabular}
\end{table}

\begingroup
\noindent\textbf{Local linearization and discrete line search.} Table~\ref{tab:gauss-newton-ablation} reports the training time and state RMSE of semi-supervised PI-SC-III using Gauss--Newton iterations for local node construction and global correction.

\endgroup
\begin{table}[H]
	\renewcommand{\arraystretch}{0.78}
	\setlength{\tabcolsep}{3pt}
	\setlength{\aboverulesep}{0.10ex}
	\setlength{\belowrulesep}{0.10ex}
	\centering
	\scriptsize
	\caption{Semi-supervised PI-SC-III results with Gauss--Newton method.}
	\label{tab:gauss-newton-ablation}
	\color{black}
	\begin{tabular}{@{}lcc@{}}
		\toprule
		{System} & {Time (s)} & {RMSE} \\
		\midrule
		{Van der Pol oscillator} & {$16.5\pm1.96$} & {$6.48e{-4}\pm2.10e{-4}$} \\
		{Allen--Cahn equation} & {$47.9\pm0.36$} & {$1.44e{-3}\pm2.13e{-4}$} \\
		{Reaction--diffusion equation} & {$352\pm5.51$} & {$5.10e{-3}\pm4.74e{-4}$} \\
		{Bratu equation} & {$18.0\pm0.13$} & {$4.60e{-6}\pm2.74e{-8}$} \\
		{Kovasznay flow} & {$464\pm107$} & {$1.75e{-12}\pm1.68e{-12}$} \\
		{Poiseuille flow} & {$56.6\pm14.1$} & {$5.47e{-13}\pm1.37e{-13}$} \\
		{Nonlinear diffusion} & {$80.7\pm3.07$} & {$5.41e{-7}\pm1.84e{-7}$} \\
		\bottomrule
	\end{tabular}
\end{table}

\subsubsection{Result Analysis}\label{app:ablation-analysis}

For the ablation study on the number of interior collocation points, Van der Pol, reaction--diffusion, Bratu, {Kovasznay, and Poiseuille} improve as $N_p$ increases. Van der Pol shows a reduction in mean RMSE from $6.05\times10^{-4}$ to $4.87\times10^{-4}$. On Bratu, the mean RMSE decreases from $9.25\times10^{-6}$ to $6.61\times10^{-6}$. {On Kovasznay, the mean RMSE decreases from $6.20\times10^{-12}$ to $1.75\times10^{-13}$.} {Poiseuille shows an RMSE reduction from $2.32\times10^{-14}$ to $1.99\times10^{-14}$.} Helmholtz remains nearly unchanged. Allen--Cahn attains its lowest mean RMSE at $N_p=1000$. {Nonlinear diffusion attains its lowest mean RMSE at $N_p=750$.} These trends show that the accuracy benefit of additional collocation points depends on the system, with diminishing gains on several benchmarks.

For the ablation study on the correction interval, increasing $L_g$ from 1 to 50 reduces training time by {factors of approximately $1.9$ to $8.8$ across the eight systems}. Larger values of $L_g$ reduce the frequency of global correction and generally lead to higher prediction errors. An appropriately chosen $L_g$ can substantially accelerate network training while maintaining predictive accuracy comparable to that obtained with $L_g=1$.

\begingroup
For the ablation study on local linearization and discrete line search, the Gauss--Newton variant in Table~\ref{tab:gauss-newton-ablation} requires approximately $6$ to $26$ times as much training time as PI-SC-III in Table~\ref{tab:correction-interval-ablation}, with limited accuracy gains. Repeated nonlinear least-squares iterations substantially increase the cost of local node construction and global correction, making the fast training advantage of SCNs harder to realize. These results support the proposed local linearization and discrete line search as an efficient approach to maintaining high predictive accuracy through analytical least-squares updates.
\endgroup

\end{document}